\documentclass[review,a4page]{elsarticle}
\usepackage[utf8]{inputenc}
\usepackage{graphicx}
\usepackage{srcltx}
\usepackage{eurosym}
\usepackage{mathtools}
\allowdisplaybreaks
\usepackage{enumerate}
\usepackage{amsmath}
\usepackage{amsfonts}
\usepackage{amssymb}
\usepackage{amsthm}
\usepackage{graphicx}
\usepackage{mathrsfs}
\usepackage{xcolor}
\usepackage{exscale}
\usepackage{latexsym}
\usepackage{esvect}
\usepackage[T1]{fontenc}

\usepackage{calc}
\usepackage[titletoc,toc,page]{appendix}
\usepackage[colorlinks,plainpages=true,pdfpagelabels,hypertexnames=true,colorlinks=true,pdfstartview=FitV,linkcolor=blue,citecolor=red,urlcolor=black]{hyperref}
\PassOptionsToPackage{unicode}{hyperref}
\PassOptionsToPackage{naturalnames}{hyperref}
\AtBeginDocument{%
\hypersetup{citecolor=red}}
\usepackage{enumerate}
\usepackage[shortlabels]{enumitem}
\usepackage{bookmark}
\usepackage{wasysym}
\usepackage{esint}
\usepackage{setspace}
\usepackage[ddmmyyyy]{datetime}
\numberwithin{equation}{section}
\everymath{\displaystyle}
\usepackage[margin=2.2cm]{geometry}

\usepackage[capitalize,nameinlink]{cleveref}
\crefname{section}{Section}{Sections}
\crefname{subsection}{Subsection}{Subsections}
\crefname{condition}{Condition}{Conditions}
\crefname{hypothesis}{Hypothesis}{Hypothesis}
\crefname{assumption}{Assumption}{Assumptions}
\crefname{lemma}{Lemma}{Lemmas}
\crefname{claim}{Claim}{Claims}
\crefname{remark}{Remark}{Remarks}

\crefformat{equation}{\textup{#2(#1)#3}}
\crefrangeformat{equation}{\textup{#3(#1)#4--#5(#2)#6}}
\crefmultiformat{equation}{\textup{#2(#1)#3}}{ and \textup{#2(#1)#3}}
{, \textup{#2(#1)#3}}{, and \textup{#2(#1)#3}}
\crefrangemultiformat{equation}{\textup{#3(#1)#4--#5(#2)#6}}%
{ and \textup{#3(#1)#4--#5(#2)#6}}{, \textup{#3(#1)#4--#5(#2)#6}}%
{, and \textup{#3(#1)#4--#5(#2)#6}}

\Crefformat{equation}{#2Equation~\textup{(#1)}#3}
\Crefrangeformat{equation}{Equations~\textup{#3(#1)#4--#5(#2)#6}}
\Crefmultiformat{equation}{Equations~\textup{#2(#1)#3}}{ and \textup{#2(#1)#3}}
{, \textup{#2(#1)#3}}{, and \textup{#2(#1)#3}}
\Crefrangemultiformat{equation}{Equations~\textup{#3(#1)#4--#5(#2)#6}}%
{ and \textup{#3(#1)#4--#5(#2)#6}}{, \textup{#3(#1)#4--#5(#2)#6}}%
{, and \textup{#3(#1)#4--#5(#2)#6}}

\crefdefaultlabelformat{#2\textup{#1}#3}
\newtheorem{theorem}{Theorem}[section]
\newtheorem{lemma}[theorem]{Lemma}

\newtheorem{corollary}[theorem]{Corollary}

\newtheorem{definition}[theorem]{Definition}
\newtheorem{remark}[theorem]{Remark}        

\numberwithin{equation}{section}

\def\YYint#1#2#3{{\setbox0=\hbox{$#1{#2#3}{\iint}$}
\vcenter{\hbox{$#2#3$}}\kern-.50\wd0}}

\def\XXint#1#2#3{{\setbox0=\hbox{$#1{#2#3}{\int}$}
\vcenter{\hbox{$#2#3$}}\kern-.50\wd0}}

\makeatletter
\def\namedlabel#1#2{\begingroup
\def\@currentlabel{#2}%
\label{#1}\endgroup
}
\makeatother
\makeatletter
\newcommand{\rmh}[1]{\mathpalette{\raisem@th{#1}}}
\newcommand{\raisem@th}[3]{\hspace*{-1pt}\raisebox{#1}{$#2#3$}}
\makeatother

\newcommand{\descref}[2]{\hyperref[#1]{\textcolor{black}{(}\textcolor{blue}{\bf #2}\textcolor{black}{)}}}

\newcommand{\dref}[2]{\hyperref[#1]{\textcolor{black}{(}\textcolor{blue}{\bf #2}\textcolor{black}{)}}}

\makeatletter
\g@addto@macro\normalsize{%
\setlength\abovedisplayskip{3pt}
\setlength\belowdisplayskip{3pt}
\setlength\abovedisplayshortskip{1pt}
\setlength\belowdisplayshortskip{3pt}
}
\makeatletter
\def\ps@pprintTitle{%
\let\@oddhead\@empty
\let\@evenhead\@empty
\def\@oddfoot{}%
\let\@evenfoot\@oddfoot}
\makeatother

\newcommand{\ep}{\epsilon}

\newcommand{\iprod}[2]{\langle #1 \ ,  #2\rangle}

\newcounter{whitney}
\refstepcounter{whitney}

\newcounter{ineqcounter}
\refstepcounter{ineqcounter}

\begin{document}
\begin{frontmatter}
\title{Regularity results for a class of mixed local and nonlocal singular problems involving distance function}
\author{Kaushik Bal and Stuti Das}
\ead{kaushik@iitk.ac.in and stutid21@iitk.ac.in/stutidas1999@gmail.com}
\address{Department of Mathematics and Statistics,\\ Indian Institute of Technology Kanpur, Uttar Pradesh, 208016, India}

\begin{abstract}
We investigate the following mixed local and nonlocal quasilinear equation with singularity given by
\begin{eqnarray*}
\begin{split}
-\Delta_pu+(-\Delta)_q^s u&=\frac{f(x)}{u^{\delta}}\text { in } \Omega, \\u&>0 \text{ in } \Omega,\\u&=0  \text { in }\mathbb{R}^n \backslash \Omega;
\end{split}
\end{eqnarray*}
where, 
\begin{equation*}
(-\Delta )_q^s u(x):= c_{n,s}\operatorname{P.V.}\int_{\mathbb{R}^n}\frac{|u(x)-u(y)|^{q-2}(u(x)-u(y))}{|x-y|^{n+sq}} d y,
\end{equation*}
with $\Omega$ being a bounded domain in $\mathbb{R}^{n}$ with $C^2$ boundary, $1<q\leq p<\infty$, $s\in(0,1)$, $\delta>0$ and $f\in L^\infty_{\mathrm{loc}}(\Omega)$ is a non-negative function which behaves like $\mathbf{dist(x,\partial \Omega)^{-\beta}}$, $\beta\geq 0$ near $\partial \Omega$. We start by proving several H\"older and gradient H\"older regularity results for a more general class of quasilinear operators when $\delta=0$. Using the regularity results we deduce existence, uniqueness and H\"older regularity of a weak solution of the singular problem in $W_{\mathrm{loc}}^{1,p}(\Omega)$ and its behavior near $\partial \Omega$ albeit with different exponents depending on $\beta+\delta$. Boundedness and H\"older regularity result to the singular equation with critical exponent were also discussed.
\end{abstract}
\begin{keyword}Mixed local-nonlocal equation; H\"older regularity; Singular nonlinearity; Distance function; Existence and uniqueness
\MSC [2020] 35J75, 35M10, 35R11.
\end{keyword}


\end{frontmatter}

\begin{singlespace}
\tableofcontents
\end{singlespace}
\section{Introduction}{\label{intro}} 
The goal of this article is to study a particular class of mixed local-nonlocal problems where mixed refer to a combination of two operators with different orders of derivative, the simplest model being $-\Delta+(-\Delta)^s $, for $s\in(0,1)$. More precisely, we aim to we study H\"older regularity results related to the quasilinear operator $-\Delta _p+(-\Delta)^s_q$, $1<q\leq p<\infty$, $s\in(0,1)$ and the existence, uniqueness and regularity of the weak solutions to the following prototype singular problem:
\begin{equation}{\label{pppp}}
    \begin{aligned}
-\Delta_p u+(-\Delta)^s_q u & =f(x) u^{-\delta}, \quad u>0 \text { in } \Omega, \\
u & =0 \text { on } \mathbb{R}^n\backslash \Omega;
\end{aligned}
\end{equation}
where $\Omega$ is a bounded domain in $\mathbb{R}^n$ with $C^2$ boundary $\partial \Omega$ and $\delta>0$ with $f$ behaving like distance function near the boundary.
The motivation behind the study is the apparent lack of homogeneity of the operator in question which makes replicating the results available in the literature to the mixed local-nonlocal cases difficult to reproduce and in some cases false even when $p=q=2$. 


We start our discussion by considering the following general form of the operator involving \cref{pppp}, given by
\begin{equation}{\label{gnn}}
\mathcal{L} u:=\mathcal{L}_L u+\mathcal{L}_N u=g\quad \text{ in }\Omega,
\end{equation}
where 
\begin{equation*}
\mathcal{L}_L u(x)=\mathcal{L}_L^A u(x):=-\operatorname{div} A(x, \nabla u(x))    
\end{equation*}
and 
\begin{equation*}
    \mathcal{L}_N u(x)=\mathcal{L}_N^{\Phi, B, s, q} u(x):= \text { P.V. } \int_{\mathbb{R}^n} \Phi(u(x)-u(y)) \frac{B(x, y)}{|x-y|^{n+s q}} d y .
\end{equation*}
Also $g\in L^m_{\mathrm{loc}}(\Omega)$, for $m\geq n$. $A: \Omega \times \mathbb{R}^n \rightarrow \mathbb{R}^n$ is a continuous vector field such that $A(x, \cdot) \in C^1(\mathbb{R}^n \backslash\{0\} ; \mathbb{R}^n)$ for all $x \in \Omega, A(\cdot, z) \in C^{\bar{\alpha}}(\Omega ; \mathbb{R}^n)$ for all $z \in \mathbb{R}^n$, and satisfies the following $p$-growth and coercivity conditions
\begin{equation*}
\begin{cases}|A(x, z)|+|z|\left|\partial_z A(x, z)\right| \leq \Lambda\left(|z|^2+\mu^2\right)^{\frac{p-2}{2}}|z| & \text { for } x \in \Omega, z \in \mathbb{R}^n \backslash\{0\}, \\ |A(x, z)-A(y, z)| \leq \Lambda\left(|z|^2+\mu^2\right)^{\frac{p-1}{2}}|x-y|^{\bar{\alpha}} & \text { for } x, y \in \Omega, z \in \mathbb{R}^n, \\ \left\langle\partial_z A(x, z) \xi, \xi\right\rangle \geq \Lambda^{-1}\left(|z|^2+\mu^2\right)^{\frac{p-2}{2}}|\xi|^2 & \text { for } x \in \Omega, z \in \mathbb{R}^n \backslash\{0\}, \xi \in \mathbb{R}^n,\end{cases}
\end{equation*}
for some constants $\bar{\alpha} \in(0,1), \mu \in[0,1]$, $\Lambda \geq 1$, while $B: \mathbb{R}^n \times \mathbb{R}^n \rightarrow[0,+\infty)$ is a measurable function satisfying
\begin{equation*}
B(x, y)=B(y, x) \quad \text { and } \quad \Lambda^{-1} \leq B(x, y) \leq \Lambda \quad \text { for a.e. } x, y \in \mathbb{R}^n
\end{equation*}
and $\Phi \in C^0(\mathbb{R})$ is an odd, non-decreasing function fulfilling the $q$-growth and coercivity condition
\begin{equation*}
    \Lambda^{-1}|t|^q \leq \phi(t) t \leq \Lambda|t|^q \quad \text { for all } t \in \mathbb{R} \text {. }
\end{equation*}

In this context, we first mention the work in \cite{Biagi1}, where the authors have considered the particular operator $-\Delta+(-\Delta)^s$ (i.e. $p=q=2, A(x,z)=z, B(x,y)=1,\Phi(t)=t$) under homogeneous Dirichlet boundary condition, and obtained several results regarding existence, uniqueness, maximum principles, higher summability of solutions depending on the summability of the source term $g$. Further, in this particular case, the same authors, in their paper \cite{Biagisymmetry}, have obtained symmetry results for sufficiently regular solutions. The Brezis-Nirenberg problem in this context was discussed in \cite{biagi2022brezis}. For the quasilinear case $-\Delta_p+(-\Delta)^s_p$ with Dirichlet boundary data, we refer \cite{necessarybrezis,eigenvalue}, where necessary and suﬃcient conditions for the
existence of a unique positive weak solution for some sublinear problems have been obtained. The authors, in \cite{eigenvalue}, have also studied the eigenvalue problem for the mixed operator. For more insights into the mixed operator, one may see [\citealp{FaberKrahn,HongKrahnSzego}] and the references therein.

Concerning the regularity of weak solutions, we would like to mention \cite{MG,garainholder,garain2023higher}, where various local H\"older regularity of solutions (defined in a suitable sense) and their gradients have been obtained. In this regard, the novelty of our work is to consider the solutions in $W^{1,p}_{\mathrm{loc}}(\Omega)\cap L^{q-1}_{sq}(\mathbb{R}^n)$ along with $g\neq 0$ and a combination of non-homogeneous $(p,q)$ operator. We first show any solution $u\in W^{1,p}_{\mathrm{loc}}(\Omega)\cap L^{q-1}_{sq}(\mathbb{R}^n)$ to the equation $-\Delta_p u+(-\Delta)^s_qu =g$ is indeed locally almost Lipschitz continuous provided $g\in L^\infty_{\mathrm{loc}}(\Omega)$ and $2\leq q\leq p<\infty$. Next, for the same range of $p$ and $q$, we improve the almost Lipschitz regularity result for the general form of the equation given by \cref{gnn} and $g\in L^n_{\mathrm{loc}}(\Omega)$ along with a local gradient H\"older regularity. 

For regularity results up to the boundary, we refer \cite{antonini,valdinoci,valdinoci2}. Following the same perturbation approach used in \cite{antonini}, we prove that if a non-negative solution $u$ to \cref{gnn} behaves like distance function near the boundary, then it is in $C^{1, \alpha}(\overline{\Omega})$, for some $\alpha \in(0,1)$, when $0 \leq g \leq C d(x)^{-\sigma}$ for $\sigma<1$ and $d(x):=\operatorname{dist}(x, \partial \Omega)$.

As far as the singular elliptic problems are concerned, we would like to start by mentioning the now classical work of Crandall-Rabinowitz-Tartar \cite{CrRaTa}, who showed that the problem given by
\begin{eqnarray*}{\label{pb}}
\begin{split}
    -\Delta u&={f}u^{-\delta} \text { in } \Omega,\quad 
u>0\;  \text { in } \Omega, \\
u&=0\;  \text { in } \partial \Omega;
\end{split}
\end{eqnarray*}
admits a unique solution $u\in C^2(\Omega)\cap C(\overline{\Omega})$ for any $\delta>0$ along with the fact that the solution must behave like a distance function near the boundary provided $f$ is H\"older Continuous. There has been an extensive study in this direction since this pioneering work, including both local and nonlocal cases, see for instance \cite{fracsingular,orsina,scase,Canino,Giacomoni,Haitao,LaMc} and the references therein.\\
Turning to the mixed local-nonlocal singular problems, we refer to the works of [\citealp{arora,stutimultiplicity,biagisingular,garaingeometric,garainforummath,garain}], where several results regarding existence, summability, multiplicity have been obtained. 
Inspired by the above discussion, in this work, we consider singular problems given by \cref{pppp}, assuming that $f \in L_{\mathrm{loc}}^{\infty}(\Omega)$ satisfies the condition
$
    c_1 d(x)^{-\beta} \leq f(x) \leq c_2 d(x)^{-\beta} \text { in } \Omega_{\rho},
$
where $c_1, c_2$ are nonnegative constants, $\beta \geq 0$ and $\Omega_{\rho}:=\{x \in \overline{\Omega}: d(x)<\rho\}$ for $\rho>0$. We obtain existence, uniqueness, boundary behavior, optimal Sobolev regularity and most importantly, up to the boundary H\"older regularity of the problem \cref{pppp}. To prove the existence of a weak solution, we use the perturbation argument used in \cite{orsina}. We prove suitable uniform boundedness of the sequence of approximated solutions in $W^{1,p}_0(\Omega)$ (or in $W^{1,p}_{\mathrm{loc}}(\Omega))$ and pass through the limit. While proving the boundary behavior of solutions, due to the nonhomogeneous nature of the leading operator, sub-super solution methods as previously used, cannot be implemented in our case. 
We overcome this difficulty by exploiting $C^2$ regularity of $\partial \Omega$ which guarantees that the distance function is $C^2$ in some neighborhood of the boundary in combination with the barrier functions as done in \cite{aroranodea} and using the property that their fractional Laplacian gives a $L^\infty$ function only in a neighborhood of $\partial \Omega$ as done in \cite{giacomonipq}. \\
The behavior, as mentioned above, of the solution near the boundary allows us to establish the optimal Sobolev regularity for the weak solution to \cref{pppp} and a non-existence result for the case $\beta>p$. Moreover, we prove a comparison principle for sub and super solution of \cref{pppp} in $W_{\mathrm{loc}}^{1, p}(\Omega)$ for the case of $\beta<2-1/p$.\\
Summarizing, we observe that the leading operator in problem \cref{pppp} can be non-homogeneous in nature, that is, there does not exist any $\alpha>0$ such that $K(t \xi)=t^\alpha K(\xi)$ for all $\xi \in \mathbb{R}$ and $t>0$, where $K(\xi):=-\Delta_p \xi+(-\Delta)^s_q \xi$ and the nonlinear term is doubly singular in the sense that it involves two singular terms, $u^{-\delta}$ and the weight function which blows up near the boundary. Thus, the novelty of the paper lies in the following:\\
(i) We consider a class of nonhomogeneous mixed local-nonlocal operators with different orders of exponents and establish almost optimal global Hölder continuity results for problems involving singular nonlinearities.\\
(ii) We prove behavior of the solution to \cref{pppp} near the boundary where the standard scaling argument fails.\\
(iii) The weak comparison principle where the source term is doubly nonlinear. Here, the sub and super solutions of the equation may not be in $W_0^{1,p}(\Omega)$. A non-existence result for the case $\beta>p$ has also been obtained.
\\
(iv) Several H\"older and gradient H\"older regularity results, which are of independent interest, have been obtained in a completely new set-up where we allow solutions to be in $W^{1,p}_{\mathrm{loc}}(\Omega)\cap L^{q-1}_{sq}(\mathbb{R}^n)$. We obtain a completely new Hölder continuity for the gradient of weak solution of some quasilinear equation involving lower order terms with nonhomogeneous nature and singular nonlinearity. This gives us $C^{1,\alpha}(\overline{\Omega})$ regularity for the case $\beta+\delta<1$.\\
(v) We prove almost Lipschitz regularity up to the boundary for the case $\beta+\delta=1$. Our method does not require solutions to lie in the energy space to prove Hölder continuity results when $\beta+\delta \geq 1$. Also, a local H\"older regularity for the gradient is obtained for $\beta+\delta\geq 1$. \\ 
(vi) We have obtained an a priori $C^{1,\alpha}(\overline{\Omega})$ regularity for solutions to \cref{pppp} under a perturbation $b(x,u)$, when $\delta<1,\beta=0$ and $0\leq b(x,u)\leq c(1+|u|^{p^*-1})$.\smallskip\\
We finally remark that, after a first draft of this article was posted on Arxiv, a related preprint \cite{dhanya2024interior} appeared,
where similar results was obtained for $qs\leq p$. It is worth mentioning that our regularity results involves a weaker notion of solutions and gives better H\"older regularity in relation to the singular problems along with better boundary behaviors.\\\smallskip
\textbf{Layout of the paper:}  In \cref{prelims}, we will write notations and give definitions and embedding results regarding Sobolev spaces as required. Appropriate notions of weak solutions for our problems will be defined, and the main results will be stated. \cref{regu1} contains proof of almost Lipschitz regularity (locally) for $-\Delta_pu+(-\Delta)_q^su=g$ in $\Omega$ where the solutions belong to $W^{1,p}_{\mathrm{loc}}(\Omega)\cap L^{q-1}_{sq}(\mathbb{R}^n)$ and $g\in L^\infty_{\mathrm{loc}}(\Omega)$; and \cref{regu2} extends this for the equation \cref{gnn} along with $g\in L^n_{\mathrm{loc}}(\Omega)$. Higher H\"older regularity for \cref{gnn} is obtained in \cref{regu3} whereas \cref{regu4} gives up to the boundary gradient H\"older regularity when the source term can be nonlinear in the sense that it blows up near the boundary. Finally, \cref{singu} contains the singular problem \cref{pppp} and all the proofs regarding this. We end with \cref{appli}, which contains an application to a perturbed problem by obtaining the $C^{1,\alpha}(\overline{\Omega})$ regularity.
\section{Preliminaries}{\label{prelims}}
\subsection{\textbf{Notations}} 
$\bullet$ We will take $n$ to be the space dimension and $\Omega$ to be a bounded domain in $\mathbb{R}^n$ with $C^2$ boundary.\smallskip\\$\bullet$ The Lebesgue measure of a measurable subset $\mathrm{S}\subset \mathbb{R}^n$ will be denoted by $|\mathrm{S}|$.\smallskip\\
$\bullet$ For $r>1$, the H\"older conjugate exponent of $r$ will be denoted by $r^\prime=\frac{r}{r-1}$.\smallskip\\
$\bullet$ For any open subset $\Omega$ of $\mathbb{R}^n$, $K\subset\subset \Omega $ will imply $K$ is compactly contained in $\Omega.$\smallskip\\
$\bullet$ $\int$ will denote integration concerning space only, and integration on $\Omega \times \Omega$ or $\mathbb{R}^n \times \mathbb{R}^n$ will be denoted by a double integral $\iint$. Moreover, the average integral will be denoted by $\fint$.\smallskip\\
$\bullet$ The notation $a \lesssim b$ will be used for $a \leq C b$, where $C$ is a universal constant that only depends on the prescribed data. 
$C$ (or $c$) may vary from line to line or even in the same line.\smallskip\\
$\bullet$ For a function $h$, we denote its positive and negative parts by $h^+=\max\{h,0\}$, $h^-=\max\{-h,0\}$ respectively.\smallskip\\
$\bullet$ For $k\in \mathbb{N}$, we denote $T_k(\sigma)=\max \{-k, \min \{k, \sigma\}\}$, for $\sigma \in \mathbb{R}$.
\subsection{\textbf{Function Spaces}}
We recall that for $E \subset \mathbb{R}^n$, the Lebesgue space
$L^p(E), 1 \leq p<\infty$, is defined to be the space of $p$-integrable functions $u: E \rightarrow \mathbb{R}$ with the finite norm
\begin{equation*}
\|u\|_{L^p(E)}=\left(\int_E|u(x)|^p d x\right)^{1 / p} .
\end{equation*}
By $L_{\mathrm{loc }}^p(E)$ we denote the space of locally $p$-integrable functions, i.e. $u \in L_{\operatorname{loc }}^p(E)$ if and only if $u \in L^p(F)$ for every $F \subset\subset E$. In the case $0<p<1$, we denote by $L^p(E)$ a set of measurable functions such that $\int_E|u(x)|^p d x<\infty$.
\begin{definition}
  Let $\Omega\subset\mathbb{R}^n$ be a bounded open set with $
C^1$ boundary. The Sobolev space $W^{1, p}(\Omega)$, for $1\leq p<\infty$ is defined as the Banach space of all integrable and weakly differentiable functions $u: \Omega \to\mathbb{R}$ equipped with the following norm
\begin{equation*}
\|u\|_{W^{1, p}(\Omega)}=\|u\|_{L^p(\Omega)}+\|\nabla u\|_{L^p(\Omega)} .
\end{equation*}
\end{definition}
The space $W_0^{1, p}(\Omega)$ is defined as the closure of the space ${C}_c^{\infty}(\Omega)$, in the norm of the Sobolev space $W^{1, p}(\Omega)$, where ${C}^\infty_c(\Omega)$ is the set of all smooth functions whose supports are compactly contained in $\Omega$. As $\partial\Omega$ is $C^1$, by [\citealp{brezis2011functional}, Proposition 9.18], $W_0^{1, p}(\Omega)$ can be identified by $\mathbb{X}_p(\Omega)=\{u\in W^{1,p}(\mathbb{R}^n): u=0 \text{ a.e. in } \mathbb{R}^n\backslash\Omega\}$.
\begin{definition}
    Let $0<s<1$ and $\Omega$ be an open connected subset of $\mathbb{R}^n$ with $C^1$ boundary. The fractional Sobolev space $W^{s, q}(\Omega)$ for any $1\leq q<\infty$ is defined as
\begin{equation*}
    W^{s, q}(\Omega)=\left\{u \in L^q(\Omega): \frac{|u(x)-u(y)|}{|x-y|^{\frac{n}{q}+s}} \in L^q(\Omega\times\Omega)\right\},
\end{equation*}
and it is endowed with the norm
\begin{equation}{\label{norm}}
\|u\|_{W^{s, q}(\Omega)}=\left(\int_{\Omega}|u(x)|^q d x+\int_{\Omega} \int_{\Omega} \frac{|u(x)-u(y)|^q}{|x-y|^{n+sq}}d x d y\right)^{1/q}.
\end{equation}
\end{definition}
It can be treated as an intermediate space between $W^{1,q}(\Omega)$ and $L^q(\Omega)$. For $0<s_1\leq s_2<1$, $W^{s_2,q}(\Omega)$ is continuously embedded in $W^{s_1,q}(\Omega)$, see [\citealp{frac}, Proposition 2.1]. The fractional Sobolev space with zero boundary values is defined by
\begin{equation*}
W_0^{s, q}(\Omega)=\left\{u \in W^{s, q}(\mathbb{R}^n): u=0 \text { in } \mathbb{R}^n \backslash \Omega\right\}.
\end{equation*}
However if $sq\neq 1$, $W_0^{s, q}(\Omega)$ can be treated as the closure of $ {C}^\infty_c(\Omega)$ in $W^{s,q}(\Omega)$ with respect to the fractional Sobolev norm defined in \cref{norm}. Both $W^{s, q}(\Omega)$ and $W_0^{s, q}(\Omega)$ are reflexive Banach spaces, for $q>1$, for details we refer to the readers [\citealp{frac}, Section 2]. The spaces $W^{1, p}(\Omega)$ and $W_0^{1, p}(\Omega)$ are also reflexive for $p>1$.\smallskip\\
The following result asserts that the classical Sobolev space is continuously embedded in the fractional Sobolev space; see [\citealp{frac}, Proposition 2.2]. The idea applies an extension property of $\Omega$ so that we can extend functions from $W^{1,q}(\Omega)$ to $W^{1,q}(\mathbb{R}^n)$ and that the extension operator is bounded.
\begin{lemma}{\label{embedding}}
    Let $\Omega$ be a bounded domain in $\mathbb{R}^n$ with $C^{1}$ boundary. Then there exists $ c\equiv c( n, s,\Omega)>0$ such that
\begin{equation*}
\|u\|_{W^{s, q}(\Omega)} \leq c\|u\|_{W^{1,q}(\Omega)},
\end{equation*}
for every $u \in W^{1,q}(\Omega)$.
\end{lemma}
The next embedding result follows from [\citealp{frac2}, Lemma 2.1] for Sobolev spaces with zero boundary value. The fundamental difference of it compared to \cref{embedding} is that the result holds for any bounded domain (without any smoothness condition on the boundary), since for the Sobolev spaces with zero boundary value, we always have a zero extension to the complement.
\begin{lemma}{\label{embedding2}} Let $\Omega$ be a bounded domain in $\mathbb{R}^n$ and $0<s<1$. Then there exists $C=C(n, s, \Omega)>0$ such that
\begin{equation*}
\int_{\mathbb{R}^n} \int_{\mathbb{R}^n} \frac{|u(x)-u(y)|^q}{|x-y|^{n+sq}} d x d y \leq C \int_{\Omega}|\nabla u|^qd x,
\end{equation*}
for every $u \in W_0^{1,q}(\Omega)$. Here, we consider the zero extension of $u$ to the complement of $\Omega$.
\end{lemma}
One particular version of \cref{embedding2} can be found in [\citealp{MG}, Lemma 2.2].
\begin{lemma}{\label{fracemb}} Let $1\leq q\leq p<\infty, s\in (0,1), B_{\varrho} \subset \mathbb{R}^n$ be a ball. If $u \in W_0^{1, p}(B_{\varrho})$, then $u \in W^{s, q}(B_{\varrho})$ and
\begin{equation*}
\left(\int_{B_{\varrho}} \fint_{B_{\varrho}} \frac{|u(x)-u(y)|^q}{|x-y|^{n+sq}} d x d y\right)^{1 / q} \leq c \varrho^{1-s}\left(\fint_{B_{\varrho}}|\nabla u|^pd x\right)^{1 / p}
\end{equation*}
holds with $c \equiv c(n, p,q, s)$.
\end{lemma}
The following useful lemma can be found in [\citealp{MG}, Lemma 3.2].
\begin{lemma}{\label{Ming}} Let $s\in(0,1)$, 
$w \in L_{\mathrm{loc}}^q(\mathbb{R}^n)$ and $B_\varrho(x_0) \subset \mathbb{R}^n$ be a ball. Then there exists $c \equiv c(n,q, s)$ such that
\begin{equation*}
\int_{\mathbb{R}^n \backslash B_\varrho} \frac{|w(y)|^{q-1}}{\left|y-x_0\right|^{n+sq}} d y \leq \frac{c}{\varrho^s}\left(\int_{\mathbb{R}^n \backslash B_\varrho} \frac{|w(y)|^q}{\left|y-x_0\right|^{n+sq}}dy\right)^{1-1/q}.
\end{equation*}
\end{lemma}
We now proceed with the essential Poincar\'{e} inequality, which can be found in [\citealp{LCE}, Chapter 5, Section 5.8.1].
\begin{lemma}{\label{poooin}}
  Let $\Omega\subset \mathbb{R}^n$ be a bounded domain with $ {C}^1$ boundary. 
  Then there exists $c\equiv c(n,\Omega)>0$, such that \begin{equation*} 
  \int_\Omega |u|^p d x\leq C\int_\Omega |\nabla u|^p d x, \qquad\forall u\in W^{1,p}_0(\Omega).
  \end{equation*}
  Specifically if we take $\Omega=B_{\bar r}$, then we will get for all $u\in W^{1,p}(B_{\bar r})$,
  \begin{equation*}
  \fint_{B_{\bar r}}\left|u-(u)_{B_{\bar r}}\right|^p d x \leq c {\bar r}^p \fint_{B_{\bar r}} |\nabla u|^pd x,
  \end{equation*}
  where $c\equiv c(n)>0$ 
  and $(u)_{B_{\bar r}}$ denotes the average of $u$ in $B_{\bar r}$, and $B_{\bar r}$ is a ball of radius ${\bar r}$ centered at $x_0\in \mathbb{R}^n$.
 \end{lemma}
Using \cref{embedding2} and \cref{poooin}, we observe that the following norm on the space $W^{1,p}_0(\Omega)$ defined by 
 \begin{equation*}
\|u\|_{W^{1,p}_0(\Omega)}=\left(\int_\Omega |\nabla u|^p d x +\int_{\mathbb{R}^n} \int_{\mathbb{R}^n} \frac{|u(x)-u(y)|^p}{|x-y|^{n+sp}} d x d y \right)^{1/p},
\end{equation*}
is equivalent to the norm
 \begin{equation*}
\|u\|_{W^{1,p}_0(\Omega)}=\left(\int_\Omega |\nabla u|^p d x  \right)^{1/p}.
 \end{equation*}
The following is a version of fractional Poincar\'{e}.
\begin{lemma}{\label{fracpoin}}
Let $\Omega\subset \mathbb{R}^n$ be a bounded domain with $ {C}^1$ boundary, $s \in(0,1)$ and $q\geq 1$. If $u \in W^{s,q}_0(\Omega)$, then
\begin{equation*}
\int_\Omega |u|^q d x \leq c\int_{\Omega} \int_{\Omega} \frac{|u(x)-u(y)|^q}{|x-y|^{n+sq}} d x d y ,
\end{equation*}
holds with $c \equiv c(n, s,\Omega)$.
\end{lemma}
In view of \cref{fracpoin}, we observe that the Banach space $W_0^{s, q}(\Omega)$ can be endowed with the norm
\begin{equation*}
\|u\|_{W_0^{s, q}(\Omega)}=\left(\int_{\Omega} \int_{\Omega} \frac{|u(x)-u(y)|^q}{|x-y|^{n+sq}} d x d y\right)^{1/q},
\end{equation*}
which is equivalent to that of $\|u\|_{W^{s, q}(\Omega)}$. 
Now, we define the local spaces as
\begin{equation*}
     W_{\operatorname{loc }}^{1, p}(\Omega)=\left\{u: \Omega \rightarrow \mathbb{R}: u \in L^p(K), \int_K |\nabla u|^p d x<\infty, \text { for every } K \subset \subset \Omega\right\} ,
\end{equation*}
and 
\begin{equation*}
     W_{\operatorname{loc }}^{s, q}(\Omega)=\left\{u: \Omega \rightarrow \mathbb{R}: u \in L^q(K), \int_K \int_K \frac{|u(x)-u(y)|^q}{|x-y|^{n+ sq}} d x d y<\infty, \text { for every } K \subset \subset \Omega\right\} .
\end{equation*}
Now for $n>p$, we define the critical Sobolev exponent as $p^*=\frac{np}{n-p}$, then we get the following embedding result for any bounded open subset $\Omega$ of class $ {C}^1$ in $\mathbb{R}^n$, see for details [\citealp{LCE}, Chapter 5].
\begin{theorem}{\label{Sobolev embedding}} The embedding operators
\begin{eqnarray*}
    W_0^{1, p}(\Omega) \hookrightarrow \begin{cases}L^t(\Omega), & \text { for } t \in\left[1, p^*\right], \text { if } p \in(1, n), \\ L^t(\Omega), & \text { for } t \in[1, \infty), \text { if } p=n, \\ L^{\infty}(\Omega), & \text { if } p>n,\end{cases}
\end{eqnarray*}
are continuous. Also, the above embeddings are compact, except for $t=p^*$, if $p \in(1, n)$.
\end{theorem}
Similarly, for $n>qs$, we define the fractional Sobolev critical exponent as $q_s^*=\frac{ nq}{n-q s}$. The following result is a fractional version of the Sobolev inequality (\cref{Sobolev embedding}), which also implies a continuous embedding of $W_0^{s,q}(\Omega)$ in the critical Lebesgue space $L^{q_s^*}(\Omega)$. One can see the proof in \cite{frac}.
 \begin{theorem}{\label{Fractional Sobolev embedding}} Let $0<s<1$ be such that $n>qs$. Then, there exists a constant $S(n, s)$ depending only on $n$ and $s$, such that for all $u \in  {C}_c^{\infty}(\Omega
 )$
\begin{equation*}
\|u\|_{L^{q_s^*}(\Omega
)} \leq S(n, s) \left(\int_{\Omega}\int_{\Omega
} \frac{|u(x)-u(y)|^q}{|x-y|^{n+sq}} d x d y \right)^{1/q}.    \end{equation*}
\end{theorem}
The following version of Gagliardo-Nirenberg-Sobolev inequality will be useful for us, see [\citealp{maly1997fine}, Corollary 1.57].
\begin{theorem}{\label{sobolev embedding 2}}
    Let $1<p<\infty, \Omega$ be an open set in $\mathbb{R}^n$ with $|\Omega|<\infty$ and
\begin{equation}{\label{kappa}}
    \kappa= \begin{cases}\frac{n}{n-p}, & \text { if } 1<p<n, \\ 2, & \text { if } p \geq n .\end{cases}
\end{equation}
There exists a positive constant $C=C(n, p)$ such that for every $u \in W_0^{1, p}(\Omega)$, it holds
\begin{equation*}
\left(\int_{\Omega}|u(x)|^{\kappa p} d x\right)^{\frac{1}{\kappa p}} \leq C|\Omega|^{\frac{1}{n}-\frac{1}{p}+\frac{1}{\kappa p}}\left(\int_{\Omega}|\nabla u(x)|^p d x\right)^{\frac{1}{p}}    .
\end{equation*}
\end{theorem} For the following theorem, see \cite{ciarlet}.
\begin{theorem}{\label{2.13}}
     Let $U$ be a nonempty closed convex subset of a real separable reflexive Banach space $V$ and let $A: V \rightarrow V^\prime$ be a coercive demicontinuous monotone operator. Then for every $f \in V^\prime$, there exists $u \in U$ so that
\begin{equation*}
    \langle A(u), v-u\rangle \geq\langle f, v-u\rangle \text { for all } v \in U.
\end{equation*}
Moreover, if $A$ is strictly monotone, then $u$ is unique.
\end{theorem}
See [\citealp{scase}, Lemma 3.5] for the following result.
\begin{lemma}{\label{algebraic2}}
    Let $m>1$ and $\epsilon>0$. For $(x, y) \in \mathbb{R}^2$, let us set
\begin{equation*}
S_\epsilon^x:=\{x \geq \epsilon\} \cap\{y \geq 0\} \text { and } S_\epsilon^y=\{y \geq \ep\} \cap\{x \geq 0\} .
\end{equation*}
Then for every $(x, y) \in S_\epsilon^x \cup S_\epsilon^y$, we have
\begin{equation*}
    \ep^{m-1}|x-y| \leq|x^m-y^m| .
\end{equation*}
\end{lemma}
Next, we state the algebraic inequality from [\citealp{ineq}, Lemma 2.1].
\begin{lemma}{\label{p case}} Let $1<p<\infty$. Then for any $a,b\in \mathbb{R}^n$, there exists a constant $C\equiv C(p)>0$, such that
\begin{equation*}
    \iprod{|a|^{p-2}a-|b|^{p-2}b}{a-b}\geq C\frac{|a-b|^2}{(|a|+|b|)^{2-p}}.
\end{equation*}
\end{lemma}
\subsection{\textbf{Weak Solutions and Main Results}}
In this subsection, we define the notion of a weak solution to different problems and state our main results.
\begin{definition} (Tail Space)
    Let $u: \mathbb{R}^n \rightarrow \mathbb{R}$ be a measurable function and let $0<m<\infty$ and $\alpha>0$. We define the tail space by
\begin{equation*}
L_\alpha^m(\mathbb{R}^n)=\left\{u \in L_{\mathrm{loc}}^m(\mathbb{R}^n): \int_{\mathbb{R}^n} \frac{|u(x)|^m d x}{(1+|x|^{n+\alpha})}<\infty\right\} .
\end{equation*}
The nonlocal tail centered at $x_0 \in \mathbb{R}^n$ with radius $R>0$ is defined by
\begin{equation*}
T_{m, \alpha}(u ; x_0, R)=\left(R^\alpha \int_{B_R(x_0)^c} \frac{|u(y)|^m}{|y-x_0|^{n+\alpha}} d y\right)^{1 / m} .    
\end{equation*}
Set $T_{m, \alpha}(u ; R)=T_{m, \alpha}(u ; 0, R)$. We will follow the notation
\begin{equation*}
      T_{q-1}(u ; x, R):=T_{q-1, s q}(u ; x, R),
\end{equation*}
unless otherwise stated.
\end{definition}
\begin{definition}{\label{11}} (Local weak solution) 
   Let $f\in L^m_{\mathrm{loc}}(\Omega)$, where $m\geq (p^*)^\prime \text { if } p<n, m>1 \text{ if } p=n \text{ and } m\geq 1 \text{ if } p>n$. A function $u\in W^{1,p}_{\mathrm{loc}}(\Omega)\cap L^{q-1}_{sq}(\mathbb{R}^n)$ is said to be a local weak sub-solution (or super-solution) of 
   \begin{equation}{\label{p}}
-\Delta_p u+(-\Delta)_q^s u=f\text { in } \Omega, 
\end{equation}
if for every $K\subset\subset \Omega$ and for every non-negative function $\phi\in W_0^{1,p}(K)$, we have
   \begin{equation}{\label{weakine}}
    \begin{array}{c}
        \int_\Omega|\nabla u|^{p-2}\nabla u\cdot\nabla \phi \,d x+\int_{\mathbb{R}^n}\int_{\mathbb{R}^n}\frac{|u(x)-u(y)|^{q-2}(u(x)-u(y))(\phi(x)-\phi(y))}{|x-y|^{n+sq}} d x d y\leq (\text{or}\geq )\int_\Omega f\phi\, dx.
        \end{array}
    \end{equation}
We say $u$ is a weak solution to \cref{p} if equality holds in \cref{weakine} for all $\phi \in W_0^{1, p}(K) $.
\end{definition}
We now state our first H\"older regularity result for \cref{p}.
\begin{theorem}{\label{holder1}}
Suppose $2 \leq q \leq p<\infty$ and $f\in L^\infty_{\mathrm{loc}}(\Omega)$. Let
$
u \in W_{\mathrm{loc}}^{1, p}(\Omega) \cap L_{sq}^{q-1}(\mathbb{R}^n)
$
be a local weak solution to \cref{p}. Then, for every $\sigma\in(0,1)$,
there holds $u \in C_{\mathrm{loc}}^{0, \sigma}(\Omega)$. Moreover, for every  $\sigma \in(0,1)$ and ${R}_0 \in(0,1)$ with $B_{2 {R}_0}(x_0) \subset\subset \Omega$, the following holds:
\begin{equation*}
    [u]_{C^{0, \sigma}(B_{R_0 / 2}(x_0))} \leq\frac{C}{R^\sigma}\left(1+T_{q-1}(u ; x_0, {R}_0)+\|u\|_{L^{\infty}(B_{R_0}(x_0))}+\|f\|_{L^{\infty}(B_{R_0}(x_0))}
\right),
\end{equation*}
for some $ C\equiv C(n,s,p,q,\sigma)>0$. 
\end{theorem}
Now, we state some improved Hölder regularity results for a general class of operators. We consider
\begin{equation}{\label{general}}
\mathcal{L} u:=\mathcal{L}_L u+\mathcal{L}_N u
\end{equation}
defined as the sum of the local term
\begin{equation*}
\mathcal{L}_L u(x)=\mathcal{L}_L^A u(x):=-\operatorname{div} A(x, \nabla u(x))    
\end{equation*}
and of the nonlocal one
\begin{equation*}
    \mathcal{L}_N u(x)=\mathcal{L}_N^{\Phi, B, s, q} u(x):= \text { P.V. } \int_{\mathbb{R}^n} \Phi(u(x)-u(y)) \frac{B(x, y)}{|x-y|^{n+s q}} d y .
\end{equation*}
Here, $A: \Omega \times \mathbb{R}^n \rightarrow \mathbb{R}^n$ is a continuous vector field such that $A(x, \cdot) \in C^1(\mathbb{R}^n \backslash\{0\} ; \mathbb{R}^n)$ for all $x \in \Omega, A(\cdot, z) \in C^{\bar{\alpha}}(\Omega ; \mathbb{R}^n)$ for all $z \in \mathbb{R}^n$, and which satisfies the following $p$-growth and coercivity conditions
\begin{equation}{\label{growth1}}
\begin{cases}|A(x, z)|+|z|\left|\partial_z A(x, z)\right| \leq \Lambda\left(|z|^2+\mu^2\right)^{\frac{p-2}{2}}|z| & \text { for } x \in \Omega, z \in \mathbb{R}^n \backslash\{0\}, \\ |A(x, z)-A(y, z)| \leq \Lambda\left(|z|^2+\mu^2\right)^{\frac{p-1}{2}}|x-y|^{\bar{\alpha}} & \text { for } x, y \in \Omega, z \in \mathbb{R}^n, \\ \left\langle\partial_z A(x, z) \xi, \xi\right\rangle \geq \Lambda^{-1}\left(|z|^2+\mu^2\right)^{\frac{p-2}{2}}|\xi|^2 & \text { for } x \in \Omega, z \in \mathbb{R}^n \backslash\{0\}, \xi \in \mathbb{R}^n,\end{cases}
\end{equation}
for some constants $\bar{\alpha} \in(0,1), \mu \in[0,1]$, and $\Lambda \geq 1$, while $B: \mathbb{R}^n \times \mathbb{R}^n \rightarrow[0,+\infty)$ is a measurable function satisfying
\begin{equation}{\label{growth2}}
B(x, y)=B(y, x) \quad \text { and } \quad \Lambda^{-1} \leq B(x, y) \leq \Lambda \quad \text { for a.e. } x, y \in \mathbb{R}^n
\end{equation}
and $\Phi \in C^0(\mathbb{R})$ is an odd, non-decreasing function fulfilling the $q$-growth and coercivity condition
\begin{equation}{\label{growth3}}
    \Lambda^{-1}|t|^q \leq \phi(t) t \leq \Lambda|t|^q \quad \text { for all } t \in \mathbb{R} \text {. }
\end{equation}
In \cref{singu}, we will use Hopf lemma, and for that, we will assume $B \in C^{0,1}(\mathbb{R}^n \times \mathbb{R}^n)$ and $\phi \in C^1(\mathbb{R} \backslash\{0\})$ with
\begin{equation}{\label{grrrr}}
    |B(x+w, y+z)-B(x, y)| \leq \Lambda(|w|+|z|)\quad\forall x, y, w, z \in \mathbb{R}^n,\quad\Lambda^{-1}|t|^{q-2} \leq \phi^{\prime}(t) \leq \Lambda|t|^{q-2}\quad \forall t \in \mathbb{R} \backslash\{0\}.
\end{equation}
A prototypical example of such an operator is, of course
$-\Delta_p u+\left(-\Delta_q\right)^s u,
$ which is obtained by taking $A(x, z)=|z|^{p-2} z, \phi(t)=|t|^{q-2} t$, and $B$ equal to a constant. Our next result is an improved version of \cref{holder1}.
\begin{definition}{\label{121}}
   Let $f\in L^m_{\mathrm{loc}}(\Omega)$, where $m\geq (p^*)^\prime \text { if } p<n, m>1 \text{ if } p=n \text{ and } m\geq 1 \text{ if } p>n$. A function $u\in W^{1,p}_{\mathrm{loc}}(\Omega)\cap L^{q-1}_{sq}(\mathbb{R}^n)$ is said to be a local weak sub-solution (or super-solution) of 
   \begin{equation}{\label{p222}}
\mathcal{L}u=f\text { in } \Omega, 
\end{equation}
if for every $K\subset\subset \Omega$ and for every non-negative function $\phi\in W_0^{1,p}(K)$, we have
   \begin{equation}{\label{weakine1}}
    \begin{array}{c}
        \int_\Omega A(x,\nabla u)\cdot\nabla \phi \,d x+\int_{\mathbb{R}^n}\int_{\mathbb{R}^n}\frac{\Phi(u(x)-u(y))(\phi(x)-\phi(y))B(x,y)}{|x-y|^{n+sq}} d x d y\leq (\text{or}\geq )\int_\Omega f\phi\, dx.
        \end{array}
    \end{equation}
We say $u$ is a weak solution to \cref{p222} if equality holds in \cref{weakine1} for all $\phi \in W_0^{1, p}(K) $.
\end{definition}\begin{theorem}{\label{holder2}}
   (a) Let $2 \leq q \leq p<\infty$, $f\in L^n_{\mathrm{loc}}(\Omega)$ and
$
u \in W_{\mathrm{loc}}^{1, p}(\Omega) \cap L_{sq}^{q-1}(\mathbb{R}^n)
$
be a local weak solution to
\begin{equation*}
    \mathcal{L}u=f \text { in } \Omega.
\end{equation*}Then, for each $\eta\in(0,1)$, one has $u \in C_{\mathrm{loc}}^{0, \eta}(\Omega)$. Moreover, for every  $\eta\in(0,1)$ and ${R}_0 \in(0,1)$ with $B_{{R}_0}(x_0) \subset\subset \Omega$,
$
    [u]_{C^{0, \eta}(B_{R_0 / 4}(x_0))} \leq C
$
holds for some $ C\equiv C(n,s,p,q,\bar{\alpha},\Lambda, \eta,\|u\|_{L^p(B_{{R}_0}(x_0))},\|f\|_{L^n(B_{{R}_0}(x_0))}, Tail_{q-1}(u;x_0,R_0))$.
\smallskip\\
(b) Suppose $1 <q \leq p<\infty$ and $f\in L^n_{\mathrm{loc}}(\Omega)$. Let
$
0\leq u \in W_{\mathrm{loc}}^{1, p}(\Omega)
$ be such that there exists some $\theta \geq 1$ with $u^\theta\in W_0^{1,p}(\Omega)$.
If $u$ is bounded and is a local weak solution to
\begin{equation*}
    \mathcal{L}u=f \text { in } \Omega,
\end{equation*} then, for every $\eta\in(0,1)$,
one has $u \in C_{\mathrm{loc}}^{0, \eta}(\Omega)$. Moreover, for every  $\eta \in(0,1)$ and ${R}_0 \in(0,1)$ with $B_{ {R}_0}(x_0) \subset\subset \Omega$,
$
    [u]_{C^{0, \eta}(B_{R_0 / 4}(x_0))} \leq C
$
holds for some $ C\equiv C(n,s,p,q,\bar{\alpha},\Lambda, \eta,\|u\|_{L^p(B_{{R}_0}(x_0))},\|f\|_{L^n(B_{{R}_0}(x_0))}, \|u^\theta\|_{L^q(\mathbb{R}^n)})>0$.
\end{theorem}
Following is the gradient H\"older regularity for local weak solutions.
\begin{theorem}{\label{holder3}}
   (a) Suppose $2 \leq q \leq p<\infty$ and $f\in L^d_{\mathrm{loc}}(\Omega)$ for some $d>n$. Let
$
u \in W_{\mathrm{loc}}^{1, p}(\Omega) \cap L_{sq}^{q-1}(\mathbb{R}^n)
$
be a local weak solution to
\begin{equation*}
    \mathcal{L}u=f \text { in } \Omega.
\end{equation*}Then, there exists $\bar{\eta}\in(0,1)$ such that $u \in C_{\mathrm{loc}}^{1, \bar{\eta}}(\Omega)$.\\
(b) Suppose $1 <q \leq p<\infty$ and $f\in L^d_{\mathrm{loc}}(\Omega)$ for some $d>n$. Let
$
0\leq u \in W_{\mathrm{loc}}^{1, p}(\Omega)
$ be such that there exists some $\theta \geq 1$ with $u^\theta\in W_0^{1,p}(\Omega)$.
If $u$ is bounded and is a local weak solution to
\begin{equation*}
    \mathcal{L}u=f \text { in } \Omega,
\end{equation*} then, there exists $\gamma\in (0,1)$,
such that $u \in C_{\mathrm{loc}}^{1, \gamma}(\Omega)$. 
\end{theorem}
Now, we state the boundary regularity result that has been obtained.
\begin{theorem}{\label{holder4}}
    Suppose $1< q \leq p<\infty$ and $f$ satisfies $0\leq f\leq C d(x, \partial\Omega)^{-\sigma}$ in $\Omega$ for some $\sigma\in[0,1)$. Let $u \in W_0^{1, p}(\Omega)$ be a solution to the problem 
    \begin{equation*}{\label{p2}}
        \mathcal {L}u=f \text{ in }\Omega; \quad u=0\text{ in }\mathbb{R}^n\backslash\Omega,
    \end{equation*} satisfying $0\leq u\leq M$ and $0\leq u\leq Cd(x,\partial\Omega)^\ep$ a.e. in $\Omega$ for some $M,C=C(\ep)>0$ and for some $\ep\in[\sigma,1)$. Then there exists $\gamma\in(0,1)$ such that $u\in C^{1, \gamma}(\overline{\Omega})$ and $\|u\|_{C^{1, \gamma}(\overline{\Omega})}\leq c$ for some $c$ depending on $n,p,q,s,\Lambda,\bar{\alpha},\Omega,M,\sigma$.
\end{theorem} We get the following result for a perturbed problem.
\begin{theorem}{\label{prtb}}
Suppose $u \in W_0^{1, p}(\Omega)$ be a solution to  \begin{equation}{\label{unif}}
        -\Delta_p u(x)+(-\Delta)^s_q u(x)=f(x,u(x)) \text{ in }\Omega; \quad u=0\text{ in }\mathbb{R}^n\backslash\Omega,
    \end{equation} with $f(x):=f(x,u(x))$ is a Carath\'eodory function satisfying $|f(x,t)|\leq C_0(1+|t|^{p^*-1})$, where $C_0>0$ and $p^*$ is the critical Sobolev exponent if $p<n$ and otherwise an arbitrary large number. Then $u\in C^{1, \gamma}(\overline{\Omega})$ for some $\gamma\in(0,1)$.
\end{theorem}
We now turn to the singular problem given by 
\begin{equation}{\label{problem}}
\begin{array}{c}
\begin{split}
-\Delta_pu+(-\Delta)_q^s u&={f(x)}{u^{-\delta}}\text { in } \Omega, \\u&>0 \text{ in } \Omega,\\u&=0  \text { in }\mathbb{R}^n \backslash \Omega;
\end{split}
\end{array}
\end{equation}
where $f\in L^\infty_{\mathrm{loc}}(\Omega )$ and satisfies the condition 
\begin{equation}{\label{condf}}
    c_1d(x,\partial\Omega)^{-\beta}\leq f(x)\leq c_2d(x,\partial\Omega)^{-\beta},
\end{equation}
in $\Omega_\rho=\{x\in\overline{\Omega}:d(x,\partial\Omega)<\rho\}$. Also, $1<q\leq p<\infty, s\in(0,1), \delta>0, \Omega\subset\mathbb{R}^n$ is a bounded domain with $C^2$ boundary. Since $\partial \Omega$ is $C^2$, it follows from [\citealp{gilbarg1977elliptic}, Lemma 14.6, Page 355] that there exists $\varrho_0>0$ such that $d \in C^2(\Omega_{\varrho_0})$. Without loss of generality, we may assume $\rho \leq \min \left\{\varrho_0/2, 1\right\}$, so that $|\Delta d| \in L^{\infty}(\Omega_{\rho})$ and \cref{condf} also holds. We define the notion of weak solution to \cref{problem} as follows.
\begin{definition}{\label{behaveu}} We say that $u \leq 0$ on $\partial \Omega$, if $(u-\epsilon)^{+} \in W_0^{1, p}(\Omega)$ for every $\epsilon>0$ and $u=0$ on $\partial \Omega$ if $u \geq 0$ in $\Omega$ and $u \leq 0$ on $\partial \Omega$.
\end{definition}\begin{definition}{\label{mainweaksol}}
    A function $u \in W_{\mathrm{loc}}^{1, p}(\Omega)$ is said to be a weak sub-solution (resp. supersolution) of problem \cref{problem} if $u>0$ in $\Omega$, ${f}{u^{-\delta}}\in L_{\mathrm{loc}}^1(\Omega)$ and the followings hold: 
\\(i) for all $\phi \in C_c^{\infty}(\Omega)$, with $\phi \geq 0$ in $\Omega$,
\begin{equation*}
    \int_{\Omega}|\nabla u|^{p-2} \nabla u\cdot \nabla \phi+\int_{\mathbb{R}^n}\int_{\mathbb{R}^n}\frac{|u(x)-u(y)|^{q-2}(u(x)-u(y))(\phi(x)-\phi(y))}{|x-y|^{n+sq}} d x d y
     \leq(\text { resp. } \geq) \int_{\Omega} f(x) u^{-\delta} \phi,
\end{equation*}
(ii) there exists $\gamma \geq 1$ such that $u^\gamma \in W_0^{1, p}(\Omega)$. This condition implies that the boundary behavior of $u$ is verified according to \cref{behaveu}.\\
A function which is both sub and super solution of \cref{problem} is called a weak solution.
\end{definition} 
\begin{remark}
    By the positivity of $u$ and condition $(ii)$ above, the corresponding tail term is finite and using \cref{embedding} and \cref{embedding2}, the inequality mentioned in the definition makes sense.
\end{remark}
\begin{definition}{\label{conicalshell}}We say a weak solution $u$ of \cref{problem}, is in the conical shell $\mathcal{C}_{d_{\beta, \delta}}$ if it is continuous and satisfies the following
\begin{equation*}
  \begin{cases}\eta d(x) \leq u(x) \leq \Gamma d(x)^{\bar{\sigma}}   & \text { if } \beta+\delta\leq 1, \text{ for all } 0<\bar{\sigma}<1, 
    \\ \eta d(x)^{\frac{p-\beta}{p-1+\delta}} \leq u(x) \leq \Gamma d(x)^{\frac{p-\beta}{p-1+\delta}} & \text { if } \beta+\delta>1 \text{ with }  \frac{p-\beta}{p-1+\delta}\in(0,s),\\\eta d(x)\leq u(x) \leq \Gamma d(x)^{\frac{p-\beta}{p-1+\delta}} & \text { if } \beta+\delta>1 \text{ with }  \frac{p-\beta}{p-1+\delta}\in[s,1), \beta \neq p-q^\prime s(p-1+\delta),\\\eta d(x) \leq u(x) \leq \Gamma d(x)^{\frac{p-\beta_1}{p-1+\delta}}& \text { if } \beta+\delta>1 \text{ with } \beta = p-q^\prime s(p-1+\delta), \text{ for all } \beta_1\in(\beta, p),
    \end{cases}
\end{equation*}
for some positive constants $\eta, \Gamma>0$.
\end{definition}
Now, we state our existence result.
\begin{theorem}{\label{mainexis}}
    Let $\beta \in[0, p)$, then \cref{problem} admits a weak solution (obtained as a limit of solution to approximated problem \cref{approximated} of problem \cref{problem}) $u \in W_{\mathrm{loc }}^{1, p}(\Omega) \cap \,\mathcal{C}_{d_{\beta, \delta}}$, in the sense of \cref{mainweaksol}. This existence result (without the fact that $u\in  \mathcal{C}_{d_{\beta, \delta}}$) also holds if we consider the same singular problems associated with the general operator given by \cref{general} (considering $\mu=0$ in \cref{growth1}).
\end{theorem}
To obtain the uniqueness result, we establish the following weak comparison principle.
\begin{theorem}{\label{uni}}
    Let $\beta<2-\frac{1}{p}$ and $u, v \in W_{\mathrm {loc }}^{1, p}(\Omega)$ be sub and super solution of \cref{problem}, respectively in the sense of \cref{mainweaksol}. Then, $u \leq v$ a.e. in $\Omega$. Again, this uniqueness result holds if we consider singular problems associated with the general operator given by \cref{general} (considering $\mu=0$ in \cref{growth1}).
\end{theorem}
Next follows the optimal Sobolev regularity obtained. 
\begin{theorem}{\label{optimalregu}}
     Let $\beta+\delta>1$, and take $\Lambda=\frac{(p-1)(p-1+\delta)}{p(p-\beta)}$. If ${p}<\beta+s(p+\delta-1)$ then we have weak solution $u$ to \cref{problem}, obtained as the limit of approximating sequence belongs to $  W_0^{1,p}(\Omega)$ iff $\Lambda<1$ i.e. iff $\delta<2+\frac{1-\beta p}{p-1}$ and
$u^\theta\in W_0^{1,p}(\Omega)$ iff $\theta>\Lambda\geq 1$. Further for ${p}\geq\beta+s(p+\delta-1)$, define $\Lambda=\frac{(p-1)(p-1+\delta)}{p(p-\beta)}$ if $\beta \neq p-q^\prime s(p-1+\delta)$. 
Then the weak solution $u$ to \cref{problem} belongs to $  W_0^{1,p}(\Omega)$ if $\Lambda<1$ and
$u^\theta\in W_0^{1,p}(\Omega)$ if $\theta>\Lambda\geq 1$.
\end{theorem}
Next, we have a non-existence result for weak solution of \cref{problem}.
\begin{theorem}{\label{nonexist}}
     Let $\beta \geq p$ in \cref{condf}. Then, there does not exist any weak solution of problem \cref{problem} in the sense of \cref{mainweaksol}.
\end{theorem}
Finally, we prove the following almost optimal global regularity result in the singular setting.
\begin{theorem}{\label{holder6}}
    Suppose $\beta \in[0, p )$ and $\delta>0$. Let $u\in W_{\mathrm{loc}}^{1, p}(\Omega)$ be either the solution given by \cref{mainexis} or the unique solution to problem \cref{problem}. Then the following assertions hold:\\
(i) If $\beta+\delta < 1$, then $u \in C^{1, \sigma}(\overline{\Omega})$ for some $\sigma\in(0,1)$.\\
(ii) If $\beta+\delta = 1$, then $u\in C^{0,\eta}(\overline{\Omega})
$ for all $\sigma\in(0,1)$.\\
(iii) If $\beta+\delta>1$ with $\beta \neq p -q^{\prime} s(p-1+\delta)$, then
\begin{equation*}
u\in C^{0, \frac{p -\beta}{p-1+\delta}}(\overline{\Omega}) . 
\end{equation*}
(iv) If $\beta+\delta>1$ with $\beta = p -q^{\prime} s(p-1+\delta)$, then
\begin{equation*}
u\in C^{0, \frac{p -\beta_1}{p-1+\delta}}(\overline{\Omega}) , 
\quad\text { for all } \beta_1 \in\left(\beta, p \right).
\end{equation*}
Also, in the cases (ii), (iii) and (iv), $u\in C_{\mathrm{loc}}^{1, \gamma}(\overline{\Omega})$ for some $\gamma\in(0,1)$.
\end{theorem} 
As an application of the above results, we deduce the following corollary (compare this with \cref{prtb}) for critical growth problems.
\begin{corollary}{\label{finalcor}}
    Let $u \in W_0^{1, p}(\Omega)$ be a weak solution to the problem
\begin{equation}{\label{problem2}}
\begin{array}{c}
\begin{split}
-\Delta_pu+(-\Delta)_q^s u&=\frac{\lambda}{u^{\delta}}+b(x,u)\text { in } \Omega, \\u&>0 \text{ in } \Omega,\\u&=0  \text { in }\mathbb{R}^n \backslash \Omega;
\end{split}
\end{array}
\end{equation}
where $\delta \in(0,1), \lambda>0$ is a parameter and $b: \Omega \times \mathbb{R} \rightarrow \mathbb{R}$ is a Carathéodory function satisfying
\begin{equation*}
    0\leq b(x, t)\leq C_b(1+|t|^{p^*-1}),
\end{equation*}
with $C_b>0$ as a constant. Then $u \in C^{1, \sigma}(\overline{\Omega})$ for some $\sigma \in(0, 1)$.
\end{corollary}
\section{Regularity results for regular problems}{\label{regu1}}
The following energy estimate will be crucial for us. Let us denote an open ball with center $x_0 \in \mathbb{R}^n$ and radius $r>0$ by $B_r(x_0)$.
\begin{lemma}{\label{caccioppoli}}
    Let $1<q\leq p<\infty, f\in L^n_{\mathrm{loc}}(\Omega)$ and $u\in W^{1,p}_{\mathrm{loc}}(\Omega)\cap L^{q-1}_{sq}(\mathbb{R}^n)$ be a local weak subsolution of \cref{p} and denote $w=(u-k)^{+}$ with $k \in \mathbb{R}$. There exists a positive constant $C=C(p,q,n,s
    )$ such that
\begin{equation}{\label{cacc}}
    \begin{array}{l}
         \int_{B_r(x_0)} \psi^p|\nabla w|^p d x+\int_{B_r(x_0)} \int_{B_r(x_0)}|w(x) \psi^{p/q}(x)-w(y) \psi^{p/q}(y)|^q d \mu \\
\leq C\left(\int_{B_r(x_0)} w^p|\nabla \psi|^p d x+\int_{B_r(x_0)} \int_{B_r(x_0)} \max \{w(x), w(y)\}^q|\psi^{p/q}(x)-\psi^{p/q}(y)|^q d \mu\right. \\
\left.\quad+\int_{B_r(x_0)}\left(\int_{\mathbb{R}^n \backslash B_r(x_0)} \frac{w(y)^{q-1}}{|x-y|^{n+ sq}} d y \right) w \psi^p d x+r^{-sq}\int_{B_r(x_0)}w^q dx\right)+\int_{B_r(x_0)}|f|w\psi^p dx,
    \end{array}
\end{equation}
whenever $B_r\left(x_0\right) \subset \Omega$ and $\psi \in C_c^{\infty}(B_r(x_0))$ is a nonnegative function. We have denoted $d\mu=\frac{1}{|x-y|^{n+sq}}dx dy$. If $u$ is a weak supersolution of \cref{p}, the estimate in \cref{cacc} holds with $w=(u-k)^{-}$.
\end{lemma}
\begin{proof} For $w=(u-k)^{+}$, we choose $\phi=w \psi^p$ as a test function in \cref{weakine}, to obtain
\begin{equation}{\label{1}}
    \begin{array}{rcl}
         \int_{B_r(x_0)}|f|w\psi^p dx\geq \int_{B_r(x_0)}fw\psi^p dx &\geq & \int_{B_r(x_0)}|\nabla u|^{p-2} \nabla u \cdot \nabla\left(w \psi^p\right) d x \smallskip\\
& &+\int_{\mathbb{R}^n} \int_{\mathbb{R}^n} |u(x)-u(y)|^{q-2}(u(x)-u(y))(w(x) \psi(x)^p-w(y) \psi(y)^p) d \mu \\
&= & I+J.
    \end{array}
\end{equation}
Proceeding as in the proof of [\citealp{Liao}, Page 14, Proposition 3.1], for some constants $c=c(p)>0$ and $C=C(p)>0$, we have
\begin{equation}{\label{20}}
    I  =\int_{B_r(x_0)}|\nabla u|^{p-2} \nabla u \cdot \nabla(w \psi^p) d x  \geq c \int_{B_r(x_0)} \psi^p|\nabla w|^p d x-C \int_{B_r(x_0)} w^p|\nabla \psi|^p d x .
\end{equation}
Moreover, from the lines of the proof of [\citealp{MG}, Pages 24-25, Lemma 3.1], for some constants $c=c(p,q,s,n, \Lambda)>0$ and $C=C(p,q,s,n, \Lambda)>0$, we have

\begin{equation}{\label{30}}
   \begin{array}{rcl} J&= &\int_{\mathbb{R}^n} \int_{\mathbb{R}^n} |u(x)-u(y)|^{q-2}(u(x)-u(y))(w(x) \psi(x)^p-w(y) \psi(y)^p) d \mu\smallskip\\&\geq & \int_{\mathbb{R}^n} \int_{\mathbb{R}^n} |w(x)-w(y)|^{q-2}(w(x)-w(y))(w(x) \psi(x)^p-w(y) \psi(y)^p) d \mu\smallskip\\
&\geq & c \int_{B_r(x_0)} \int_{B_r(x_0)}|w(x) \psi^{p/q}(x)-w(y) \psi^{p/q}(y)|^q d \mu\smallskip\\
&& \quad-C \int_{B_r(x_0)} \int_{B_r(x_0)} \max \{w(x), w(y)\}^q|\psi^{p/q}(x)-\psi^{p/q}(y)|^q d \mu\smallskip\\
&& \quad-Cr^{-sq}\int_{B_r(x_0)}w^q dx-C \int_{B_r(x_0)}\left(\int_{\mathbb{R}^n \backslash B_r(x_0)} \frac{w(y)^{q-1}}{|x-y|^{n+ sq}} d y \right) w \psi^p d x.
\end{array}
\end{equation}By applying \cref{20,30} in \cref{1}, we obtain \cref{cacc}. In the case of a weak supersolution, the estimate in \cref{cacc} follows by applying the obtained result to $-u$
\end{proof}
\begin{remark}{\label{caccremark}}
It should be noted that the same form of energy estimate also holds for the general operator $\mathcal{L}$ described in \cref{general}. The proof follows precisely similar to [\citealp{MG}, Lemma 3.1].
\end{remark}
Next, we show the local boundedness of weak solutions. We apply the following real analysis lemma. For the proof, see [\citealp{dibenedetto2012degenerate}, Lemma 4.1].
\begin{lemma}{\label{iteration}}
    Let $\left(Y_i\right)_{i=0}^{\infty}$ be a sequence of positive real numbers such that $Y_0 \leq$ $c_0^{-\frac{1}{\beta}} b^{-\frac{1}{\beta^2}}$ and $Y_{i+1} \leq c_0 b^i Y_i^{1+\beta}, i=0,1,2, \ldots$, for some constants $c_0, b>1$ and $\beta>0$. Then $\lim _{i \rightarrow \infty} Y_i=0$.
\end{lemma} 
\begin{theorem}{\label{local boundedness}}
(Local boundedness). Let $1< q\leq p<\infty$ and $u$ be a local weak subsolution of \cref{p}. Then there exists a positive constant $c=c(n, p, q, s
)$, such that
\begin{equation}{\label{bdd}}
    \underset{B_{\frac{r}{2}}(x_0)}{\operatorname{ess} \sup } \,u \leq \operatorname{Tail}_{q-1}(u^{+} ; x_0, \frac{r}{2})^{\frac{q-1}{p-1}}+c \left(\fint_{B_r(x_0)} (u^{+})^p d x\right)^{\frac{1}{p}}+\|f\|_{L^\infty(B_r)}^{\frac{1}{p-1}}+1,
\end{equation}
whenever $B_r(x_0) \subset \subset\Omega$ with $r \in(0,1]$. Moreover, if $u$ is a weak supersolution, then \cref{bdd} holds with $-u$. 
\end{theorem} 
\begin{proof} Let $B_r(x_0) \subset \Omega$ with $r \in(0,1]$. For $i=0,1,2, \ldots$, let us denote $r_i=\frac{r}{2}(1+$ $\left.2^{-i}\right), \bar{r}_i=\frac{r_i+r_{i+1}}{2}, B_i=B_{r_i}(x_0)$ and $\bar{B}_i=B_{\bar{r}_i}(x_0)$. Let $(\psi_i)_{i=0}^{\infty} \subset C_c^{\infty}(\bar{B}_i)$ be a sequence of cutoff functions such that $0 \leq \psi_i \leq 1$ in $\bar{B}_i, \psi_i=1$ in $B_{i+1}$ and $|\nabla \psi_i| \leq \frac{2^{i+3}}{r}$ for every $i=0,1,2, \ldots$. For $i=0,1,2, \ldots$ and $k, \bar{k} \geq 0$, we denote $k_i=k+(1-2^{-i}) \bar{k}, \bar{k}_i=\frac{k_i+k_{i+1}}{2}, w_i=\left(u-k_i\right)^{+}$ and $\bar{w}_i=\left(u-\bar{k}_i\right)^{+}$. Then $\exists \,c=c(n, p)>0$ such that
\begin{equation}{\label{4.2}}
\left(\frac{\bar{k}}{2^{i+2}}\right)^{\frac{p(\kappa-1)}{\kappa}}\left(\fint_{B_{i+1}} w_{i+1}^p d x\right)^{\frac{1}{\kappa}}  =\left(k_{i+1}-\bar{k}_i\right)^{\frac{p(\kappa-1)}{\kappa}}\left(\fint_{B_{i+1}} w_{i+1}^p d x\right)^{\frac{1}{\kappa}}  \leq c\left(\fint_{\bar{B}_i}|\bar{w}_i \psi_i|^{p \kappa} d x\right)^{\frac{1}{\kappa}},    
\end{equation}
where $\kappa$ is given by \cref{kappa}. Next by Sobolev inequality in \cref{sobolev embedding 2}, with $c=c(n, p, s)>0$, we obtain
\begin{equation}{\label{I}}
    \begin{array}{l}
\left(\fint_{\bar{B}_i}|\bar{w}_i \psi_i|^{p \kappa} d x\right)^{\frac{1}{\kappa}}  \leq c r^{p-n} \int_{B_i}|\nabla(\bar{w}_i \psi_i)|^p d x \leq c r^{p-n}\left(\int_{B_i} \bar{w}_i^p|\nabla \psi_i|^p d x+\int_{B_i} \psi_i^p|\nabla \bar{w}_i|^p d x\right)=I_1+I_2 .
\end{array}
\end{equation}
\textbf{Estimate of $I_1$:} Using the properties of $\psi_i$, for some $c=c(n, p, s)>0$, we have
\begin{equation}{\label{I1}}
I_1=c r^{p-n} \int_{B_i} \bar{w}_i^p|\nabla \psi_i|^p d x \leq c 2^{i p} \fint_{B_i} w_i^p d x .    
\end{equation}
\textbf{Estimate of $I_2$:} By \cref{cacc}, with $c=c(n, p, s)$ and $C=(n, p,q, s
)$ positive, we obtain
\begin{equation}{\label{J}}
    \begin{array}{rcl}
I_2&= & c r^{p-n} \int_{B_i} \psi_i^p|\nabla \bar{w}_i|^p d x \\
&\leq & C r^{p-n}\left(\int_{B_i} \bar{w}_i^p|\nabla \psi_i|^p d x+\int_{B_i} \int_{B_i} \max \{\bar{w}_i(x), \bar{w}_i(y)\}^q|\psi_i^{p/q}(x)-\psi_i^{p/q}(y)|^qd \mu\right. \\
&& \left.+\int_{B_i} \bar{w}_i(x) \psi_i(x)^p d x \cdot \underset{x \in \operatorname{supp} \psi_i}{\operatorname{ess} s u p} \int_{\mathbb{R}^n \backslash B_i}\frac{ \bar{w}_i(y)^{q-1}}{|x-y|^{n+sq}}  d y+r^{-sq}\int_{B_i}\bar{w}_i^q dx+\int_{B_i}|f|\bar{w}_i\psi_i^p dx\right) \\
&= & J_1+J_2+J_3+J_4+J_5 .         
    \end{array}
\end{equation}
\textbf{Estimates of $J_1, J_2$ and $J_4$:} To estimate $J_1$, we use the estimate of $I_1$ in \cref{I1} above and to estimate $J_2$, noticing the bounds on $\psi_i$ and the relation
$
\bar{w}_i^q \leq \frac{w_i^p}{(\bar{k}_i-k_i)^{p-q}},    
$
we have
\begin{equation*}
\begin{array}{rcl}
     \int_{B_i}\int_{ B_i} \frac{|\psi_i^{p/q}(x)-\psi_i^{p/q}(y)|^q}{|x-y|^{n+s q}}\max \{\bar{w}_i(x), \bar{w}_i(y)\}^q d x d y & \leq &2c \frac{2^{iq}}{r^q} \int_{B_i} \frac{w_i^p(x)}{(\bar{k}_i-k_i)^{p-q}}\left(\int_{B_i} \frac{d y}{|x-y|^{n+s q-q}}\right) d x \smallskip\\
& \leq &c \frac{2^{i(q+p-q)}}{q(1-s)} \frac{1}{\bar{k}^{p-q}} \frac{r^n}{r^{s q}} \fint_{B_i} w_i^p(x) d x. 
\end{array}
\end{equation*} 
Finally for $J_4$, we estimate as 
\begin{equation*}
  J_4=C   r^{p-n-sq}\int_{B_i}\bar{w}_i^q dx\leq Cr^{p-n-sq}\int_{B_i} \frac{w_i^p(x)}{(\bar{k}_i-k_i)^{p-q}} dx\leq C r^{p-sq}\frac{2^{i(p-q)}}{\bar{k}^{p-q}}\fint_{B_i} w_i^p(x) d x,
\end{equation*}
so that we have
\begin{equation}{\label{J1J2}}
    J_j \leq c(n, p, q,s) 2^{ip} \fint_{B_i} w_i^p d x, \quad j=1,2, 4 .
\end{equation}
\textbf{Estimate of $J_3$:} For $x \in \bar{B}_i=\operatorname{supp}(\psi_i)$ and $y \in \mathbb{R}^n\backslash B_i$, we have
$
    \frac{|y-x_0|}{|x-y|} \leq 1+\frac{|x-x_0|}{|x-y|} \leq 1+\frac{\bar{r}_i}{r_i-\bar{r}_i} \leq 2^{i+4} .
$
Then, using the relation
$
    \bar{w}_i \leq \frac{w_i^p}{(\bar{k}_i-k_i)^{p-1}},
$
we obtain
\begin{equation}{\label{J3}}
\begin{array}{r}
     \int_{B_i} \bar{w}_i(x) \psi_i(x)^p d x \cdot \underset{x \in \operatorname{supp} \psi_i}{\operatorname{ess} s u p} \int_{\mathbb{R}^n \backslash B_i}\frac{ \bar{w}_i(y)^{q-1}}{|x-y|^{n+sq}}  d y 
 \leq c \frac{2^{i(n+s q)}}{(\bar{k}_i-k_i)^{p-1}}\left(\int_{B_i} w_i^p(x) d x\right)\left(\int_{\mathbb{R}^n \backslash B_i} \frac{\bar{w}_i(y)^{q-1} d y}{|y-x_0|^{n+s q}}\right) \smallskip\\
 \quad\leq c(n, p,q, s) \frac{2^{i(n+s q+p-1)}}{\bar{k}^{p-1}} \frac{r^n}{r^{s q}} T_{q-1}(w_0 ; x_0, \frac{r}{2})^{q-1} \fint_{B_i }w_i^p(x) d x .
\end{array}
\end{equation}
\textbf{Estimate of $J_5$:} Proceeding similarly to above, we obtain
\begin{equation}{\label{J5}}
\int_{B_i}|f(x)| \bar{w}_i(x) \psi_i^p(x) d x \leq c(n, p,q, s) \frac{2^{i(p-1)}}{\bar{k}^{p-1}} \| f\|_{L^{\infty}(B_r)} r^n\fint_{B_i} w_i^p(x) dx .    
\end{equation}
Substituting the estimate of \cref{J1J2,J3,J5} into \cref{J}, we get
\begin{equation}{\label{I2}}
    I_2 \leq c(n, p,q, s) 2^{i(n+sq+p-1)} \fint_{B_i} w_i^p d x,
\end{equation}
whenever 
$
    \bar{k} \geq T_{q-1}(w_0 ; x_0, \frac{r}{2})^{\frac{q-1}{p-1}}+\|f\|_{L^{\infty}(B_r)}^{1 /(p-1)}+1 .
$
Inserting \cref{I1,I2} into \cref{I} we have
\begin{equation}{\label{4.9}}
\left(\fint_{\bar{B}_i}|\bar{w}_i \psi_i|^{p \kappa} d x\right)^{\frac{1}{\kappa}} \leq c(n, p,q, s) 2^{i(n+sq+p-1)} \fint_{B_i} w_i^p d x .    
\end{equation}
Setting
\begin{equation*}
Y_i=\left(\fint_{B_i} w_i^p d x\right)^{\frac{1}{p}}    
\quad \text{ and }\quad
\bar{k}=T_{q-1}(w_0 ; x_0, \frac{r}{2})^{\frac{q-1}{p-1}}+\|f\|_{L^{\infty}(B_r)}^{1 /(p-1)}+1 +c_0^{\frac{1}{\beta}} b^{\frac{1}{\beta^2}}\left(\fint_{B_r(x_0)} w_0^p d x\right)^{\frac{1}{p}},    
\end{equation*}
where
\begin{equation*}
c_0=c(n, p,q, s), \quad b=2^{\left(\frac{n+sq+p-1}{p}+\frac{\kappa-1}{\kappa}\right) \kappa} \quad \text { and } \quad \beta=\kappa-1 ,   
\end{equation*}
from \cref{4.2,4.9} we obtain
\begin{equation*}
    \frac{Y_{i+1}}{\bar{k}} \leq c(n, p, q, s) 2^{i\left(\frac{n+sq+p-1}{p}+\frac{\kappa-1}{\kappa}\right) \kappa} \left(\frac{Y_i}{\bar{k}}\right)^\kappa .
\end{equation*}
Moreover, by the definition of $\bar{k}$ above we have
\begin{equation*}
    \frac{Y_0}{\bar{k}} \leq c_0^{-\frac{1}{\beta}} b^{-\frac{1}{\beta^2}}.
\end{equation*}
Thus from \cref{iteration}, we obtain $Y_i \rightarrow 0$ as $i\rightarrow \infty$. This implies that
\begin{equation*}
   \underset{B_{\frac{r}{2}}(x_0)}{\operatorname{ess} \sup } \,u \leq k+\bar{k},
\end{equation*}
which gives \cref{bdd} by choosing $k=0$.
Analogous treatment for $(u-k)^-$ gives the lower bound of $u$ in $B_{r/2}(x_0)$. This implies $u\in L^\infty(B_{r/2}(x_0))$. A standard covering argument gives $u\in L^\infty_{\mathrm{loc}}(\Omega)$.
\end{proof}
Now we prove \cref{holder1}. We first fix some notations which will be used in this subsection. For a measurable function $u: \mathbb{R}^n \rightarrow \mathbb{R}$ and $h \in \mathbb{R}^n$, we define
\begin{equation*}
    u_h(x)=u(x+h), \quad \delta_h u(x)=u_h(x)-u(x), \quad \delta_h^2 u(x)=\delta_h(\delta_h u(x))=u_{2 h}(x)+u(x)-2u_h(x) .
\end{equation*}
Now, we recall the discrete Leibniz rule
\begin{equation*}
    \delta_h(u v)=u_h \delta_h v+v \delta_h u \quad \text { and } \quad \delta_h^2(u v)=u_{2 h} \delta_h^2 v+2 \delta_h v \delta_h u_h+v \delta_h^2 u .
\end{equation*}
The following lemma iteration lemma will be useful.
\begin{lemma}{\label{iter}} Let $2 \leq q\leq p<\infty$ and $ 0<s<1$. Suppose that $u \in W_{\mathrm{loc}}^{1, p}\left(B_2(x_0)\right) \cap L_{s q}^{q-1}(\mathbb{R}^n)$ is a local weak solution of $-\Delta_p u+(-\Delta_q)^s u=f$ in $B_2(x_0)$. Let $f\in L^\infty_{\mathrm{loc}}(B_2(x_0))$ and
\begin{equation}{\label{connnn}}
    \|u\|_{L^{\infty}(B_1(x_0))} \leq 1, \quad \int_{\mathbb{R}^n\backslash B_1(x_0)} \frac{|u(y)|^{q-1}}{|y|^{n+sq}} d y \leq 1,\quad \|f\|_{L^\infty(B_1(x_0))}\leq 1.
\end{equation}
Let $0<h_0<{1}/{10}$ and $R$ be such that $4 h_0<R \leq 1-5 h_0$ and $\nabla u \in L^m(B_{R+4 h_0}(x_0))$ for some $m \geq p$. Then
\begin{equation*}
\sup _{0<|h|<h_0}\left\|\frac{\delta_h^2 u}{|h|^{1+\frac{1}{m+1}}}\right\|_{L^{m+1}(B_{R-4 h_0}(x_0))}^{m+1} \leq C\left(\int_{B_{R+4 h_0(x_0)}}|\nabla u|^m d x+1\right),
\end{equation*}
for some constant $C=Cn, h_0, p, q,m, s)>0$.
\end{lemma}
\begin{proof} Without loss of generality, one can assume $x_0=0$. Let $r=R-4 h_0$ and $\phi \in W^{1, p}(B_R)$ vanish outside $B_{\frac{R+r}{2}}$. Since $u$ is a local weak solution of $-\Delta_p u+(-\Delta_q)^s u=f$ in $B_2$, we have
\begin{equation}{\label{sub1}}
    \int_{B_R}|\nabla u|^{p-2} \nabla u \cdot\nabla \phi\, d x+ \int_{\mathbb{R}^n} \int_{\mathbb{R}^n}\frac{|u(x)-u(y)|^{q-2}(u(x)-u(y))(\phi(x)-\phi(y))}{|x-y|^{n+qs}} d x d y=\int_{B_R}f \phi\, dx .
\end{equation}
Let $h \in \mathbb{R}^n \backslash\{0\}$ be such that $|h|<h_0$. Choosing $\phi_{-h}$ as a test function and using a change of variable, we have
\begin{equation}{\label{sub2}}
    \int_{B_R}\left|\nabla u_h\right|^{p-2} \nabla u_h \cdot\nabla \phi\, d x  +  \int_{\mathbb{R}^n} \int_{\mathbb{R}^n}\frac{|u_h(x)-u_h(y)|^{q-2}(u_h(x)-u_h(y))(\phi(x)-\phi(y))}{|x-y|^{n+qs}} d x d y=\int_{B_R}f_h \phi\, dx.
\end{equation}
Subtracting \cref{sub1} with \cref{sub2} and dividing the resulting equation by $|h|$, we obtain
\begin{equation}{\label{hol3}}
    \begin{array}{l}
    \int_{B_R} \frac{(|\nabla u_h|^{p-2} \nabla u_h-|\nabla u|^{p-2} \nabla u)}{|h|} \cdot\nabla \phi \,d x 
+\int_{\mathbb{R}^n} \int_{\mathbb{R}^n} \frac{(J_q(u_h(x)-u_h(y))-J_q(u(x)-u(y)))}{|h|}(\phi(x)-\phi(y)) d \mu\smallskip\\=\int _{B_R}(f(x+h)-f(x))\phi\,dx,
\end{array}
\end{equation}
 $\forall\phi \in W^{1, p}(B_R)$ vanishing outside $B_{\frac{R+r}{2}}$. We have denoted $J_\beta(t)=|t|^{\beta-2}t, \beta >1, t\in \mathbb{R}$ and $d\mu=\frac{dxdy}{|x-y|^{n+qs}}$. Let $\eta$ be a nonnegative Lipschitz cut-off function such that
$
    \eta \equiv 1 \text { on } B_r, \eta \equiv 0 \text { on } \mathbb{R}^n\backslash B_{\frac{R+r}{2}}, |\nabla \eta| \leq \frac{C}{R-r}=\frac{C}{4 h_0},
$
for some constant $C=C(n)>0$. Suppose $\alpha \geq 1, \theta>0$ and testing \cref{hol3} with
$
    \phi=J_{\alpha+1}\left(\frac{u_h-u}{|h|^\theta}\right) \eta^p, 0<|h|<h_0,
$
we get
\begin{equation}{\label{mainholder}}
    I+J=F
\end{equation}
where
\begin{equation*}
\begin{array}{l}
     I=\int_{B_R} \frac{(|\nabla u_h|^{p-2} \nabla u_h-|\nabla u|^{p-2} \nabla u)}{|h|^{1+\theta \alpha}} \cdot\nabla(J_{\alpha+1}(u_h-u) \eta^p) d x,\smallskip\\J  =\int_{\mathbb{R}^n} \int_{\mathbb{R}^n} \frac{(J_q(u_h(x)-u_h(y))-J_q(u(x)-u(y)))}{|h|^{1+\theta \alpha}}
\times(J_{\alpha+1}(u_h(x)-u(x)) \eta^p(x)-J_{\alpha+1}(u_h(y)-u(y)) \eta^p(y)) d \mu 
\end{array}
\end{equation*}
and
\begin{equation*}
    F=\int _{B_R}\frac{f(x+h)-f(x)}{|h|^{1+\theta\alpha}}J_{\alpha+1}({u_h-u}) \eta^p\,dx.
\end{equation*}
Proceeding similarly like [\citealp{garain2023higher}, Proposition 4.1, Step 2] one gets
\begin{equation}{\label{II}}
I  \geq c \int_{B_R}\left|\nabla\left(\frac{|u_h-u|^{\frac{\alpha-1}{p}}(u_h-u) \eta}{|h|^{\frac{1+\theta \alpha}{p}}}\right)\right|^p d x-C \int_{B_R}|\nabla u_h|^md x 
 -C \int_{B_R}|\nabla u|^m d x-C \int_{B_R} \frac{|\delta_h u|^{\frac{\alpha m}{m-p+2}}}{|h|^{\frac{(1+\theta \alpha) m}{m-p+2}}} d x-C    
\end{equation}
for $c=c(p, \alpha)>0$ and $C=C(n, h_0, p, m, \alpha)>0$.\\
\textbf{Estimate of the nonlocal integral $J$:} First, we notice that
$
J=J_1+J_2+J_3,   $
where
\begin{equation*}
    \begin{array}{l}
         J_1  =\int_{B_R} \int_{B_R}\frac{(J_q(u_h(x)-u_h(y))-J_q(u(x)-u(y)))}{|h|^{1+\theta \alpha}}
(J_{\alpha+1}(u_h(x)-u(x)) \eta^p(x)-J_{\alpha+1}(u_h(y)-u(y)) \eta^p(y)) d \mu,\smallskip \\
J_2 =\int_{B_{\frac{R+r}{2}}} \int_{\mathbb{R}^n\backslash B_R} \frac{(J_q(u_h(x)-u_h(y))-J_q(u(x)-u(y)))}{|h|^{1+\theta \alpha}}
J_{\alpha+1}(u_h(x)-u(x)) \eta^p(x) d \mu 
    \end{array}
\end{equation*}
and
\begin{equation*}
    J_3=-\int_{\mathbb{R}^n \backslash B_R} \int_{B_{\frac{R+r}{2}}}\frac{(J_q(u_h(x)-u_h(y))-J_q(u(x)-u(y)))}{|h|^{1+\theta \alpha}}
J_{\alpha+1}(u_h(y)-u(y)) \eta^p(y) d \mu  .
\end{equation*}
Proceeding similarly like [\citealp{giacomonipq}, Proposition 3.9, Step 1], we get 
\begin{equation}{\label{rightht}}
\begin{array}{rcl}
      J_1&\geq&  c \int_{B_R} \int_{B_R}\left|\frac{J_{\frac{\alpha+q-1}{q}}(\delta_h u(x))}{|h|^{\frac{1+\theta \alpha}{q}}}-\frac{J_{\frac{\alpha+q-1}{q}}(\delta_h u(y))}{|h|^{\frac{1+\theta\alpha}{q}}}\right|^q(\eta^p(x)+\eta^p(y)) d \mu \smallskip \\&&-C \int_{B_R} \int_{B_R}\frac{|\delta_h u(x)|^{\alpha+1}+|\delta_h u(y)|^{\alpha+1}}{|h|^{1+\theta\alpha}}\left(|u_h(x)-u_h(y)|^{\frac{q-2}{2}}+|u(x)-u(y)|^{\frac{q-2}{2}}\right)^2|\eta^{\frac{p}{2}}(x)-\eta^{\frac{p}{2}}(y)|^2 d \mu.
\end{array}
\end{equation}
Noting that the first term on the right hand side is non-negative, we estimate the second term only. Using the bounds on $\eta$ and $|\nabla \eta|$, we deduce that
\begin{equation*}
J_{11}=\int_{B_R} \int_{B_R} \frac{|\delta_h u(x)|^{\alpha+1}}{|h|^{1+\theta\alpha}}|u(x)-u(y)|^{q-2}|\eta^{\frac{p}{2}}(x)-\eta^{\frac{p}{2}}(y)|^2 d \mu 
\leq \frac{C}{h_0^2} \int_{B_R} \int_{B_R} \frac{|u(x)-u(y)|^{q-2}}{|x-y|^{n+s q-2}} \frac{|\delta_h u(x)|^{\alpha+1}}{|h|^{1+\theta\alpha}} d x d y .  
\end{equation*}
For $q=2$, noticing the fact that $R<1$, we have
\begin{equation}{\label{J21}}
    J_{11}\leq \frac{C(n, s)}{h_0^2}\|u\|_{L^{\infty}(B_{R+h_0})} \int_{B_R} \frac{|\delta_h u(x)|^\alpha}{|h|^{1+\theta\alpha}} d x .
\end{equation} 
For $q>2$, take
$
0<\epsilon<\min \left\{q/2-1,1/s-1\right\} 
$ and use Young's inequality (with exponents $m /(q-2)$ and $m /(m-q+2)$ ), to have
\begin{equation}{\label{J22}}
    \begin{array}{rcl}
         J_{11}&\leq& \frac{C}{h_0^2} \int_{B_R} \int_{B_R} \frac{|u(x)-u(y)|^{q-2}}{|x-y|^{n+s q-2}} \frac{|\delta_h u(x)|^{\alpha+1}}{|h|^{1+\theta\alpha}} d x d y \smallskip\\&\leq & \frac{C}{h_0^2} \int_{B_R} \int_{B_R} \frac{|u(x)-u(y)|^{m}}{|x-y|^{n+ms \frac{q-2-\ep}{q-2}}} dx dy +
 \frac{C}{h_0^2} \int_{B_R} \int_{B_R}|x-y|^{-n+ \frac{m(2-2s-\ep s)}{m-q+2}}\frac{|\delta_h u(x)|^{\frac{(\alpha+1)m}{m-q+2}}}{|h|^{\frac{(1+\theta\alpha)m}{m-q+2}}} d x d y \smallskip\\&\leq & C[u]^m_{W^{\frac{s(q-2-\ep)}{q-2},m}(B_{R+h_0})}+C\|u\|_{L^\infty(B_{R+h_0)}}^{\frac{m}{m-q+2}}\int_{B_R}\frac{|\delta_h u(x)|^{\frac{\alpha m}{m-q+2}}}{|h|^{\frac{(1+\theta\alpha)m}{m-q+2}}} d x ,
    \end{array}
\end{equation}
where $C=C\left(n, s, q, h_0\right)>0$ is a constant (which depends inversely on $h_0$), and in the last line, we have used the fact that
$
\frac{m(2-2 s-\epsilon s)}{m-q+2}>0 .
$
For the first term, we use Sobolev embedding of \cref{embedding} to have
\begin{equation}{\label{J23}}
[u]^m_{W^{\frac{s(q-2-\ep)}{q-2},m}(B_{R+h_0})}\leq C\int_{B_{R+h_0}}|\nabla u|^m dx+C\|u\|_{L^{\infty}(B_{R+h_0})}^m .
\end{equation} A similar estimate gives 
\begin{equation}{\label{J24}}
\begin{array}{rcl}
    J_{12}&=&\int_{B_R} \int_{B_R} \frac{|\delta_h u(x)|^{\alpha+1}}{|h|^{1+\theta\alpha}}|u_h(x)-u_h(y)|^{q-2}|\eta^{\frac{p}{2}}(x)-\eta^{\frac{p}{2}}(y)|^2 d \mu\smallskip\\& \leq& C\int_{B_{R+4h_0}}|\nabla u|^m dx+C\|u\|_{L^{\infty}(B_{R+4h_0})}^m +C\|u\|_{L^\infty(B_{R+4h_0)}}^{\frac{m}{m-q+2}}\int_{B_R}\frac{|\delta_h u(x)|^{\frac{\alpha m}{m-q+2}}}{|h|^{\frac{(1+\theta\alpha)m}{m-q+2}}} d x .
\end{array}\end{equation}Combining \cref{J21,J22,J23,J24} and putting in \cref{rightht} together with Young's inequality, we obtain using \cref{connnn}
\begin{equation}{\label{J1final}}
    J_1\geq -C
    \left[\int_{B_R}\frac{|\delta_h u(x)|^{\frac{\alpha m}{m-q+2}}}{|h|^{\frac{(1+\theta\alpha)m}{m-q+2}}} d x +\int_{B_{R+4h_0}}|\nabla u|^m dx+1\right]
\end{equation}
where $C=C(n, q, s, h_0)>0$ is a constant (which depends inversely on $h_0$). 
Proceed now as [\citealp{giacomonipq}, Proposition 3.9, Step 3], to get
\begin{equation}{\label{F}}
    |J_2|+|J_3|+|F|\leq C\left(\|u\|^{q-1}_{L^{\infty}(B_{1})}+ \int_{\mathbb{R}^n\backslash B_1} \frac{|u(y)|^{q-1}}{|y|^{n+sq}} d y+\|f\|_{L^{\infty}(B_{1})}\right)\left[1+\int_{B_R}\frac{|\delta_h u(x)|^{\frac{\alpha m}{m-q+2}}}{|h|^{\frac{(1+\theta\alpha)m}{m-q+2}}} d x\right]
\end{equation}
Putting \cref{II,J1final,F} in \cref{mainholder}, we get
\begin{equation*}
    \begin{array}{l}
         \int_{B_R}\left|\nabla\left(\frac{|u_h-u|^{\frac{\alpha-1}{p}}(u_h-u) \eta}{|h|^{\frac{1+\theta \alpha}{p}}}\right)\right|^p d x\\\leq C\left( \int_{B_R}|\nabla u_h|^md x 
 + \int_{B_R}|\nabla u|^m d x+ \int_{B_R}\left| \frac{\delta_h u}{|h|^{\frac{(1+\theta \alpha)}{\alpha}}}\right|^{\frac{\alpha m}{m-p+2}}d x+1\right) \\   
\quad+C\left(  \int_{B_R}\left| \frac{\delta_h u}{|h|^{\frac{(1+\theta \alpha)}{\alpha}}}\right|^{\frac{\alpha m}{m-q+2}}d x+\int_{B_{R+4h_0}}|\nabla u|^m dx+1\right)+C\|f\|_{L^{\infty}(B_{1})}\left(1+\int_{B_R}\left| \frac{\delta_h u}{|h|^{\frac{(1+\theta \alpha)}{\alpha}}}\right|^{\frac{\alpha m}{m-q+2}}d x\right)\\\leq  C\left(\int_{B_R}|\nabla u_h|^md x 
 + \int_{B_R}|\nabla u|^m d x+\int_{B_{R+4h_0}}|\nabla u|^m dx+\int_{B_R}\left| \frac{\delta_h u}{|h|^{\frac{(1+\theta \alpha)}{\alpha}}}\right|^{\frac{\alpha m}{m-p+2}}d x+1\right),
    \end{array}
\end{equation*}
for some $C=C(n,h_0,p,q,m,s,\alpha)$. We have used Young's inequality with exponents $\frac{m-q+2}{m-p+2}$ and $\frac{m-q+2}{p-q}$ in the last line. One can now proceed as [\citealp{garain2023higher}, Proposition 4.1, Step 4 and 5] to get the desired result.
\end{proof}
\begin{lemma}{\label{semi}}{(Estimate of local seminorm)}
    Let $2 \leq q\leq p<\infty$ and $ 0<s<1$. Suppose that $u \in W_{\mathrm{loc}}^{1, p}(B_2(x_0)) \cap L_{s q}^{q-1}(\mathbb{R}^n)$ is a local weak solution of $-\Delta_p u+(-\Delta_q)^s u=f$ in $B_2(x_0)$. Let $f\in L^\infty_{\mathrm{loc}}(B_2(x_0))$ and
\begin{equation*}
    \|u\|_{L^{\infty}(B_1(x_0))} \leq 1, \quad \int_{\mathbb{R}^n\backslash B_1(x_0)} \frac{|u(y)|^{q-1}}{|y|^{n+sq}} d y \leq 1, \quad\|f\|_{L^\infty(B_1(x_0))}\leq 1.
\end{equation*}
Then 
\begin{equation*}
\int_{B_{\frac{7}{8}}\left(x_0\right)}|\nabla u|^p d x \leq C(n, p,q, s).
\end{equation*}
\end{lemma}
\begin{proof}
    Without loss of generality, we assume $x_0=0$. We only provide the proof for $w=u^{+}$, the proof of $u^{-}$ is similar. We apply \cref{caccioppoli} with $r=1, x_0=0, k=0$ and with $\psi \in$ $C_c^{\infty}(B_{\frac{8}{9}})$ such that $\psi=1$ on $B_{\frac{7}{8}}, 0 \leq \psi \leq 1$ and $|\nabla \psi| \leq C$ for some $C=C(n)>0$. By using the properties of $\psi$, this yields
\begin{equation*}
    \begin{array}{rcl}
         \int_{B_{\frac{7}{8}}}|\nabla u|^p d x & \leq& C(n, p,q,s)\left(\int_{B_1}|u|^p d x+\int_{B_1}|u|^q d x+\int_{B_1} \int_{B_1}\left(|u(x)|^p+|u(y)|^q\right)|x-y|^{q-s q-n} d x d y\right) \\
& &+C(n, p,q,s) \int_{\mathbb{R}^n\backslash B_1} \frac{|u(y)|^{q-1}}{|y|^{n+s q}} d y \int_{B_{\frac{8}{9}}}|u(x)| d x +\|f\|_{L^\infty(B_1)}\int_{B_1} w\, d x\smallskip\\
& \leq &C(n, p,q,s)(1+C(q, s)+\|f\|_{L^\infty(B_1)}).
    \end{array}
\end{equation*}
Hence, the result follows.
\end{proof}
\begin{remark}{\label{imp}}
\cref{iter} and \cref{semi} holds for the operator $-\Delta_p +A(-\Delta)^s_q$ for $0<A\leq 1$ which can be proved following the same lines.
\end{remark}
We are now ready to prove \cref{holder1}.
\subsection*{Proof of \cref{holder1}} We first observe that $u \in L_{\mathrm{loc }}^{\infty}(\Omega)$, by \cref{local boundedness}. We assume for simplicity that $x_0=0$, then we set $   \mathcal{M}_{R_0}=\|u\|_{L^{\infty}(B_{R_0})}+\operatorname{Tail}_{q-1}(u ; 0, R_0)+\|f\|_{L^\infty(B_{R_0})}+1>0 .$
It is sufficient to consider the rescaled functions
\begin{equation*}
u_{R_0}(x):=\frac{1}{\mathcal{M}_{R_0}} u(R_0x), \quad f_{R_0}(x):=\frac{1}{\mathcal{M}_{R_0}} f(R_0x), \quad \text { for } x \in B_2,    
\end{equation*}
and to show that $u_{R_0}$ satisfies the estimate
\begin{equation*}
\left[u_{R_0}\right]_{C^\sigma(B_{1 / 2})} \leq C .    
\end{equation*}
By scaling back, we would get the desired estimate. Note that by definition, the function $u_{R_0}$ is a local weak solution of $-\Delta_p u_{R_0}+R_0^{p-q s}\mathcal{M}_{R_0}^{q-p}(-\Delta_q)^s u_{R_0}=R_0^p\mathcal{M}_{R_0}^{2-p}f_{R_0}=\tilde{f}$ in $B_2$ and satisfies (by \cref{semi}; note that as $2\leq p$, so $\|\tilde{f}\|_{L^\infty(B_1)}\leq 1$)
\begin{equation*}
    \|u_{R_0}\|_{L^{\infty}(B_1)} \leq 1, \quad \int_{\mathbb{R}^n\backslash B_1} \frac{|u_{R_0}(y)|^{q-1}}{|y|^{n+sq}} d y \leq 1, \quad [u_{R_0}]_{W^{1,p}(B_{\frac{7}{8}})}\leq C(n,p,q,s).
\end{equation*}
By noting \cref{imp}, the rest of the proof follows similarly like [\citealp{garain2023higher}, Theorem 1.3].
\section{Improved H\"older regularity results}{\label{regu2}} We now relax the condition on the source term $f$ to be in $L^\infty_{\mathrm{loc}}(\Omega)$ and take $f\in L^n_{\mathrm{loc}}(\Omega)$. For a function $g\in W_{\mathrm{loc}}^{1, p}(\Omega) \cap L_{s q}^{q-1}(\mathbb{R}^n)$ and any $B_r(x_0)\subset\subset\Omega$, define 
\begin{equation*}
\begin{array}{c}
    (g)_{B_r(x_0)}=\fint_{B_r(x_0)} g \,dx ,\quad \operatorname{av}_m(g, B_r(x_0)):=\left(\fint_{B_r(x_0)}|g-(g)_{B_r}|^mdx\right)^{\frac{1}{m}},  m>0 \smallskip\\\text{ and } \quad\operatorname{snail}_\delta(r)\equiv \operatorname{snail}_\delta(g, B_r(x_0)):=r^{\frac{\delta}{q}}\left(r^s\int_{\mathbb{R}^n\backslash B_r(x_0)}\frac{|g(y)-(g)_{B_r}|^{q-1}}{|y-x_0|^{n+sq}}dy\right)^{\frac{1}{q-1}}, \delta \geq sq.
\end{array}
\end{equation*}
To prove H\"older regularity results for \cref{general}, we adopt the perturbation techniques used in \cite{MG}. The set-up used in \cite{MG} needs the solution to be in $W^{s,q}(\mathbb{R}^n)$, while we will prove the results for local solutions only. The following important lemma will be the most crucial for our case.
\begin{lemma}{\label{mgiter}}
    Let $B_t(x_0) \subset B_r(x_0)$ be two concentric balls, $q\geq 2, \delta \geq sq$ and $g \in W_{\mathrm{loc}}^{1, p}(\Omega) \cap L_{s q}^{q-1}(\mathbb{R}^n)$.\\
(a) There exists $c \equiv c(n, s, q)>0$ such that, whenever $0<t<r \leq 1$, it holds 
\begin{equation}{\label{snail}}
    \begin{array}{l}
         \operatorname{snail}_\delta(g, B_t(x_0))\leq c\left(\frac{t}{r}\right)^{\delta / q} \operatorname{snail}_\delta(g, B_{r}(x_0)) +c t^{\delta / q-s}\left( \int_t^{r}\left(\frac{t}{v}\right)^s \operatorname{av}_q(g, B_v(x_0)) \frac{d v}{v}+\left(\frac{t}{r}\right)^s \operatorname{av}_q(g, B_{r}(x_0))\right).
    \end{array}
\end{equation}
(b) With $m \geq 1$, if $v>0$ and $\theta \in(0,1)$ are such that $\theta r \leq v \leq r$, then
\begin{equation}{\label{av}}
    \operatorname{av}_{{m}}(g, B_v(x_0)) \leq 2 \theta^{-n / {m}} \operatorname{av}_{{m}}(g, B_{r}(x_0)).
\end{equation}
\end{lemma}
\begin{proof}
    In what follows, all the balls will be centred at $x_0$. Let us first recall the standard property
\begin{equation}{\label{av2}}
    \left(\fint_{B_r}|g-(g)_{B_r}|^m d x\right)^{1 / m} \leq 2\left(\fint_{B_r}|g-k|^md x\right)^{1 / m}
\end{equation}
that holds $\forall k \in \mathbb{R}$ and $m \geq 1$; from this \cref{av} follows immediately. To prove \cref{snail}, let $B_t \subset B_{r}$, we then use
$
    d \lambda_{x_0}(\mathbb{R}^n \backslash B_t)=c(n,s,q) t^{-s q}, d \lambda_{x_0}(x):=\frac{d x}{|x-x_0|^{n+sq}}
$
and split
\begin{equation*}
    \begin{array}{rcl}
\operatorname{snail}_\delta(t)& \leq & c\left(\frac{t}{r}\right)^{\delta / q}r^{\delta / q} \left(t^s\int_{\mathbb{R}^n\backslash B_r}\frac{|g(y)-(g)_{B_t}|^{q-1}}{|y-x_0|^{n+sq}}dy\right)^{\frac{1}{q-1}}+ct^{\delta/q}\left(t^s\int_{B_r\backslash B_t}\frac{|g(y)-(g)_{B_t}|^{q-1}}{|y-x_0|^{n+sq}}dy\right)^{\frac{1}{q-1}}
\\&\leq&c\left(\frac{t}{r}\right)^{\delta / q} \operatorname{snail}_\delta(r) +c t^{\delta /q-s}\left(\frac{t}{r}\right)^s|(g)_{B_t}-(g)_{B_r}|+ct^{\delta/q}\left(t^s\int_{B_r\backslash B_t}\frac{|g(y)-(g)_{B_t}|^{q-1}}{|y-x_0|^{n+sq}}dy\right)^{\frac{1}{q-1}}\\
&= & c\left(\frac{t}{r}\right)^{\delta / q} \operatorname{snail}_\delta(r)+c T_1+c T_2    ,     
    \end{array}
\end{equation*}
where $c \equiv c(n, s, q)$. If $r/4 \leq t < r$, standard manipulations based on \cref{av} gives $T_1 + T_2 \leq c t^{\delta / q-s}\left(\frac{t}{r}\right)^s \operatorname{av}_q(g, B_{r})$ with $c \equiv c(n,s,q)$. We can therefore assume that $t < r/4$. Then there exists
$\lambda \in(1/4, 1/2)$ and $\kappa\in \mathbb{N}$, $\kappa\geq 2$ so that $t = \lambda^\kappa r$. We now use triangle and Hölder’s
inequalities, and estimate, using \cref{av,av2} repeatedly
\begin{equation}{\label{T1}}
    \begin{array}{rcl}
         T_1 & \leq& t^{\delta / q-s}\left(\frac{t}{r}\right)^s\left(|(g)_{B_{\lambda r}}-(g)_{B_r}|+|(g)_{B_{\lambda r}}-(g)_{B_{\lambda^\kappa r}}|\right) \\
& \leq &c t^{\delta / q-s}\left(\frac{t}{r}\right)^s \left(\operatorname{av}_q(r)+\sum_{i=1}^{\kappa-1}|(g)_{B_{\lambda^i r}}-(g)_{B_{\lambda^{i+1} r}}| \right)\\
& \leq &c t^{\delta / q-s}\left(\frac{t}{r}\right)^s \left(\operatorname{av}_q(r)+\sum_{i=1}^{\kappa-1}\left(\fint_{B_{\lambda^i r}}|g(x)-(g)_{B_{\lambda^i r}}|^q d x\right)^{1 / q} \right)\\
& \leq& c t^{\delta / q-s}\left(\frac{t}{r}\right)^s \left(\operatorname{av}_q(r)+ \sum_{i=1}^{\kappa-1} \int_{\lambda^i r}^{\lambda^{i-1} r} \operatorname{av}_q(\lambda^i r) \frac{d v}{v}\right) \\
& \leq& c t^{\delta / q-s}\left(\frac{t}{r}\right)^s \left(\operatorname{av}_q(r)+\sum_{i=1}^{\kappa-1} \int_{\lambda^i r}^{\lambda^{i-1} r} \operatorname{av}_q(v) \frac{d v}{v} \right)\\&\leq &\leq c t^{\delta / q-s}\left(\frac{t}{r}\right)^s \left(\operatorname{av}_q(r)+\int_t^{r} \operatorname{av}_q(v) \frac{d v}{v}\right),
    \end{array}
\end{equation}
with $c \equiv c(n, s, q)$. For $T_2$, we rewrite $r=\lambda^{-\kappa} t$ and estimate, by telescoping and Jensen's inequality
\begin{equation}{\label{jensen}}
    \left(\fint_{B_{\lambda^{-i}t}}|g(x)-(g)_{B_t}|^q d x\right)^{1 / q} \leq 2^{n / q+1} \sum_{m=0}^i \operatorname{av}_q(\lambda^{-m} t),
\end{equation}
for $0 \leq i \leq k$. Then, using \cref{av2}, \cref{jensen} and the discrete Fubini theorem, 
we get by noting $0<\lambda<1$ and $q-1\geq 1$,
\begin{equation}{\label{T2}}
    \begin{array}{rcl}
      T_2 &\leq & c t^{\delta / q}\left(t^s\sum_{i=0}^{\kappa-1} \lambda^{i s q}(\lambda^{-i} t)^{-n}t^{-sq} \int_{B_{\lambda^{-i-1}t} \backslash B_{\lambda^{-i}t}}|g(x)-(g)_{B_t}|^{q-1} d x\right)^{\frac{1}{q-1}} \\
& \leq&c t^{\delta / q-s}\left(\sum_{i=0}^{\kappa-1} \lambda^{i s q} \fint_{B_{\lambda^{-i-1}t} }|g(x)-(g)_{B_t}|^{q-1} d x\right)^{\frac{1}{q-1}} \\& \leq&c t^{\delta / q-s}\sum_{i=0}^{\kappa-1} \left(\lambda^{i s q} \fint_{B_{\lambda^{-i-1}t} }|g(x)-(g)_{B_t}|^{q-1} d x\right)^{\frac{1}{q-1}}\\& \leq&c t^{\delta / q-s}\sum_{i=0}^{\kappa-1} \lambda^{\frac{isq}{q-1}} \left(\fint_{B_{\lambda^{-i-1}t} }|g(x)-(g)_{B_t}|^{q} d x\right)^{\frac{1}{q}}\\
& \leq &c t^{\delta / q-s} \sum_{i=0}^\kappa \lambda^{i s} \sum_{m=0}^i \operatorname{av}_q(\lambda^{-m} t)\\& =& c t^{\delta / q-s} \sum_{m=0}^\kappa \operatorname{av}_q(\lambda^{-m} t) \sum_{i=m}^\kappa \lambda^{i s}
\\
& \leq& c t^{\delta / q-s} \sum_{m=0}^\kappa \lambda^{ms} \operatorname{av}_q(\lambda^{-m} t) \\
& \leq &c t^{\delta / q-s} \sum_{m=0}^{\kappa-1} \int_{\lambda^{-m}t}^{\lambda^{-m-1} t} \lambda^{ms} \operatorname{av}_q(v) \frac{d v}{v}+c t^{\delta / q-s}\left(\frac{t}{r}\right)^s \operatorname{av}_q(r) \\
& \leq& c t^{\delta / q-s} \sum_{m=0}^{\kappa-1} \int_{\lambda^{-m}t}^{\lambda^{-m-1} t}\left(\frac{t}{v}\right)^s \operatorname{av}_q(v) \frac{d v}{v}+c t^{\delta / q-s}\left(\frac{t}{r}\right)^s \operatorname{av}_q(r) \\
& \leq& c t^{\delta / q-s} \int_t^{r}\left(\frac{t}{v}\right)^s \operatorname{av}_q(v) \frac{d v}{v}+c t^{\delta / q-s}\left(\frac{t}{r}\right)^s \operatorname{av}_q(r),
   \end{array}
\end{equation}
for $c \equiv c(n, s, q)$. Merging the \cref{T1,T2} we obtain \cref{snail}.
\end{proof}
From hereafter, we take $u$ as a local weak solution to $\mathcal{L}u=f$ in $\Omega$. This assumption holds until the end of H\"older regularity results. Before Proceeding further, we note \cref{caccremark} and take $k=(u)_{B_r}$ in \cref{caccioppoli} and set $\psi\in C_c^{\infty}(B_r)$ such that $0\leq \psi\leq 1$ in $B_r$, $\psi\equiv 1$ in $B_{r/2}$, $\psi\equiv 0$ outside $B_{3r/4}$ and $|\nabla \psi|\leq {C}/{r} $ for some constant $C>0$. Then \cref{cacc} gives
\begin{equation}{\label{caccimproved}}
     \begin{array}{l}
         \fint_{B_{r}} \psi^p|\nabla u|^p d x+\fint_{B_{r/2}} \int_{B_{r/2}}\frac{|u(x)-u(y)|^q }{|x-y|^{n+qs}}dxdy \\
\leq C\left(r^{-p}\fint_{B_r} |u-(u)_{B_r}|^p d x+\fint_{B_r} \int_{B_r} \max \{|u(x)-(u)_{B_r}|, |u(y)-(u)_{B_r}|\}^q|\psi^{p/q}(x)-\psi^{p/q}(y)|^q d \mu\right. \smallskip\\
\left.\quad+\fint_{B_r}\left(\int_{\mathbb{R}^n \backslash B_r} \frac{|u(y)-(u)_{B_r}|^{q-1}}{|x-y|^{n+ sq}} d y \right) |u(x)-(u)_{B_r}|\psi^p d x\right.\smallskip\\\left.\quad+r^{-sq}\fint_{B_r}|u-(u)_{B_r}|^q dx+\fint_{B_r}|f||u-(u)_{B_r}|\psi^p dx\right)\smallskip\\= I+II+III+IV+V,
    \end{array}
\end{equation}
for some $C\equiv C(n,p,q,s,\Lambda)$.\smallskip\\
\textbf{Estimate of $II$:} Noting the properties of $\psi$, one gets
\begin{equation}{\label{2}}
    \begin{array}{l}
         \quad\fint_{B_r} \int_{B_r} \max \{|u(x)-(u)_{B_r}|, |u(y)-(u)_{B_r}|\}^q|\psi^{p/q}(x)-\psi^{p/q}(y)|^q d \mu\\\leq Cr^{-q}\fint_{B_r}\int_{B_r}\frac{\max \{|u(x)-(u)_{B_r}|, |u(y)-(u)_{B_r}|\}^q}{|x-y|^{n+sq-q}}dxdy\leq Cr^{-sq}\fint_{B_r}|u-(u)_{B_r}|^q dx,
    \end{array}
\end{equation}
for some $C\equiv C(n,p,q,s)$.\\
\textbf{Estimate of $III$:} Note that $x\in B_{3r/4}$
and $y\in \mathbb{R}^n\backslash B_r$ implies $1\leq \frac{|y-x_0|}{|x-y|}\leq 4$ and then recalling $\psi$ is supported in $B_{3r/4}$ we get using Young's inequality
\begin{equation}{\label{3}}
    \begin{array}{l}
         \quad\fint_{B_r}\left(\int_{\mathbb{R}^n \backslash B_r} \frac{|u(y)-(u)_{B_r}|^{q-1}}{|x-y|^{n+ sq}} d y \right) |u(x)-(u)_{B_r}|\psi^p d x\smallskip\\\leq C\int_{\mathbb{R}^n \backslash B_r} \frac{|u(y)-(u)_{B_r}|^{q-1}}{|y-x_0|^{n+ sq}} d y\fint_{B_r}|u(x)-(u)_{B_r}|d x
       \smallskip \\ \leq  Cr^s\int_{\mathbb{R}^n \backslash B_r} \frac{|u(y)-(u)_{B_r}|^{q-1}}{|y-x_0|^{n+ sq}} d y\times r^{-s}\left(\fint_{B_r}|u(x)-(u)_{B_r}|^qd x\right)^{1/q}\smallskip\\\leq Cr^{-sq}\fint_{B_r}|u-(u)_{B_r}|^q dx+C\left(r^s\int_{\mathbb{R}^n \backslash B_r} \frac{|u(y)-(u)_{B_r}|^{q-1}}{|y-x_0|^{n+ sq}} d y\right)^{\frac{q}{q-1}},
    \end{array}
\end{equation}
for some $C\equiv C(n,p,q,s)$.\\
\textbf{Estimate of $V$:}
As $f\in L^n_{\mathrm{loc}}(\Omega)$ and $0\leq \psi\leq 1$, so
\begin{equation}{\label{4}}
    \begin{array}{rcl}
         \fint_{B_r}|f||u-(u)_{B_r}|\psi^p dx&\leq &C\|f\|_{L^n(B_r)}\left(\fint_{B_r}|\psi ^p(u-(u)_{B_r})|^{p^*}dx\right)^{1/p^*}\\&\leq &C\|f\|_{L^n(B_r)}\left(\fint_{B_r}|\nabla(\psi ^p(u-(u)_{B_r}))|^{p}dx\right)^{1/p}\\&\leq &C\|f\|_{L^n(B_r)}\left[\left(\fint_{B_r}(p\psi^{p-1}|\nabla\psi||u-(u)_{B_r}|)^{p}dx\right)^{1/p}+\left(\fint_{B_r}\psi^{p^2}|\nabla u|^{p}dx\right)^{1/p}\right]\\&\leq &C\|f\|_{L^n(B_r)}r^{-1}\left(\fint_{B_r}|u-(u)_{B_r}|^{p}dx\right)^{1/p}+C\|f\|_{L^n(B_r)}\left(\fint_{B_r}\psi^p|\nabla u|^{p}dx\right)^{1/p}\\&\leq & C\|f\|_{L^n(B_r)}^{\frac{p}{p-1}}+Cr^{-p}\fint_{B_r}|u-(u)_{B_r}|^{p}dx+C\|f\|_{L^n(B_r)}^{\frac{p}{p-1}}+\frac{1}{2}\fint_{B_r}\psi^p|\nabla u|^{p}dx,
    \end{array}
\end{equation}
for some $C\equiv C(n,p)$. Finally for $r<1$, it holds 
\begin{equation}{\label{ccp}}
    r^{-sq}[\operatorname{av}_q(u,B_r(x_0))]^q\leq cr^{(1-s)q}(r^{-p}[\operatorname{av}_p(u,B_r(x_0))]^p)^{q/p}\leq c(r^{-p}[\operatorname{av}_p(u,B_r(x_0))]^p+1).
\end{equation}
In view of \cref{ccp}, merging \cref{2,3,4} and using in \cref{caccimproved}, we get 
\begin{equation}{\label{caccimprovedfinal}}
     \begin{array}{l}
         \fint_{B_{r/2}(x_0)}|\nabla u|^p d x+\fint_{B_{r/2}(x_0)} \int_{B_{r/2}(x_0)}\frac{|u(x)-u(y)|^q }{|x-y|^{n+qs}}dxdy \smallskip\\
\leq C\left(r^{-p}[\operatorname{av}_p(u,B_r(x_0))]^p+r^{-\delta}[\operatorname{snail}_\delta(u, B_r(x_0))]^q+\|f\|_{L^n(B_r)}^{\frac{p}{p-1}}+1\right).
    \end{array}
\end{equation}
We now define $h \in u+W_0^{1 , p}(B_{r / 4}(x_0))$ as the (unique) solution to
\begin{equation}{\label{h}}
   -\operatorname{div} A(x_0,\nabla h(x))=0 \text { in } B_{r / 4}(x_0), \quad h=u \text { on } \partial B_{r / 4}(x_0).
\end{equation}
The function $h$ solves the Euler-Lagrange equation
\begin{equation}{\label{heu}}
    \int_{B_{r/ 4}(x_0)} A(x_0,\nabla h) \cdot \nabla \varphi \,d x=0 \quad \text { for every } \varphi \in W_0^{1, p}(B_{r / 4}) .
\end{equation}
The following lemma includes some crucial properties satisfied by $h$ that we will eventually pass to $u$. 
\begin{lemma}{\label{proph}} Let $h$ be a solution to \cref{h}. Then there exists constants $\alpha_0\in(0,1)$, and $C,c>0$ such that
\begin{equation}{\label{h1}}
    \fint_{B_{r / 4}(x_0)}(|\nabla h|^2+\mu^2)^{p / 2} d x \leq C  \fint_{B_{r / 4}(x_0)}(|\nabla u|^2+\mu^2)^{p / 2} d x,
\end{equation}
\begin{equation}{\label{h2}}
    \|h\|_{L^{\infty}(B_{r/ 4}(x_0))} \leq\|u\|_{L^{\infty}(B_{r/ 4}(x_0))}, \quad \underset{B_{r/ 4}(x_0)}{\operatorname{osc}} h\leq  \underset{B_{r/ 4}(x_0)}{\operatorname{osc}} u,
\end{equation}
and 
\begin{equation}{\label{h3}}
    \|\nabla h\|_{L^{\infty}(B_{r/ 8}(x_0))}^p \leq c \fint_{B_{r/ 4}(x_0)}(|\nabla h|^2+\mu^2)^{p / 2} dx, \quad \underset{B_{t}(x_0)}{\operatorname{osc}}\nabla h\leq c\left(\frac{t}{r}\right)^{\alpha_0}\left( \fint_{B_{r / 4}(x_0)}(|\nabla h|^2+\mu^2)^{p / 2} d x\right)^{1/p}
\end{equation}
holds for all $t\in (0,r/8]$.
\end{lemma}
\begin{proof}
    We only prove \cref{h1}, as \cref{h2} follows from the standard maximum principle of the elliptic operator $h\mapsto A(x_0, \nabla h)$ and \cref{h3} follows from [\citealp{manfredi}, Theorem 1, Theorem 2]. To prove \cref{h1}, we test \cref{heu} with $u-h \in W_0^{1,p}(B_{r/4}({x}_0))$ and use \cref{growth1} to get
\begin{equation*}
    \int_{B_{r/4}({x}_0)} A(x_0,\nabla h) \cdot \nabla h\, d x=\int_{B_{r/4}({x}_0)} A(x_0,\nabla h) \cdot \nabla u \,d x\leq \Lambda \int_{B_{r/4}({x}_0)}(|\nabla h|^2+\mu^2)^{\frac{p-2}{2}}|\nabla h || \nabla u| d x .
\end{equation*}
Using again hypothesis \cref{growth1}, we see that $A({x}_{0}, z) \cdot z \geq \min \left\{1, \frac{1}{p-1}\right\} \Lambda^{-1}(|z|^2+\mu^2)^{\frac{p-2}{2}}|z|^2$ for every $z \in \mathbb{R}^n$, so that
\begin{equation*}
   \int_{B_{r/4}({x}_0)} A(x_0,\nabla h) \cdot \nabla h\, d x  \geq \frac{{\Lambda}^{-1}}{p} \int_{B_{r/4}({x}_0)}(|\nabla h|^2+\mu^2)^{\frac{p-2}{2}}|\nabla h |^2d x .
\end{equation*}
Thus,
\begin{equation}{\label{h5}}
    \int_{B_{r/4}({x}_0)}(|\nabla h|^2+\mu^2)^{\frac{p-2}{2}}|\nabla h |^2d x \leq p\Lambda^2\int_{B_{r/4}({x}_0)}(|\nabla h|^2+\mu^2)^{\frac{p-2}{2}}|\nabla h || \nabla u| d x .
\end{equation}
Now, if $p \geq 2$ this yields
\begin{equation*}
    \int_{B_{r/4}({x}_0)}|\nabla h|^p dx\leq p {\Lambda}^2 \int_{B_{r/4}({x}_0)}(|\nabla h|^2+\mu^2)^{\frac{p-1}{2}}(|\nabla u|^2+\mu^2)^{\frac{1}{2}} d x.
\end{equation*}
which immediately gives \cref{h1} after an application of Holder's inequality. If $p \in(1,2)$, we also exploit Hölder's inequality along with the fact that\begin{equation*}
    \frac{t^{\frac{p}{p-1}}}{(t^2+\mu^2)^{\frac{(2-p) p}{2(p-1)}}} \leq(t^2+\mu^2)^{\frac{p-2}{2}} t^2, \quad \text { for all } t \geq 0
\end{equation*}
to deduce from \cref{h5} that
\begin{equation*}
    \begin{array}{rcl}
      \int_{B_{r/4}({x}_0)}(|\nabla h|^2+\mu^2)^{\frac{p-2}{2}}|\nabla h |^2d x    &  \leq &p {\Lambda}^2\left(\int_{B_{r/4}({x}_0)} \frac{|\nabla h|^{\frac{p}{p-1}}}{(|\nabla h|^2+\mu^2)^{\frac{(2-p) p}{2(p-1)}}} d x\right)^{\frac{p-1}{p}}\left(\int_{B_{r/4}({x}_0)}|\nabla u|^p d x\right)^{\frac{1}{p}} 
        \smallskip\\&\leq &p {\Lambda}^2\left( \int_{B_{r/4}({x}_0)}(|\nabla h|^2+\mu^2)^{\frac{p-2}{2}}|\nabla h |^2d x \right)^{\frac{p-1}{p}}\left(\int_{B_{r/4}({x}_0)}|\nabla u|^p d x\right)^{\frac{1}{p}} .
    \end{array}
\end{equation*}
This gives
\begin{equation*}
     \int_{B_{r/4}({x}_0)}(|\nabla h|^2+\mu^2)^{\frac{p-2}{2}}|\nabla h |^2d x   \leq2^p\Lambda^{2p}\int_{B_{r/4}({x}_0)}|\nabla u|^p d x\leq2^p\Lambda^{2p}\int_{B_{r/4}({x}_0)}(|\nabla u|^2+\mu^2)^{p/2} d x
\end{equation*}
which, together with the trivial estimate
\begin{equation*}
    \int_{B_{r/4}({x}_0)}(|\nabla h|^2+\mu^2)^{\frac{p-2}{2}}\mu^2d x   \leq\int_{B_{r/4}({x}_0)}\mu^p d x\leq\int_{B_{r/4}({x}_0)}(|\nabla u|^2+\mu^2)^{p/2} d x
\end{equation*}
readily yields \cref{h1}.
\end{proof}
Next, we give the difference estimate to pass the regularity of $h$ to $u$.
\begin{lemma}{\label{differenec}}
    Let $h \in u+W_0^{1, p}(B_{r / 4}(x_0)), 0<r<1$ be the solution of \cref{h}. There exists $\sigma_0 \equiv \sigma_0(p, s, q,\bar{\alpha}) \in$ $(0,1)$ such that
\begin{equation}{\label{differenceestimate}}
    \fint_{B_{r / 4}}|u-h|^p dx \leq c r^{\theta\sigma_0}\left[\fint_{B_r}|u-(u)_{B_r}|^pdx+r^\delta\left(r^s\int_{\mathbb{R}^n \backslash B_r} \frac{|u(y)-(u)_{B_r}|^{q-1}}{|y-x_0|^{n+ sq}} d y\right)^{\frac{q}{q-1}}+r^{p-\theta}(\|f\|_{L^n(B_r)}^{\frac{p}{p-1}}+1)\right],
\end{equation}
holds for every $\theta\in (0,1)$ where $\delta \in(sq,p)$, $c\equiv c(n,p,q,s,\Lambda)$.
\end{lemma}
\begin{proof}
Consider the vector field $V_\mu: \mathbb{R}^n \rightarrow \mathbb{R}^n$, defined by
$
    V_\mu(z):=(|z|^2+\mu^2)^{(p-2) / 4} z,
$
whenever $z \in \mathbb{R}^n$, where $p \in(1, \infty)$ and $\mu \in[0,1]$. It follows that
\begin{equation}{\label{g1}}
|V_\mu(z_1)-V_\mu(z_2)| \approx(|z_1|^2+|z_2|^2+\mu^2)^{(p-2) / 4}|z_1-z_2|,    
\end{equation}
where the equivalence holds up to constants depending only on $n, p$. A standard consequence of \cref{growth1}$_3$ is the following strict monotonicity inequality:
\begin{equation}{\label{g2}}
    |V_\mu(z_1)-V_\mu(z_2)|^2 \leq c(A(x,z_2)-A(x,z_1)) \cdot(z_2-z_1)
\end{equation}
holds for all $x\in \Omega$ whenever $z_1, z_2 \in \mathbb{R}^n$, where $c \equiv c(n, p, \Lambda)$. The last two inequalities are, however, based on the following one
\begin{equation*}
\int_0^1(|z_1+\lambda(z_2-z_1)|^2+\mu^2)^{t / 2} d\lambda \approx_{n, t}(|z_1|^2+|z_2|^2+\mu^2)^{t / 2}    
\end{equation*}
that holds whenever $t>-1$ and $z_1, z_2 \in \mathbb{R}^n$ are such that $|z_1|+|z_2|+\mu>0$. Noting [\citealp{MG}, Section 2.5], we now choose $u-h \in W_0^{1,p}(B_{r/4}({x}_0))$ as a test function in the weak formulation of both $u$ and $v$ and use \cref{g2} to get
\begin{equation}{\label{diff1}}
    \begin{array}{rcl}
     \fint_{B_{r / 4}}|V_\mu(\nabla u)-V_\mu(\nabla h)|^2 dx &\leq &c\fint_{B_{r/ 4}}(A(x_0,\nabla u)-A(x_0,\nabla h)) \cdot(\nabla u-\nabla h)dx\\&
     =&c \fint_{B_{r / 4}} A(x_0,\nabla u)\cdot (\nabla u-\nabla h)dx\\&=&c\fint_{B_{r / 4}} (A(x_0,\nabla u)-A(x,\nabla u))\cdot (\nabla u-\nabla h)dx+ c \fint_{B_{r / 4}} f (u-h)d x\\&&-c \fint_{B_{r/ 2}} \int_{B_{r / 2}}\frac{|u(x)-u(y)|^{q-2}(u(x)-u(y))((u-h)(x)-(u-h)(y))}{|x-y|^{n+sq}}dxdy \\
&&-2 c \fint_{B_{r / 2}}\int_{\mathbb{R}^n \backslash B_{r/ 2}} \frac{|u(x)-u(y)|^{q-2}(u(x)-u(y)) (u-h)(x)}{|x-y|^{n+sq}}dydx \\
&=&(\mathrm{I})+(\mathrm{II})+(\mathrm{III})+(\mathrm{IV}),
    \end{array}
\end{equation}where $c \equiv c(n, p, \Lambda)$. In view of [\citealp{MG}, Lemma 4.2], it is enough to estimate $(\mathrm{I})$ and $(\mathrm{IV})$. Using \cref{growth1}$_2$, we get 
\begin{equation}{\label{di1}}
    \begin{array}{rcl}
       (\mathrm{I})&=&  \fint_{B_{r / 4}} (A(x_0,\nabla u)-A(x,\nabla u))\cdot (\nabla u-\nabla h)dx\leq C r^{\bar{\alpha}}\fint_{B_{r / 4}} (|\nabla u|^2+\mu^2)^{\frac{p-1}{2}}|\nabla u-\nabla h|dx\\&\leq& Cr^{\bar{\alpha}}\left(\fint_{B_{r / 4}} (|\nabla u|^2+\mu^2)^{\frac{p}{2}}dx\right)^{\frac{p-1}{p}}\left(\fint_{B_{r / 4}}|\nabla u-\nabla h|^p\right)^{\frac{1}{p}}\\&\stackrel{\cref{caccimprovedfinal},\cref{h1}}{\leq}& Cr^{\bar{\alpha}}\left(r^{-p}[\operatorname{av}_p(u,B_r(x_0))]^p+r^{-\delta}[\operatorname{snail}_\delta(u, B_r(x_0))]^q+\|f\|_{L^n(B_r)}^{\frac{p}{p-1}}+1\right),
    \end{array}
\end{equation}
for some $C\equiv C(n,p,q,s,\Lambda)$. For $(\mathrm{IV})$ we note that we can replace $u$ by $u - (u)_{B_{r/2}}$ and use that $x \in  B_{r/4}, y \in\mathbb{R}^n\backslash B_{r/2}$ imply $|y - x_0|/|x - y| \leq  2$. Recalling that $u-h$ is supported in $B_{r/4}$, we then have using \cref{av,av2,snail} 
\begin{equation}{\label{di}}
    \begin{array}{l}
         (\mathrm{IV})\leq  C\fint_{B_{r / 2}}\int_{\mathbb{R}^n \backslash B_{r/ 2}}  \frac{\operatorname{max}\{|u(x)-(u)_{B_{r/2}}|^{q-1},|u(y)-(u)_{B_{r/2}}|^{q-1}\}|u(x)-h(x)|}{|x-y|^{n+sq}}dydx \\\leq C\fint_{B_{r / 2}}\int_{\mathbb{R}^n \backslash B_{r/ 2}}  \frac{\operatorname{max}\{|u(x)-(u)_{B_{r/2}}|^{q-1},|u(y)-(u)_{B_{r/2}}|^{q-1}\}|u(x)-h(x)|}{|y-x_0|^{n+sq}}dydx,\\
          \leq Cr^{-sq}\fint_{B_{r / 2}}|u(x)-(u)_{B_{r/2}}|^{q-1}|u(x)-h(x)|dx\\+C\int_{\mathbb{R}^n \backslash B_{r/ 2}}  \frac{|u(y)-(u)_{B_{r/2}}|^{q-1}}{|y-x_0|^{n+sq}}dy\times\fint_{B_{r / 2}}|u(x)-h(x)|dx\\\leq C\left[ r^{-sq}\left(\fint_{B_{r / 2}}|u(x)-(u)_{B_{r/2}}|^{q}\right)^{\frac{q-1}{q}}+\int_{\mathbb{R}^n \backslash B_{r/ 2}}  \frac{|u(y)-(u)_{B_{r/2}}|^{q-1}}{|y-x_0|^{n+sq}}dy\right]\\\times\left(\fint_{B_{r / 2}}|u(x)-h(x)|^qdx\right)^{1/q}\\\leq C\left[ \left(r^{-sq}\fint_{B_{r / 2}}|u(x)-(u)_{B_{r/2}}|^{q}\right)^{\frac{q-1}{q}}+r^s\int_{\mathbb{R}^n \backslash B_{r/ 2}}  \frac{|u(y)-(u)_{B_{r/2}}|^{q-1}}{|y-x_0|^{n+sq}}dy\right]\\\times\left(r^{-sq}\fint_{B_{r / 2}}|u(x)-h(x)|^qdx\right)^{1/q}\\\stackrel{\cref{ccp}}{\leq} C\left([r^{-p}\operatorname{av}_p(u,B_r(x_0))]^p+r^{-\delta}[\operatorname{snail}_\delta(u, B_r(x_0))]^q+1\right)^{\frac{q-1}{q}}\left(r^{-sq}\fint_{B_{r / 2}}|u(x)-h(x)|^qdx\right)^{1/q},
    \end{array}
\end{equation}
for some $C\equiv C(n,p,q,s,\Lambda)$. Now by Sobolev embedding \cref{sobolev embedding 2}, one has

\begin{equation*}
    \begin{array}{c}
         \left(r^{-sq}\fint_{B_{r / 2}}|u-h|^qdx\right)^{1/q}\leq Cr^{(1-s)}\left(r^{-p}\fint_{B_{r / 2}}|u-h|^pdx\right)^{1/p}\leq Cr^{(1-s)}\left(\fint_{B_{r / 2}}|\nabla u-\nabla h|^pdx\right)^{1/p}\\\stackrel{\cref{caccimprovedfinal},\cref{h1}}{\leq}Cr^{(1-s)}\left(r^{-p}[\operatorname{av}_p(u,B_r(x_0))]^p+r^{-\delta}[\operatorname{snail}_\delta(u, B_r(x_0))]^q+\|f\|_{L^n(B_r)}^{\frac{p}{p-1}}+1\right)^{1/p},
    \end{array}
\end{equation*}
for some $C\equiv C(n,p,q,s,\Lambda)$. This along with \cref{di} gives
\begin{equation}{\label{di2}}
     (\mathrm{IV})\leq C r^{1-s}\left(r^{-p}[\operatorname{av}_p(u,B_r(x_0))]^p+r^{-\delta}[\operatorname{snail}_\delta(u, B_r(x_0))]^q+\|f\|_{L^n(B_r)}^{\frac{p}{p-1}}+1\right)^{1-1/q+1/p},
\end{equation}
for some $C\equiv C(n,p,q,s,\Lambda)$. Denoting $\operatorname{ccp}(r)=r^{-p}[\operatorname{av}_p(u,B_r(x_0))]^p+r^{-\delta}[\operatorname{snail}_\delta(u, B_r(x_0))]^q+\|f\|_{L^n(B_r)}^{\frac{p}{p-1}}+1$ and then using \cref{di1}, \cref{di2} along with the estimates of other terms from [\citealp{MG}, Lemma 4.2], we get from \cref{diff1}
\begin{equation}{\label{diff2}}
    \begin{array}{l}
     \fint_{B_{r / 4}}|V_\mu(\nabla u)-V_\mu(\nabla h)|^2 dx\leq C r^{\bar{\alpha}\operatorname{min}\{1,p/2\}} \operatorname{ccp}(r)+Cr^{1-s}\operatorname{ccp}(r)^{1-1/q+1/p}+C\|f\|_{L^n(B_r)}\operatorname{ccp}(r)^{1/p},
     \end{array}
     \end{equation}
     for some $C\equiv C(n,p,q,s,\Lambda)$. As $r<1$ and $p>\delta$, so for each $\theta\in(0,1)$,
     \begin{equation*}
         r^p\operatorname{ccp}(r)\leq \left[\fint_{B_r}|u-(u)_{B_r}|^pdx+r^\delta\left(r^s\int_{\mathbb{R}^n \backslash B_r} \frac{|u(y)-(u)_{B_r}|^{q-1}}{|y-x_0|^{n+ sq}} d y\right)^{\frac{q}{q-1}}+r^{p-\theta}(\|f\|_{L^n(B_r)}^{\frac{p}{p-1}}+1)\right].
     \end{equation*}
     Using this, one can now estimate from \cref{diff2} like [\citealp{MG}, Lemma 4.2, Pages 29-30], to get \begin{equation*}
    \begin{array}{l}
     \fint_{B_{r / 4}}|u-h|^p dx\leq C r^{\bar{\alpha}\operatorname{min}\{1,p/2\}} r^p\operatorname{ccp}(r)\\\qquad\qquad\qquad+Cr^{\theta\sigma}  \left[\fint_{B_r}|u-(u)_{B_r}|^pdx+r^\delta\left(r^s\int_{\mathbb{R}^n \backslash B_r} \frac{|u(y)-(u)_{B_r}|^{q-1}}{|y-x_0|^{n+ sq}} d y\right)^{\frac{q}{q-1}}+r^{p-\theta}(\|f\|_{L^n(B_r)}^{\frac{p}{p-1}}+1)\right],
     \end{array}
     \end{equation*}
     for some $C\equiv C(n,p,q,s,\Lambda)$. Here $\sigma$ is given by [\citealp{MG}, Lemma 4.2]. \cref{differenceestimate} now holds with $\sigma_0=\operatorname{min}\{\sigma,\bar{\alpha},p\bar{\alpha}/2\}.$
\end{proof}
We are now ready to prove \cref{holder2}.
\subsection*{Proof of \cref{holder2}}
(a) For $2\leq q\leq p$, we have already proved \cref{mgiter} which shows the same decay estimate as of [\citealp{MG}, Lemma 3.1] holds for our $\operatorname{snail}$ term. Also \cref{caccimprovedfinal,differenceestimate} implies that the same Caccioppoli inequality and difference estimate holds like [\citealp{MG}, lemma 4.1, Lemma 4.2] for the local solutions involving our $\operatorname{snail}$ term. Further, regularity estimates for $h$ have been obtained in \cref{proph}. One can thus follow [\citealp{MG}, Subsection 4.3, Pages 30-34] and end up with ([\citealp{MG}, Equation (4.40)])
\begin{equation}{\label{holfinal}}
    \begin{array}{l}
         \left[\fint_{B_r}|u-(u)_{B_r}|^pdx+r^\delta\left(r^s\int_{\mathbb{R}^n \backslash B_r} \frac{|u(y)-(u)_{B_r}|^{q-1}}{|y-x_0|^{n+ sq}} d y\right)^{\frac{q}{q-1}}+r^{p-\theta}(\|f\|_{L^n(B_r)}^{\frac{p}{p-1}}+1)\right]
         \\\leq C\left(\frac{r}{\tilde{r}}\right)^{p\eta} 
         \left[\fint_{B_{\tilde{r}}}
         |u-(u)_{B_{\tilde{r}}}|^pdx
         +\tilde{r}^\delta\left(\tilde{r}^s\int_{\mathbb{R}^n \backslash B_{\tilde{r}}} \frac{|u(y)-(u)_{B_{\tilde{r}}}|^{q-1}}{|y-x_0|^{n+ sq}} d y\right)^{\frac{q}{q-1}}+\tilde{r}^{p\eta}(\|f\|_{L^n(B_{\tilde{r}})}^{\frac{p}{p-1}}+1)\right] 
    \end{array}
\end{equation}
for each $\eta\in (0,1)$ and for all $0<r\leq \tilde{r}\leq r_*\equiv r_*(n,p,q,s,\bar{\alpha},\Lambda,\eta)$, where $ C\equiv C(n,p,q,s,\Lambda,\eta)$. Further estimating
\begin{equation}{\label{estisnail}}
    \begin{array}{l}
       \quad  \fint_{B_{r_*}}|u-(u)_{B_{r_*}}|^pdx+{r_*}^\delta\left({r_*}^s\int_{\mathbb{R}^n \backslash B_{r_*}} \frac{|u(y)-(u)_{B_{r_*}}|^{q-1}}{|y-x_0|^{n+ sq}} d y\right)^{\frac{q}{q-1}}\\\leq \frac{C}{{r_*}^n}\|u\|^p_{L^p(B_{r_*})}+Cr_*^\delta \left(r_*^{-s{(q-1)}}\operatorname{Tail}_{q-1}(u;x_0,r_*)^{q-1}+r_*^{-s{(q-1)}}\left(\fint_{B_{r_*}}|u|\right)^{q-1}\right)^{\frac{q}{q-1}}\\\leq \frac{C}{{r_*}^n}\|u\|^p_{L^p(B_{r_*})}+Cr_*^{\delta-sq}r_*^{-nq/p}\left(\operatorname{Tail}_{q-1}(u;x_0,r_*)^q+\|u\|^q_{L^p(B_{r_*})}\right),
    \end{array}
\end{equation}
holds for some $ C\equiv C(n,p,q,s)$. Therefore, in view of \cref{holfinal}, we have proved that for any $\eta\in(0,1)$, $\exists \,r_*\in(0,1)$ depending on $n,p,q,s,\bar{\alpha},\Lambda\text{ and }\eta$ such that 
\begin{equation}{\label{finalholder}}
     \fint_{B_r}|u-(u)_{B_r}|^pdx\leq C\left(\frac{r}{{r_*}}\right)^{p\eta}
\end{equation}
holds with \begin{equation}{\label{data}}
    C\equiv C(n,s,p,q,\bar{\alpha},\Lambda, \eta,\|u\|_{L^p(B_{{r}_*})},\|f\|_{L^n(B_{{r}_*})}, Tail_{q-1}(u;x_0,r_*))\equiv C(\operatorname{data}_h(\eta),\eta)\geq 1
\end{equation} whenever $B_r\subset\subset \Omega$ and $r\leq r_*$. \cref{finalholder,data} proves the first part of \cref{holder2} along with the desired estimate via the classical Campanato-Meyers integral characterization of H\"older continuity.\smallskip\\(b) For this part, we note [\citealp{MG}, Remark 3, Remark 7], i.e. their set-up allows solutions to be in $W^{1,p}_{\mathrm{loc}}(\Omega)$ with the property that $\int_{\mathbb{R}^n} \frac{|u(x)|^q }{(1+|x|^{n+sq})}dx<\infty$ (this is easy to check from [\citealp{MG}, Equation 4.40]). 
Since $u\geq 0$ is bounded and for some $\theta\geq 1$, $u^\theta\in W_0^{1,p}(\Omega)$ i.e. $u^\theta\in L^q(\mathbb{R}^n)$, therefore for some $C>0$, it holds 
\begin{equation*}
    \int_{\mathbb{R}^n} \frac{u(x)^q }{(1+|x|^{n+sq})}dx\leq C\left(\int_{\mathbb{R}^n} \frac{u(x)^{\theta q}}{(1+|x|^{n+sq})}dx\right)^{1/\theta}=C\left(\int_{\Omega} \frac{u(x)^{\theta q} }{(1+|x|^{n+sq})}dx\right)^{1/\theta}<\infty.
\end{equation*} Also, \cref{caccremark} and estimate of $\mathrm{(I)}$ in \cref{differenec} implies we can take the local operator as $\mathcal{L}_L u(x)=\mathcal{L}_L^A u(x):=-\operatorname{div} A(x, \nabla u(x))$ in the proofs of \cite{MG}. Hence, the local H\"older regularity of $u$ follows from \cite{MG}.
\section{Higher H\"older regularity results for local solutions}{\label{regu3}}
Using the almost Lipschitz regularity 
obtained in \cref{regu2}, we now show that gradients of local weak solutions of $\mathcal{L}u=f$ in $\Omega$ are also H\"older continuous, provided $f\in L^d_{\mathrm{loc}}(\Omega)$ for some $d>n$. In what follows, we take arbitrary open sets $\Omega_0 \subset\subset \Omega_1 \subset\subset \Omega$, and denote $\bar{d}:=\operatorname{min}\{\operatorname{dist}(\Omega_0, \Omega_1),
\operatorname{dist}(\Omega_1, \Omega), 1\}$. We consider $B_{r} \equiv B_{r}(x_0) \subset\subset \Omega_1$ with $x_0 \in \Omega_0$ and $0<r \leq \bar{d} / 4$ and all the balls will be centred at $x_0$. Moreover, $\eta, \lambda$ will be real numbers satisfying $\eta\in(s,1)$ and $\lambda>0$; their precise values will be prescribed later depending on the context. We shall often use \cref{holder2} in the form $\|u\|_{C^{0, \eta}(\Omega_1)} \leq c \equiv c(\operatorname{data}(\eta),\eta,\bar{d})$, ($\operatorname{data}_h$ is given in \cref{data}) for every $\eta<1$.
\begin{lemma}{\label{higher1}}
Under the assumptions on \cref{holder3} and $2\leq q\leq p<\infty$,\\
(a) If $s<\eta<1$, then there exists $c \equiv c( \operatorname{data}_h(\eta), \bar{d}, \eta)>0$ such that
\begin{equation}{\label{higherhol1}}
    \fint_{B_{r/ 2}} \int_{B_{r/ 2}} \frac{|u(x)-u(y)|^q}{|x-y|^{n+s q}} dx d y \leq c r^{(\eta-s) q}.
\end{equation}
(b) There exists $c \equiv c( \operatorname{data}_h, \bar{d})>0$, such that whenever $0<t \leq r$, it holds
\begin{equation}{\label{higherholder2}}
   t^{-\delta}[\operatorname{snail}(u, B_t(x_0))]^q = \left(t^s\int_{\mathbb{R}^n \backslash B_t} \frac{|u(y)-(u)_{B_t}|^{q-1}}{|y-x_0|^{n+ sq}} d y\right)^{\frac{q}{q-1}} \leq c.
\end{equation}
(c) If $\lambda>0$ then there exists $c \equiv c( \operatorname{data}_h(\lambda), \bar{d}, \lambda)>0$ such that it holds
\begin{equation}{\label{higherholder3}}
    \fint_{B_{r / 2}}(|\nabla u|^2+\mu^2)^{p / 2} d x \leq c r^{-p \lambda}.
\end{equation}
\end{lemma}
\begin{proof}
    Estimate \cref{higherhol1} follows from \cref{holder2}. To prove \cref{higherholder2}, we notice from the proof of \cref{mgiter} (the part estimate of $T_1$), that 
    \begin{equation}{\label{oneresult}}
        |(u)_{B_{\bar{d}}}-(u)_{B_{t}}|\leq c\int_t^{\bar{d}} \left(\fint_{B_v}|u(x)-(u)_{B_v}|^pd x  \right)^{1/p}\frac{d v}{v} 
+ c\left(\fint_{B_{\bar{d}}}|u(x)-(u)_{B_{\bar{d}}}|^p d x \right)^{1/p}\leq c\int_t^{\bar{d}} v^{\eta-1}dv+c\bar{d}^{\eta}\leq c \bar{d}^{\eta}, 
    \end{equation}
    where $c \equiv c(\operatorname{data}_h(\eta), d, \eta)$. Using \cref{oneresult}, we estimate as follows:
    \begin{equation*}
    \begin{array}{l}
      t^{-\delta}[\operatorname{snail}(u, B_t(x_0))]^q \leq C\left(t^s\int_{\mathbb{R}^n \backslash B_{\bar{d}}} \frac{|u(y)-(u)_{B_{\bar{d}}}|^{q-1}}{|y-x_0|^{n+ sq}} d y\right)^{\frac{q}{q-1}} +C\left(t^s\int_{\mathbb{R}^n \backslash B_{\bar{d}}} \frac{  |(u)_{B_{\bar{d}}}-(u)_{B_{t}}|^{q-1}}{|y-x_0|^{n+ sq}} d y\right)^{\frac{q}{q-1}} \\\qquad\qquad\qquad\qquad\qquad+C\left(t^s\int_{B_{\bar{d}} \backslash B_t} \frac{|u(y)-u(x_0)|^{q-1}}{|y-x_0|^{n+ sq}} d y\right)^{\frac{q}{q-1}}+C\left(t^s\int_{B_{\bar{d}}\backslash B_{t}} \frac{|u(x_0)-(u)_{B_{t}}|^{q-1}}{|y-x_0|^{n+ sq}} d y\right)^{\frac{q}{q-1}}\smallskip\\\qquad\qquad\qquad\qquad\quad\stackrel{\cref{estisnail}}{\leq} C(\bar{d}^{-sq}\operatorname{Tail}_{q-1}(u;x_0,\bar{d})^q+\bar{d}^{-n-sq}\|u\|_{L^q(\Omega_1)})+C(t^s\bar{d}^{-sq}|(u)_{B_{\bar{d}}}-(u)_{B_{t}}|^{q-1})^{\frac{q}{q-1}}\smallskip\\\qquad\qquad\qquad\qquad\qquad+C\left(\int_{B_{\bar{d}} \backslash B_t} \frac{|y-x_0|^{\eta(q-1)}}{|y-x_0|^{n+ s(q-1)}} d y\right)^{\frac{q}{q-1}}+C\left(t^s\int_{\mathbb{R}^n\backslash B_{t}} \frac{t^{\eta(q-1)}}{|y-x_0|^{n+ sq}} d y\right)^{\frac{q}{q-1}}\smallskip\\\qquad\qquad\qquad\qquad\quad\leq C\bar{d}^{-n-sq}+C(\bar{d}^{-s(q-1)}\bar{d}^{\eta(q-1)})^{\frac{q}{q-1}}+C(\bar{d}^{(\eta-s)(q-1)})^{\frac{q}{q-1}}+C(t^{(\eta-s)(q-1)})^{\frac{q}{q-1}}\smallskip\\\qquad\qquad\qquad\qquad\quad\leq C\bar{d}^{-n-sq}+ C\bar{d}^{(\eta-s)q}+ Ct^{(\eta-s)q}\leq C,
        \end{array}
    \end{equation*}
    where $C \equiv C( \operatorname{data}_h, \bar{d},\eta)$, which gives \cref{higherholder2} if we choose $\eta=(1+s)/2$. Finally, to prove \cref{higherholder3}, we use \cref{caccimprovedfinal} and estimate the various terms as
\begin{equation*}
    r^{-p}[\operatorname{av}_p(r)]^p\leq c r^{p(\eta-1)},
\end{equation*}
holds with $c \equiv c( \operatorname{data}_h, \bar{d},\eta)$, by \cref{holder2}. By \cref{higherholder2} we instead have
\begin{equation*}
 r^{-\delta}[\operatorname{snail}(r)]^q+\|f\|_{L^n(B_r)}^{p /(p-1)}+1 \leq c \leq c r^{p(\eta-1)} .
\end{equation*}
Choosing $\eta$ such that $1-\eta \leq \lambda$, we arrive at \cref{higherholder3}.
\end{proof}
\begin{lemma}{\label{higher2}}
   Let $2\leq q\leq p<\infty$.  If $h \in u+W_0^{1, p}(B_{r/ 4}(x_0))$ the solution of \cref{h}, then
\begin{equation}{\label{higherholder4}}
    \fint_{B_{r/ 4}(x_0)}|\nabla u-\nabla h|^p d x \leq c r^{\sigma_2 p}
\end{equation}
holds where $\sigma_2 \equiv \sigma_2(n, p, s, q, \bar{d}) \in(0,1)$ and $c \equiv c(\operatorname{data}_h, \bar{d},\|f\|_{L^d(\Omega_1)})$.
\end{lemma}
\begin{proof} 
We go back to \cref{differenec}, estimate \cref{diff1}, and adopting the notation introduced there, we improve the estimates for the terms $\mathrm{(I)}-\mathrm{(IV)}$. As in \cref{di1}, we find using \cref{higherholder3}
\begin{equation}{\label{high1}}
\begin{array}{l}
     |(\mathrm{I})| \leq Cr^{\bar{\alpha}}\left(\fint_{B_{r / 4}} (|\nabla u|^2+\mu^2)^{\frac{p}{2}}dx\right)^{\frac{p-1}{p}}\left(\fint_{B_{r / 4}}|\nabla u-\nabla h|^p\right)^{\frac{1}{p}}\stackrel{\cref{h1}}{\leq} C_\lambda r^{\bar{\alpha}-p\lambda},
\end{array}  
\end{equation}
for some $C_\lambda >0$. Next, we estimate by Sobolev inequality
\begin{equation}{\label{high2}}
\begin{array}{rcl}
     |(\mathrm{II})| {\leq} C\|f\|_{L^n(B_{r / 4})}\left(\fint_{B_{r/ 4}}|\nabla  u-\nabla h|^{p } d x\right)^{1 / p} 
&\stackrel{\cref{h1}}{\leq}& C\|f\|_{L^n(B_{r / 4})}\left(\fint_{B_{r/ 4}}(|\nabla  u|^2+\mu^2)^{p / 2} d x\right)^{1 / p} 
 \\&\stackrel{\cref{higherholder3}}{\leq}& C_\lambda\|f\|_{L^d(B_{r / 4})} r^{1-n / d-\lambda} ,  
\end{array} 
\end{equation}
for some $C_\lambda>0$. Now we use H\"older inequality and \cref{fracemb} and note that $u-h\equiv 0$ outside $B_{r/4}(x_0)$,  to estimate $\mathrm{(III)}$ as
\begin{equation}{\label{high3}}
    \begin{array}{rcl}
         |(\mathrm{III})| &\leq & C\left(\fint_{B_{r / 2}} \int_{B_{r / 2}} \frac{|u(x)-u(y)|^q}{|x-y|^{n+s q}} d x d y\right)^{\frac{q-1}{q}}\left(\fint_{B_{r / 2}} \int_{B_{r / 2}} \frac{|(u-h)(x)-(u-h)(y)|^q}{|x-y|^{n+s q}} d x d y\right)^{\frac{1}{q}}\\& \stackrel{\cref{higherhol1}}{\leq} &C r^{(\eta-s)(q-1)}r^{1-s} \left(\fint_{B_{r/ 4}}|\nabla u-\nabla h|^p dx\right)^{1/ p}\\
& \stackrel{\cref{h1}}{\leq}& C r^{(\eta-s)(q-1)}r^{1-s} \left(\fint_{B_{r/ 4}}(|\nabla u|^2+\mu^2)^{p/2} dx\right)^{1/ p}  \stackrel{\cref{higherholder3}}{\leq}C_\lambda r^{(\eta-s)(q-1)+1-s-\lambda} ,
    \end{array}
\end{equation}
whenever $s<\eta<1$ and $\lambda>0$. As $u$ is locally H\"older continuous, by Maz'ya-Wiener boundary regularity theory, $h$ is continuous on $\bar{B}_{r/4}$ and therefore
\begin{equation}{\label{MW}}
    \|u-h\|_{L^\infty(B_{r/4})}\stackrel{\cref{h2}}{\leq}2 \underset{\partial B_{r/4}}{\operatorname{osc} }\, u\leq 4[u]_{C^{0, \eta}(B_{r/4})}r^\eta\leq Cr^\eta ,
\end{equation} where $C \equiv C(\operatorname{data}_h,\bar{d},\eta) $. To estimate $\mathrm{(IV)}$ we restart from the third line of \cref{di}, and using also \cref{higherholder2,MW}, we easily find
\begin{equation}{\label{high4}}
    \begin{array}{l}
       |(\mathrm{IV})|   \leq C\left[ r^{-sq}\left(\fint_{B_{r / 2}}|u(x)-(u)_{B_{r/2}}|^{q}\right)^{\frac{q-1}{q}}+\int_{\mathbb{R}^n \backslash B_{r/ 2}}  \frac{|u(y)-(u)_{B_{r/2}}|^{q-1}}{|y-x_0|^{n+sq}}dy\right]\left(\fint_{B_{r / 2}}|u(x)-h(x)|^qdx\right)^{1/q}\\\qquad\stackrel{\cref{higherholder2}}{\leq} (r^{\eta(q-1)-sq}+r^{-s}) \|u-h\|_{L^\infty(B_{r/4})}\leq  C(r^{\eta(q-1)-sq}+r^{-s}) r^\eta\leq C_\eta r^{\eta-s}.
    \end{array}
\end{equation}
Merging \cref{high1,high2,high3,high4}, we get from \cref{diff1}
\begin{equation}{\label{high5}}
      \fint_{B_{r / 4}}|V_\mu(\nabla u)-V_\mu(\nabla h)|^2 dx\leq C_{\eta,\lambda}(r^{\bar{\alpha}-p\lambda}+r^{1-n / d-\lambda} +r^{(\eta-s)(q-1)-\lambda}+r^{\eta-s}),
\end{equation}
for every $\eta\in(s,1)$ and $\lambda>0$. We then choose $\eta, \lambda$ such that
\begin{equation*}
    \eta:=\frac{1+s}{2}, \quad 0<\lambda =\min \left\{\frac{\bar{\alpha}}{2p}, \frac{(1-s)(q-1)}{4}, \frac{1}{2}\left(1-\frac{n}{d}\right)\right\}
\end{equation*}
and put in \cref{high5} to get
\begin{equation}{\label{high6}}
\left\{\begin{array}{l}
   \fint_{B_{r / 4}}|V_\mu(\nabla u)-V_\mu(\nabla h)|^2 dx \leq C r^{\sigma_1 p} \\
\sigma_1:=\frac{1}{p} \min \left\{\frac{\bar{\alpha}}{2},\frac{1}{2}\left(1-\frac{n}{d}\right), \frac{(1-s)(q-1)}{4}, \frac{1-s}{2}\right\}>0
\end{array}\right.
\end{equation}
for $C \equiv C(\operatorname{data}_h, \bar{d},\|f\|_{L^d(\Omega_1)})$. By H\"older inequality, \cref{high6,g1} gives \cref{higherholder4} with $\sigma_2=\sigma_1$ as $p\geq 2$.
\end{proof}
\subsection*{Proof of \cref{holder3}}(a) Once \cref{higherholder4} is established, we can conclude the local Hölder continuity of $\nabla u$ by a classical comparison argument (see for instance \cite{manfredi}). We briefly include it here for the sake of completeness. For every $0<t\leq r/8$, we estimate using \cref{av}
\begin{equation*}
    \begin{array}{l}
\fint_{B_t}|\nabla u-(\nabla u)_{B_t}|^p d x\leq    c(\underset{B_t}{\operatorname{osc}}\, \nabla h)^p+c\left(\frac{r}{t}\right)^n \fint_{B_{r / 4}}|\nabla u-\nabla h|^p d x \\\qquad\qquad\qquad\stackrel{\cref{h3},\cref{higherholder4}}{\leq} c\left(\frac{t}{r}\right)^{\alpha_0 p} \fint_{B_{r / 4}}(|\nabla h|^2+\mu^2)^{p / 2} d x+c\left(\frac{r}{t}\right)^n r^{\sigma_2 p}\stackrel{\cref{h1},\cref{higherholder3}}{\leq} c\left(\frac{t}{r}\right)^{\alpha_0 p}r^{-\lambda p}+c\left(\frac{r}{t}\right)^n r^{\sigma_2 p},
    \end{array}
\end{equation*}
with $c \equiv c( \operatorname{data}_h, \bar{d},\lambda, \|f\|_{L^d(\Omega_1)})$. In the above inequality, we take $t=\frac{1}{8}r^{\frac{2n+\sigma_2 p}{2 n}}$ and choose $\lambda:=\frac{\sigma_2 p \alpha_0}{4 n}$ in \cref{higherholder3}. We conclude with
\begin{equation*}
\fint_{B_t}|\nabla u-(\nabla u)_{B_t}|^p d x\leq c t^{\bar{\eta} p}, \quad \bar{\eta}:=\min \left\{\frac{n {\sigma}_2}{2 n+{\sigma}_2p}, \frac{\sigma_2 \alpha_0p}{2 \sigma_2 p+4 n}\right\}    \end{equation*}
where $c \equiv c( \operatorname{data}_h, \bar{d}, \|f\|_{L^d(\Omega_1)})$. It holds whenever $B_t \subset\subset \Omega$ is a ball centred in $\Omega_0$, with $t \leq \bar{d}^{1+\sigma_2 p /(2 n)}/c(n,p) $. As $\Omega_0 \subset\subset \Omega_1 \subset\subset \Omega$ are arbitrary, this implies the local $C^{1, \bar{\eta}}$-regularity of $u$ in $\Omega$, via the classical Campanato-Meyers' integral characterization of Hölder continuity together with the estimate for $[\nabla u]_{0, \bar{\eta}; \Omega_0}$, and the proof is complete.
\\(b) Note that (as already described in the proof of \cref{holder2}, part (b)) $u$ is positive and $  \int_{\mathbb{R}^n} \frac{u(x)^q }{(1+|x|^{n+sq})}dx<\infty$ implies $ \int_{\mathbb{R}^n\backslash B_r} \frac{u(x)^q }{|x-x_0|^{n+sq}}dx<\infty, \forall B_r(x_0)\subset\subset\Omega$ and hence $\int_{\mathbb{R}^n\backslash B_r} \frac{|u(x)-(u)_{B_r}|^q }{|x-x_0|^{n+sq}}dx<\infty, \forall B_r(x_0)\subset\subset\Omega$. Thus lemma 6.1 of \cite{MG} follows. In view of the proof of \cref{higher2}, estimate of $\mathrm(I)$ and choice of $\lambda, \sigma_1$ it is clear that [\citealp{MG}, Lemma 6.2] holds for the choice of operator $\mathcal{L}$, with $\sigma_1$ given by \cref{high6}. Therefore the gradient regularity follows from \cite{MG}.
\section{Up to the boundary H\"older regularity}{\label{regu4}}
We now proceed with the proof of \cref{holder4}. The local $C^{1,\gamma}$ (for some $\gamma\in(0,1)$) regularity follows from \cite{MG}. We show the boundary regularity under the conditions of \cref{holder4}. To do this, it is convenient to locally straighten the boundary around any given point $x_0 \in \partial \Omega$. Without loss of generality (by translation) we can assume that $x_0 \in \{x_n = 0\}$ and that $\Omega$ touches $\{x_n = 0\}$ tangentially, so that its normal at $x_0$ is $e_n$, where $\{e_i\}_{i\leq n}$ is the standard basis of $\mathbb{R}^n$. As $\partial \Omega$ is $C^2$, following the argument of [\citealp{MG}, Section 5] (and mostly adopting its notation), we see that there exists a global $C^{1, \bar{\alpha}}$-diffeomorphism $\mathcal{T}$ of $\mathbb{R}^n$ such that
\begin{equation*}
    \mathcal{T}(x_0)=x_0, \quad B_{2r_0}^{+}(x_0) \subset \mathcal{T}(\Omega_{3 r_0}(x_0)) \subset B_{4 r_0}^{+}(x_0),\quad
\Gamma_{2r_0}(x_0) \subset \mathcal{T}(\partial \Omega \cap B_{3 r_0}(x_0)) \subset \Gamma_{4 r_0}(x_0),
\end{equation*}
for some small radius $r_0 \in(0,1]$. Here $ B_{r_0}^{+}(x_0)=B_r(x_0)\cap \{x\in \mathbb{R}^n:x_n>0\}, \,\Omega_r(x_0)=\Omega \cap B_r(x_0)$ and $\Gamma_r(x_0)=B_r(x_0) \cap\{x_n=0\}$. We write $\mathcal{S}:=\mathcal{T}^{-1}$ and $\mathbf{c}:=|\mathcal{J}_{\mathcal{S}}|$, with $\mathcal{J}_{\mathcal{S}}$ denoting the Jacobian determinant of the inverse $\mathcal{S}$. Let $\widetilde{\Omega}:=\mathcal{T}(\Omega), \tilde{g}:=g \circ \mathcal{S}, \tilde{f}:=\mathbf{c}(f \circ \mathcal{S})$, and $\tilde{u}:=u \circ \mathcal{S} .$ The diffeomorphism $\mathcal{T}$ leaves all the structural conditions on $\mathcal{L}$ unchanged and the same is true for $f$ and $u$ given in \cref{holder4}. Moreover, $\tilde{u}\in W_0^{1,p}(\widetilde{\Omega})$ is a weak solution of
  \begin{equation}{\label{p3}}
    -\operatorname{div} \tilde{A}(\cdot, \nabla \tilde{u}) +\widetilde{Q}_N \tilde{u}  =\tilde{f} \text{ in }\widetilde{\Omega}; \quad  \tilde{u}=0\text{ in }\mathbb{R}^n\backslash\widetilde{\Omega},
    \end{equation}
where
\begin{equation*}
\widetilde{A}(x, z):=\mathbf{c}(x) A(\mathcal{S}(x), z(D \mathcal{T} \circ \mathcal{S})(x))(D \mathcal{T} \circ \mathcal{S})(x)^T    
\end{equation*}
and
\begin{equation*}
    \widetilde{Q}_N u(x):= \text { P.V. } \int_{\mathbb{R}^n} \phi(u(x)-u(y)) \widetilde{K}(x, y) d y,
\end{equation*}
with
\begin{equation*}
    \widetilde{K}(x, y):=\mathbf{c}(x) \mathbf{c}(y) \frac{B(\mathcal{S}(x), \mathcal{S}(y))}{|\mathcal{S}(x)-\mathcal{S}(y)|^{n+s q}} .
\end{equation*}
From assumptions \cref{growth1,growth2,growth3} and the regularity of $\mathcal{T}$, we infer that
\begin{equation}{\label{growth4}}
    \begin{cases}0<\widetilde{\Lambda}^{-1} \leq \mathbf{c}(x) \leq \widetilde{\Lambda} & \text { for all } x \in \mathbb{R}^n, \\ |\mathbf{c}(x)-\mathbf{c}(y)| \leq \widetilde{\Lambda}|x-y|^{\bar{\alpha}} & \text { for all } x, y \in B_{r_0}(x_0), \\ \widetilde{K}(x, y)=\widetilde{K}(y, x) & \text { for a.e. } x, y \in \mathbb{R}^n, \\ \frac{\widetilde{\Lambda}^{-1}}{|x-y|^{n+s q}} \leq \widetilde{K}(x, y) \leq \frac{\widetilde{\Lambda}}{|x-y|^{n+s q}} & \text { for a.e. } x, y \in \mathbb{R}^n ,\end{cases}
\end{equation}
and that the $p$-growth and coercivity conditions are preserved, namely
\begin{equation}{\label{growth5}}
\begin{cases}|\widetilde{A}(x, z)|+|z||\partial_z \widetilde{A}(x, z)| \leq \widetilde{\Lambda}(|z|^2+\mu^2)^{\frac{p-2}{2}}|z| & \text { for all } x \in B_{r_0}^{+}(x_0), \\ |\widetilde{A}(x, z)-\widetilde{A}(y, z)| \leq \widetilde{\Lambda}(|z|^2+\mu^2)^{\frac{p-1}{2}}|x-y|^{\bar{\alpha}} & \text { for all } x, y \in B_{r_0}^{+}(x_0) \\ \langle\partial_z \widetilde{A}(x, z) \xi, \xi\rangle \geq \widetilde{\Lambda}^{-1}(|z|^2+\mu^2)^{\frac{p-2}{2}}|\xi|^2 & \text { for all } x \in B_{r_0}^{+}(x_0), \xi \in \mathbb{R}^n,\end{cases}
\end{equation}
for every $z \in \mathbb{R}^n \backslash\{0\}$ and for some constant $\widetilde{\Lambda} \geq 1$ depending only on $n, p, \bar{\alpha}, \Lambda$, and $\Omega$.
The conditions on $f$ and $u$ in \cref{holder4} now reads as 
\begin{equation}{\label{growth6}}
    0\leq \tilde{f}(x)\leq cx_n^{-\sigma}, \quad 0\leq \tilde{u}(x)\leq M_\ep x_n^\ep,\, \quad \text{ for some }\ep\in[\sigma,1),\quad \text{ for all } x \in B_{r_0}^{+}(x_0).
\end{equation}
We take any point $\tilde{x}_0 \in \Gamma_{r_0 / 2}(x_0)$, radius $\varrho \in\left(0, \frac{r_0}{4}\right]$ , and upper balls $B_\varrho\equiv B^+_\varrho(\tilde{x}_0)\subset B^+_{r_0}({x}_0)$. Unless otherwise stated, all the upper balls will be centred at $\tilde{x}_0 $, and $\tilde{x}_0 $ will be a fixed but generic point as specified.
\begin{lemma}{\label{boundarycacc}}{(Boundary Caccioppoli type inequality)}
Under the assumptions of \cref{holder4}, the inequality
\begin{equation}{\label{bddcaccio}}
    \begin{array}{l}
        \quad \fint_{B_{\varrho/ 2}^{+}(\tilde{x}_0)}(|\nabla \tilde{u}|^2+\mu^2)^{p / 2} d x+ \fint_{B_{\varrho / 2}(\tilde{x}_0)} \int_{B_{\varrho/ 2}(\tilde{x}_0)} \frac{|\tilde{u}(x)-\tilde{u}(y)|^q}{|x-y|^{n+s q}} d x d y\smallskip\\\leq  c\left(\varrho^{-p} \fint_{B_\varrho^{+}(\tilde{x}_0)}|\tilde{u}|^p d x+\int_{\mathbb{R}^n \backslash B_\varrho(\tilde{x}_0)} \frac{|\tilde{u}(y)-(\tilde{u})_{B_\varrho(\tilde{x}_0)}|^q}{|y-\tilde{x}_0|^{n+s q}} d y 
         +1\right)
    \end{array}
\end{equation}
holds with $c \equiv c(n,p,s,q,\Lambda, \sigma,\Omega)$.
\end{lemma}
\begin{proof}
    Let us fix parameters $\varrho / 2 \leq \tau_1<\tau_2 \leq \varrho$, a function $\eta \in C_0^1(B_{\tau_2})$ such that $\chi_{B_{\tau_1}} \leq\eta \leq \chi_{B_{(3 \tau_2+\tau_1) / 4}}$ and $|\nabla \eta| \lesssim 1 /(\tau_2-\tau_1)$. Consider $\tilde{\varphi}:=\eta^p\tilde{u}$. By definition, $\tilde{\varphi}$ vanishes outside $B_{\tau_2}^{+} \subset B_{\varrho}^{+} \subset B_{r_0}^{+}(x_0)$, so that $\tilde{\varphi} \in W_0^{1,p}(B_{\varrho}^{+})$. Taking $\tilde{\varphi}$ as a test function and noting [\citealp{MG}, Section 2.5], we find
\begin{equation*}
    \begin{array}{rcl}
         0&=&\int_{B_\varrho^{+}} \eta^p \left[ \tilde{A}(x, \nabla\tilde{u}) \cdot \nabla \tilde{u}-\tilde{f} \tilde{u}\right] d x+p\int_{B_\varrho^{+}}\tilde{u}\, \eta^{p-1}  \tilde{A}(x, \nabla \tilde{u}) \cdot \nabla \eta \,dx\smallskip\\&&+\int_{B_{\tau_2}} \int_{B_{\tau_2}}|\tilde{u}(x)-\tilde{u}(y)|^{q-2}(\tilde{u}(x)-\tilde{u}(y)) \cdot(\eta^p(x) \tilde{u}(x)-\eta^p(y) \tilde{u}(y)) \tilde{K}(x, y) dx d y\smallskip\\&&+2 \int_{\mathbb{R}^n \backslash B_{\tau_2}} \int_{B_{\tau_2}}|\tilde{u}(x)-\tilde{u}(y)|^{q-2}(\tilde{u}(x)-\tilde{u}(y)) \eta^p(x) \tilde{u}(x) \tilde{K}(x, y) d x d y \\
& =:&(\mathrm{I})+(\mathrm{II})+(\mathrm{IIII})+(\mathrm{IV}).
    \end{array}
\end{equation*}
In view of [\citealp{MG}, Lemma 5.1], it is enough to estimate the term involving $\tilde{f}$ only. As $0<\sigma<1$, $0\leq \eta\leq 1$ and $\varrho<1$, we have by \cref{growth4,growth6}
\begin{equation*}
    \int_{B_\varrho^{+}} \eta^p \tilde{f} \tilde{u}\, dx \leq c \int_{B_\varrho^{+}}x_n^{\ep-\sigma}dx\leq c \varrho^{n+\ep-\sigma}\leq c,
\end{equation*}
for some $c\equiv c(\sigma,n)$. Proof of \cref{bddcaccio} now follows from [\citealp{MG}, Lemma 5.1] (taking $\tilde{g}\equiv 0$ there).
\end{proof}
Consider now the boundary $p$-harmonic function $\tilde{h} \in W_{\tilde{u}}^{1, p}(B_{\varrho / 4}^{+}(\tilde{x}_0))$ of the homogeneous Dirichlet problem
\begin{equation}{\label{boundarypharmonic}}
    \begin{cases}-\operatorname{div} \widetilde{A}(\tilde{x}_0, \nabla\tilde{h})=0 & \text { in } B_{\varrho / 4}^{+}(\tilde{x}_0), \\ \tilde{h}=\tilde{u} & \text { on } \partial B_{\varrho / 4}^{+}(\tilde{x}_0)\end{cases}
\end{equation}
In the next lemma, which is a boundary version of \cref{proph}, we collect some useful properties of $\tilde{h}$, which are essentially all contained in [\citealp{garymlib}, Lemma 5]. The existence and uniqueness of $\tilde{h}$ is classical and is mentioned, for instance, in \cite{garymlib}. 
\begin{lemma}{\label{proph2}}
    Let $\tilde{h}$ be the solution of problem \cref{boundarypharmonic}. Then, there exist constants $\tilde{\alpha}_0 \in(0,1)$ and $C>0$ such that,
    \begin{equation}{\label{bdh1}}
    \fint_{B_{\varrho/ 4}^{+}(\tilde{x}_0)}(|\nabla \tilde{h}|^2+\mu^2)^{p / 2} d x \leq C  \fint_{B_{\varrho/ 4}^{+}(\tilde{x}_0)}(|\nabla \tilde{u}|^2+\mu^2)^{p / 2} d x,
\end{equation}
\begin{equation}{\label{bdh2}}
    \|\tilde{h}\|_{L^{\infty}(B^+_{\varrho/ 4}(\tilde{x}_0))} \leq\|\tilde{u}\|_{L^{\infty}(B^+_{\varrho/ 4}(\tilde{x}_0))}, \quad \underset{B^+_{\varrho/ 4}(\tilde{x}_0)}{\operatorname{osc}} \tilde{h}\leq  \underset{B^+_{\varrho/ 4}(\tilde{x}_0)}{\operatorname{osc}} \tilde{u},
\end{equation}
\begin{equation}{\label{bdh3}}
    \fint_{B_{t}^{+}(\tilde{x}_0)}(|\nabla \tilde{h}|^2+\mu^2)^{p / 2} d x\leq c \left(\frac{t}{\varrho}\right)^{-n+b}  \fint_{B_{\varrho/ 4}^{+}(\tilde{x}_0)}(|\nabla \tilde{h}|^2+\mu^2)^{p / 2} d x
\end{equation}
and 
\begin{equation}{\label{bdh4}}
   \underset{B^+_{t}(\tilde{x}_0)}{\operatorname{osc}}\nabla \tilde{h}\leq c\left(\frac{t}{\varrho}\right)^{\tilde{\alpha}_0}\left( \fint_{B^+_{\varrho / 4}(\tilde{x}_0)}(|\nabla \tilde{h}|^2+\mu^2)^{p / 2} d x\right)^{1/p}
\end{equation}
holds for all $t\in (0,\varrho/8]$ with $c\equiv c(n,p,q,s,\Lambda)$ and $b$ such that $0\leq b<n$.
\end{lemma}
\begin{proof}
The proof of \cref{bdh1} follows similarly to the proof of \cref{h1}. The proof of \cref{bdh2} follows immediately from the weak maximum principle. Proof of \cref{bdh3} is given in [\citealp{MG}, section 5.7] and finally the proof of \cref{bdh4} is established in [\citealp{garymlib}, lemma 5].
\end{proof}
\begin{lemma}{\label{differenecboundary}}
  Let the conditions of \cref{holder4} are satisfied and $\tilde{h} \in W_{\tilde{u}}^{1, p}(B_{\varrho / 4}^{+}(\tilde{x}_0))$ be the solution of \cref{boundarypharmonic}. There exists $\tilde{\sigma}_0 \equiv \tilde{\sigma}_0(p, s, q,\bar{\alpha}) \in$ $(0,1)$ such that
\begin{equation}{\label{differenceestimateboundary}}
    \fint_{B^+_{\varrho / 4}(\tilde{x}_0)}|\tilde{u}-\tilde{h}|^p dx \leq c \varrho^{\theta\tilde{\sigma}_0}\left[\fint_{B^+_\varrho(\tilde{x}_0)}|\tilde{u}|^pdx+\varrho^\delta\int_{\mathbb{R}^n \backslash B_\varrho(\tilde{x}_0)} \frac{|\tilde{u}(y)-(\tilde{u})_{B_\varrho(\tilde{x}_0)}|^{q}}{|y-\tilde{x}_0|^{n+ sq}} d y+\varrho^{p-\theta}
    \right],
\end{equation}
holds for every $\theta\in (0,1)$ where $\delta \in(sq,p)$, $c\equiv c(n,p,q,s,\sigma,\Lambda)$.
\end{lemma}
\begin{proof}
    Similarly like \cref{differenec}, noting [\citealp{MG}, Section 2.5], we take the test function $\tilde{u}-\tilde{h}\in W_0^{1,p}(B_{\varrho/4}^+(x_0))$ in the equations of $\tilde{u}$ and $\tilde{h}$ and subtract to get 
    \begin{equation}{\label{diffbdd1}}
    \begin{array}{rcl}
     \fint_{B^+_{\varrho / 4}}|V_\mu(\nabla \tilde{u})-V_\mu(\nabla \tilde{h})|^2 dx &\leq &c\fint_{B^+_{\varrho/ 4}}(\tilde{A}(\tilde{x}_0, \nabla\tilde{u})-\tilde{A}(\tilde{x}_0, \nabla\tilde{h})) \cdot(\nabla \tilde{u}-\nabla \tilde{h})dx\\&
     =&c \fint_{B^+_{\varrho/ 4}}\tilde{A}(\tilde{x}_0, \nabla\tilde{u}) \cdot(\nabla \tilde{u}-\nabla \tilde{h})dx\\&=&c\fint_{B^+_{\varrho / 4}} (\tilde{A}(\tilde{x}_0, \nabla\tilde{u})-\tilde{A}(x, \nabla\tilde{u})) \cdot(\nabla \tilde{u}-\nabla \tilde{h})dx+ c \fint_{B^+_{\varrho / 4}} \tilde{f} (\tilde{u}-\tilde{h})d x\\&&-\frac{c}{|B_\varrho^+|}\int_{\mathbb{R}^n}\int_{\mathbb{R}^n}|\tilde{u}(x)-\tilde{u}(y)|^{q-2}(v)((\tilde{u}-\tilde{h})(x)-(\tilde{u}-\tilde{h})(y))\tilde{K}(x, y) d x d y
    \smallskip \\
&=&(\mathrm{I})+(\mathrm{II})+(\mathrm{III}),
    \end{array}
\end{equation}where $c \equiv c(n, p, \tilde{\Lambda})$. In view of [\citealp{MG}, Lemma 5.2] and our growth conditions \cref{growth4,growth5}, we only estimate the second term. As $0<\sigma<1$, we have using \cref{growth6} and \cref{bdh2}
\begin{equation*}{\label{ff}}
     \fint_{B^+_{\varrho / 4}} |\tilde{f} (\tilde{u}-\tilde{h})|d x
    \leq 2 \fint_{B^+_{\varrho / 4}} |\tilde{f}| \|\tilde{u}\|_{L^\infty(B^+_{\varrho / 4})}d x\leq c\varrho^\ep \fint_{B^+_{\varrho / 4}}x_n^{-\sigma}dx\leq c\varrho^{\ep-\sigma}\leq c,
\end{equation*}
for some $c\equiv c(n,\sigma)$. Proof of \cref{differenceestimateboundary} now follows from [\citealp{MG}, Lemma 5.2].
\end{proof}
Once \cref{boundarycacc} and \cref{differenecboundary} are established, then using the growth conditions \cref{growth4}, \cref{growth5} and \cref{proph2}, one can follow [\citealp{MG}, Section 5.6], to establish the following global H\"older regularity of solutions $u$ satisfying conditions of \cref{holder4}.
\begin{theorem}{\label{globalholder1}}
Under the conditions of \cref{holder4}, $u\in C^{0,\eta}(\mathbb{R}^n)$ for all $\eta\in(0,1)$ and $[u]_{0,\eta,\mathbb{R}^n}\leq   c(\operatorname{data})\equiv c(n,p,s,\eta,q,\Lambda, \sigma,\Omega)$.
\end{theorem}
To prove \cref{holder4}, we proceed with the ideas of [\citealp{MG}, Section 6] and prove the following lemmas.
\begin{lemma}{\label{gradhol1}}Under the conditions of \cref{holder4}, $\tilde{u}$ (solution of \cref{p3}) satisfies
\begin{equation}{\label{grad1}}
    \fint_{B_{t}(\tilde{x}_0)} \int_{B_t(\tilde{x}_0)} \frac{|\tilde{u}(x)-\tilde{u}(y)|^q}{|x-y|^{n+s q}} d x d y \leq C_\eta t^{(\eta-s) q} \quad \text { for all } \eta \in(s, 1) \text { and } t \in\left(0, \frac{r_0}{4}\right) \text {, } 
\end{equation}
\begin{equation}{\label{grad2}}
    \int_{\mathbb{R}^n \backslash B_t(\tilde{x}_0)} \frac{|\tilde{u}(y)-(\tilde{u})_{B_t(\tilde{x}_0)}|^q}{|y-\tilde{x}_0|^{n+s q}} d y \leq C \quad \text { for all } t \in\left(0, \frac{r_0}{4}\right),
\end{equation}
\begin{equation}{\label{grad3}}
     \fint_{B^+_{\varrho / 2}(\tilde{x}_0)}(|\nabla \tilde{u}|^2+\mu^2)^{p / 2} d x \leq C_\lambda \varrho^{-\lambda p} \quad \text { for all } \lambda>0 \text { and } \varrho \in\left(0, \frac{r_0}{4}\right) \text {. }
\end{equation}
The constant $C_\eta$ depends on $\eta$, while $C_\lambda$ depends on $\lambda$.
\end{lemma}
\begin{proof}
\cref{grad1} follows from \cref{globalholder1}. To establish \cref{grad2}, we estimate as
\begin{equation*}
    \begin{array}{l}
         \int_{\mathbb{R}^n \backslash B_t(\tilde{x}_0)} \frac{|\tilde{u}(y)-(\tilde{u})_{B_t(\tilde{x}_0)}|^q}{|y-\tilde{x}_0|^{n+s q}} d y \leq     \int_{\mathbb{R}^n \backslash B_{r_0}(\tilde{x}_0)} \frac{|\tilde{u}(y)-(\tilde{u})_{B_t(\tilde{x}_0)}|^q}{|y-\tilde{x}_0|^{n+s q}} d y +2^{q-1} \int_{B_{r_0}(\tilde{x}_0) \backslash B_t(\tilde{x}_0)} \frac{|\tilde{u}(y)-\tilde{u}(\tilde{x}_0)|^q}{|y-\tilde{x}_0|^{n+s q}} d y \smallskip\\\qquad\qquad\qquad\qquad\qquad\qquad\qquad+2^{q-1}|\tilde{u}(\tilde{x}_0)-(\tilde{u})_{B_t(\tilde{x}_0)}|^q \int_{B_{r_0}(\tilde{x}_0) \backslash B_t(\tilde{x}_0)} \frac{d y}{|y-\tilde{x}_0|^{n+s q}} \smallskip\\\leq C\left[r_0^{-s q}\|\tilde{u}\|_{L^{\infty}(\mathbb{R}^n)}^q+[\tilde{u}]_{C^\eta(B_{r_0}(\tilde{x}_0))}^q\left(\int_{B_{r_0}(\tilde{x}_0)} \frac{d y}{|y-\tilde{x}_0|^{n+(s-\eta) q}}+t^{\eta q} \int_{\mathbb{R}^n \backslash B_t(\tilde{x}_0)} \frac{d y}{|y-\tilde{x}_0|^{n+s q}}\right)\right] \smallskip\\\leq C_\eta(r_0^{-s q}+r_0^{(\eta-s) q}+t^{(\eta-s) q}),
    \end{array}
\end{equation*}
for every $\eta \in(s, 1)$ and some constant $C_\eta>0$ depending on $\eta$ and $\operatorname{data}$ as mentioned in \cref{globalholder1}. By choosing $\beta=(1+s) / 2$, we find the desired inequality \cref{grad2}.
Finally, to prove \cref{grad3}, we recall the boundary Caccioppoli inequality \cref{boundarycacc}. Given \cref{grad2}, we only need to estimate the first term on the right hand side of \cref{bddcaccio}. Noting that $\tilde{u}(\tilde{x}_0)=0$, we have 
\begin{equation*}
    \varrho^{-p} \fint_{B_\varrho^{+}(\tilde{x}_0)}|\tilde{u}|^p d x=\varrho^{-p} \fint_{B_\varrho^{+}(\tilde{x}_0)}|\tilde{u}(x)-\tilde{u}(\tilde{x}_0)|^p d x\leq C\varrho^{(\eta-1)p},
\end{equation*}
for some $C>0$ depending on the parameters mentioned in \cref{globalholder1}. This along with \cref{bddcaccio,grad2} gives
\begin{equation*}
     \fint_{B^+_{\varrho / 2}(\tilde{x}_0)}(|\nabla \tilde{u}|^2+\mu^2)^{p / 2} d x\leq C\varrho^{(\eta-1)p},
\end{equation*}
as $\varrho<1$. Choosing $\eta$ such that $1-\eta\leq \lambda$, we get \cref{grad3}.
\end{proof}
\begin{remark}{\label{oscila}}
In view of \cref{bdh2} and \cref{globalholder1}, it holds
\begin{equation*}
  \|\tilde{u}-\tilde{h}\|_{L^\infty(B^+_{\varrho/ 4}(\tilde{x}_0))}\leq    \underset{B^+_{\varrho/ 4}(\tilde{x}_0)}{\operatorname{osc}} \tilde{h}+ \underset{B^+_{\varrho/ 4}(\tilde{x}_0)}{\operatorname{osc}} \tilde{u}\leq 2 \underset{B^+_{\varrho/ 4}(\tilde{x}_0)}{\operatorname{osc}} \tilde{u}\leq 2[\tilde{u}]_{C^\eta(B^+_{\varrho/ 4}(\tilde{x}_0))}\varrho^\eta\leq C_\eta\varrho^\eta
\end{equation*}
for all $\eta\in(0,1)$
\end{remark}
\begin{lemma}{\label{gradhol2}}
Let $\tilde{u}$ and $\tilde{h}$ be solutions of \cref{p3} and \cref{boundarypharmonic} respectively. Then under the condition of \cref{holder4}, there exists constants $C>0$ and $\bar{\sigma}\in(0,1)$ such that 
    \begin{equation}{\label{gradientboundaryregu}}
          \fint_{B^+_{\varrho / 4}(\tilde{x}_0)}|\nabla\tilde{u}-\nabla\tilde{h}|^p dx\leq C\varrho^{\bar{\sigma}p}.
    \end{equation}
\end{lemma}
\begin{proof}First note that, by \cref{bdh1} and \cref{grad3}, it holds
\begin{equation}{\label{gra}}
    \fint_{B_{\varrho / 4}^{+}(\tilde{x}_0)}|\nabla\tilde{u}-\nabla\tilde{h}|^p d x \leq C_\lambda \varrho^{-\lambda p},
\end{equation}
for every $\lambda>0$ and for some constant $C_\lambda>0$ depending also on $\lambda$. Now we estimate different terms present in the right hand side of \cref{diffbdd1}. The first term can be estimated like the homonym term in \cref{higher2} using \cref{growth5,bdh1,grad3,gra} and yields 
\begin{equation}{\label{last1}}
    \fint_{B^+_{\varrho / 4}(\tilde{x}_0)} (\tilde{A}(\tilde{x}_0, \nabla\tilde{u})-\tilde{A}(x, \nabla\tilde{u})) \cdot(\nabla \tilde{u}-\nabla \tilde{h})dx\leq C_\lambda \varrho^{\bar{\alpha}-\lambda p},
\end{equation}
for every $\lambda>0$. For, the second term, we use \cref{growth6} and \cref{oscila} to get
\begin{equation}{\label{last2}}
    \fint_{B^+_{\varrho / 4}(\tilde{x}_0)} \tilde{f} |(\tilde{u}-\tilde{h})|d x
   \leq C \|\tilde{u}-\tilde{h}\|_{L^\infty(B^+_{\varrho/ 4}(\tilde{x}_0))} \fint_{B^+_{\varrho / 4}(\tilde{x}_0)} \tilde{f} dx\leq C_{\tilde{\eta}}\varrho^{\tilde{\eta}}\fint_{B^+_{\varrho / 4}(\tilde{x}_0)} x_n^{-\sigma} dx\leq C_{\tilde{\eta},\sigma}\varrho^{\tilde{\eta}-\sigma},
\end{equation}for every $\tilde{\eta}\in(\sigma, 1)$. the third term can be broken as $|\mathrm{(III)}|=J_1+J_2$ where
\begin{equation*}
J_1:=\fint_{B_{\varrho / 2}(\tilde{x}_0)} \int_{B_{\varrho / 2}(\tilde{x}_0)}\frac{|\tilde{u}(x)-\tilde{u}(y)|^{q-1}|(\tilde{u}-\tilde{h})(x)-(\tilde{u}-\tilde{h})(y)|}{|x-y|^{n+s q}} d x d y,\end{equation*} and \begin{equation*}
J_2:=\int_{\mathbb{R}^n \backslash B_{\varrho/ 2}(\tilde{x}_0)}\left(\fint_{B_{\varrho / 2}(\tilde{x}_0)} \frac{|\tilde{u}(x)-\tilde{u}(y)|^{q-1}|(\tilde{u}-\tilde{h})(x)|}{|x-y|^{n+s q}} d x\right) d y .
\end{equation*}
The term $J_1$ can be estimated as the term $\mathrm{(III)}$ in \cref{higher2} using \cref{fracemb}, \cref{grad1}, \cref{gra} and yields
\begin{equation}{\label{last3}}
    J_1\leq C_{\eta,\lambda} \varrho^{(\eta-s)(q-1)-\lambda},
\end{equation}
for every $\eta\in(s,1)$ and $\lambda>0$. 
Finally, we estimate $J_2$. Since $\tilde{u}-\tilde{h}$ is supported in $B_{\varrho / 4}^{+}(\tilde{x}_0)$ and it holds
$
    {|y-\tilde{x}_0|}/{|x-y|} \leq 2 \text { for every } x \in B_{\varrho / 4}^{+}(\tilde{x}_0) \text { and } y \in \mathbb{R}^n \backslash B_{\varrho / 2}(\tilde{x}_0)
$
Therefore, we have
\begin{equation}{\label{I_4}}
    \begin{array}{rcl}
         J_2&\leq & C\int_{\mathbb{R}^n \backslash B_{\varrho / 2}(\tilde{x}_0)}\left(\fint_{B_{\varrho/ 4}^{+}(\tilde{x}_0)} \frac{\left(| \tilde{u}(x)-(\tilde{u})_{B_{\varrho / 2}(\tilde{x}_0)}|^{q-1}+| \tilde{u}(y)-(\tilde{u})_{B_{\varrho / 2}(\tilde{x}_0)}|^{q-1}\right)|\tilde{u}(x)-\tilde{h}(x)|}{|y-\tilde{x}_0|^{n+s q}} d x\right) d y\smallskip\\&\leq &C \varrho^{-s q} \fint_{B_{\varrho / 4}^{+}(\tilde{x}_0)}|\tilde{u}(x)-(\tilde{u})_{B_{\varrho / 2}(\tilde{x}_0)}|^{q-1}|\tilde{u}(x)-\tilde{h}(x)| d x \smallskip\\&&+C\left(\int_{\mathbb{R}^n \backslash B_{\varrho / 2}(\tilde{x}_0)} \frac{|\tilde{u}(y)-(\tilde{u})_{B_{\varrho/ 2}(\tilde{x}_0)}|^{q-1}}{|y-\tilde{x}_0|^{n+s q}} d y\right)\left(\fint_{B_{\varrho / 4}^{+}(\tilde{x}_0)}|\tilde{u}(x)-\tilde{h}(x)| d x\right).
    \end{array}
\end{equation}
By Hölder's inequality, \cref{globalholder1} and \cref{oscila} we obtain
\begin{equation*}
    \begin{array}{l}
\quad\varrho^{-s q} \fint_{B_{\varrho / 4}^{+}(\tilde{x}_0)}|\tilde{u}(x)-(\tilde{u})_{B_{\varrho / 2}(\tilde{x}_0)}|^{q-1}|\tilde{u}(x)-\tilde{h}(x)| d x \\
 \leq C \varrho^{-s q}\left(\fint_{B_{\varrho / 4}^{+}(\tilde{x}_0)}|\tilde{u}(x)-(\tilde{u})_{B_{\varrho / 2}(\tilde{x}_0)}|^{q}\right)^{\frac{q-1}{q}}\left(\fint_{B_{\varrho / 4}^{+}(\tilde{x}_0)}|\tilde{u}(x)-\tilde{h}(x)|^q d x\right)^{\frac{1}{q}}\smallskip \\
 \leq C \varrho^{-s q+\eta(q-1)}[\tilde{u}]_{C^\eta(B_{\varrho / 2}(\tilde{x}_0))}^{q-1}\|\tilde{u}(x)-\tilde{h}(x)\|_{L^{\infty}(B_{\varrho/ 4}^{+}(\tilde{x}_0))} \leq C_\eta \varrho^{(\eta-s) q},
\end{array}
\end{equation*}
whereas, by Hölder's inequality, \cref{Ming}, \cref{grad2} and \cref{oscila}, we get
\begin{equation*}
    \begin{array}{l}
         \quad\left(\int_{\mathbb{R}^n \backslash B_{\varrho / 2}(\tilde{x}_0)} \frac{|\tilde{u}(y)-(\tilde{u})_{B_{\varrho/ 2}(\tilde{x}_0)}|^{q-1}}{|y-\tilde{x}_0|^{n+s q}} d y\right)\left(\fint_{B_{\varrho / 4}^{+}(\tilde{x}_0)}|\tilde{u}(x)-\tilde{h}(x)| d x\right)\smallskip\\\leq C \varrho^{-s}\left(\int_{\mathbb{R}^n \backslash B_{\varrho / 2}(\tilde{x}_0)} \frac{|\tilde{u}(y)-(\tilde{u})_{B_{\varrho/ 2}(\tilde{x}_0)}|^{q}}{|y-\tilde{x}_0|^{n+s q}} d yd y\right)^{\frac{q-1}{q}}\|\tilde{u}(x)-\tilde{h}(x)\|_{L^{\infty}(B_{\varrho / 4}^{+}(\tilde{x}_0))} \leq C_\eta\varrho^{\eta-s}.
    \end{array}
\end{equation*}
Merging the estimates in the last two displays and using in \cref{I_4}, we get 
\begin{equation}{\label{last4}}
    J_2\leq C_\eta \varrho^{\eta-s},
\end{equation}
for every $\eta\in(s,1)$. We now merge \cref{last1,last2,last3,last4} to get from \cref{diffbdd1}
\begin{equation}{\label{last5}}
     \fint_{B^+_{\varrho / 4}(\tilde{x}_0)}|V_\mu(\nabla \tilde{u})-V_\mu(\nabla \tilde{h})|^2 dx\leq C_{\eta,\tilde{\eta},\lambda,\sigma}(\varrho^{\bar{\alpha}-\lambda p}+\varrho^{\tilde{\eta}-\sigma}+ \varrho^{(\eta-s)(q-1)-\lambda}+\varrho^{\eta-s}),
\end{equation}
for every $\eta \in(s, 1)$, $\tilde{\eta}\in(\sigma,1)$ and $\lambda>0$. We now choose the constants $\lambda$, $\tilde{\eta}$ and $\eta$ as follows:
\begin{equation*}
    \eta:=\frac{1+s}{2}, \quad \tilde{\eta}=\frac{1+\sigma}{2} \quad \text { and } \quad \lambda:=\min \left\{\frac{\bar{\alpha}}{2 p},  \frac{(1-s)(q-1)}{4}\right\},
\end{equation*}
so that \cref{last5} becomes just
\begin{equation}{\label{last6}}
     \fint_{B^+_{\varrho / 4}(\tilde{x}_0)}|V_\mu(\nabla \tilde{u})-V_\mu(\nabla \tilde{h})|^2 dx\leq C_\sigma\varrho^{\sigma_0 p},
\end{equation}
with $\sigma_0:=\frac{1}{p} \min \left\{\frac{\bar{\alpha}}{2}, \frac{(1-s)(q-1)}{4}, \frac{1-s}{2},\frac{1-\sigma}{2}\right\}$.
We are now in position to conclude, using \cref{last6} in combination with \cref{g1}. When $p \geq 2$, estimate \cref{gradientboundaryregu} follows immediately with $\bar{\sigma}=\sigma_0$. On the other hand, when $p \in(1,2)$ we estimate using \cref{grad3,bdh1}, obtaining
\begin{equation*}
\begin{array}{rcl}
    \fint_{B_{\varrho / 4}^{+}(\tilde{x}_0)}|\nabla \tilde{u}-\nabla\tilde{h}|^p d x &\leq& C\left( \fint_{B^+_{\varrho / 4}(\tilde{x}_0)}|V_\mu(\nabla \tilde{u})-V_\mu(\nabla \tilde{h})|^2 dx\right)^{p/2}\left(\fint_{B_{\varrho / 4}^{+}(\tilde{x}_0)}(|\nabla \tilde{u}|^2+|\nabla\tilde{h}|^2+\mu^2)^{p/2}\right)^{\frac{2-p}{2}}\smallskip\\&\leq& C_\lambda \varrho^{\frac{\sigma_0 p-(2-p) \lambda}{2} p} ,
    \end{array}
\end{equation*}
 for every $\lambda>0$. Therefore, by choosing $\lambda:=\frac{\sigma_0 p}{2(2-p)}$, we obtain the desired estimate \cref{gradientboundaryregu} with $\bar{\sigma}=\frac{\sigma_0 p}{4}$.
\end{proof}
\subsection*{Proof of \cref{holder4}} Once \cref{gradientboundaryregu} is established, one can follow the same calculations of the proof of \cref{holder3} using \cref{av2,bdh1,bdh4,grad3}. Let $t \in\left(0, \frac{\varrho}{8}\right]$, with $\varrho \in(0, \frac{r_0}{4}]$. For every $\tilde{x}_0 \in \Gamma_{r_0 / 2}$, we have
\begin{equation*}
    \fint_{B_t^{+}(\tilde{x}_0)}|\nabla \tilde{u}-(\nabla \tilde{u})_{B_t^{+}(\tilde{x}_0)}|^p d x\leq C_{\sigma,\lambda}\left\{\left(\frac{\varrho}{t}\right)^n \varrho^{\bar{\sigma} p}+\left(\frac{t}{\varrho}\right)^{\bar{\alpha}_0 p} \varrho^{-\lambda p}\right\},
\end{equation*}
for every $\lambda>0$. By choosing $t:=\frac{\varrho^{1+\frac{\bar{\sigma}p}{2 n}}}{8}$ and $\lambda:=\frac{\bar{\alpha}_0 \bar{\sigma} p}{4 n}$, we then obtain
\begin{equation*}
    \fint_{B_t^{+}(\tilde{x}_0)}|\nabla \tilde{u}-(\nabla \tilde{u})_{B_t^{+}(\tilde{x}_0)}|^p d x\leq C t^{\gamma p} \quad \text { for every } t \in\left(0, \varrho_0\right),
\end{equation*}
with $\varrho_0:=\frac{1}{8}\left(\frac{r_0}{4}\right)^{1+\frac{\bar{\sigma}p}{2 n}}$ and $\gamma:=\min \left\{\frac{n \bar{\sigma}}{2 n+\bar{\sigma} p}, \frac{\bar{\alpha}_0 \bar{\sigma} p}{2(2 n+\bar{\sigma} p)}\right\}$. This concludes up to relabeling $t$ as $\varrho$
\begin{equation}{\label{finallast}}
\sup _{\tilde{x}_0 \in \Gamma_{r_0 / 2}}  \fint_{B_\varrho^{+}(\tilde{x}_0)}|\nabla \tilde{u}-(\nabla \tilde{u})_{B_\varrho^{+}(\tilde{x}_0)}|^p d x \leq C \varrho^{\gamma p}, \quad \forall \varrho \in(0, \varrho_0), \text { for some }\varrho_0\in(0,1).
\end{equation}
By combining the interior Campanato estimate of [\citealp{MG}, Theorem 3, Theorem 5] and the boundary estimate \cref{finallast}, the result follows via a standard covering argument and Campanato's characterization of Hölder continuity.
\subsection*{Proof of \cref{prtb}}
 As for $p\geq n$, embedding results imply $u\in L^t(\Omega)$ for all $t\in[1,\infty)$ and hence $f\in L^d(\Omega)$ for some $d>n$, then the gradient regularity follows from [\citealp{antonini}, Theorem 1.1]. Therefore, without loss of generality, we assume $p<n$. The idea is to show $u\in L^\infty(\Omega)$. By [\citealp{MG}, Proposition 2.1], it is enough to show that $f\in L^t(\Omega)$ for some $t>n/p$. Our method is by ideas found in \cite{Guedda}.
\smallskip \\\textbf{Step 1:} For $k>0$ and $t \geq p$ we define two $C^1$ functions $h$ and $\phi$ on $\mathbb{R}$ by
\begin{equation}{\label{hhh}}
    h(r)= \begin{cases}\operatorname{sign}(r)|r|^{t / p} & \text { for }|r| \leq k, \\ \operatorname{sign}(r)\{(t / p) k^{t / p-1}|r|+(1-t / p) k^{t / p}\} & \text { for }|r|>k,\end{cases}
\end{equation}
and \begin{equation}{\label{phii}}
    \phi(r)=\int_0^r\left(h^{\prime}(\nu)\right)^p d\nu.
\end{equation} Clearly $\phi$ is non-decreasing as $(h^{\prime})^p$ is non-negative. As $\phi^{\prime}$ is bounded, $\phi(u) \in W_0^{1, p}(\Omega)$ whenever $u\in W^{1,p}_0(\Omega)$. Let $v\in W^{1,p}_0(\Omega)$ be any weak solution to 
\begin{equation*}
        -\Delta_p v+(-\Delta)^s_q v+K(x)|v|^{p-2}v=F(x) \text{ in }\Omega; \quad v=0\text{ in }\mathbb{R}^n\backslash\Omega,
    \end{equation*} 
    where $K,F\in L^{n/p}(\Omega)$. We show $v\in L^t(\Omega)$ for every $t\in [1,\infty)$. Taking $\phi(v)$ as a test function in the weak formulation, we get
\begin{equation}{\label{coroo1}}
\begin{array}{l}
       \int_\Omega|\nabla v|^p(h^{\prime}(v))^p dx+\int_{\mathbb{R}^{n}}\int_{\mathbb{R}^{n}}\frac{|v(x)-v(y)|^{q-2}(v(x)-v(y))(\phi(v(x))-\phi(v(y)))}{|x-y|^{n+sq}}dxdy\\\qquad+\int_\Omega K \phi(v) v|v|^{p-2}dx =\int_\Omega F \phi(v)dx  .
\end{array}
\end{equation}
By [\citealp{parini}, Lemma A.2]
\begin{equation*}
    \int_{\mathbb{R}^{n}} \int_{\mathbb{R}^{n}}\frac{|v(x)-v(y)|^{q-2}(v(x)-v(y))(\phi(v(x))-\phi(v(y)))}{|x-y|^{n+sq}}dxdy\geq   \int_{\mathbb{R}^{n}}\int_{\mathbb{R}^{n}}\frac{|G(u(x))-G(v(y))|^{q}}{|x-y|^{n+sq}}dxdy\geq 0,
\end{equation*}
where $G(t)=\int_0^t\phi^\prime(\tau)^{\frac{1}{q}}d\tau$. Therefore, \cref{coroo1} implies 
\begin{equation}{\label{hol72}}
\begin{array}{l}
       \int_\Omega|\nabla v|^p(h^{\prime}(v))^p dx+\int_\Omega K \phi(v) v|v|^{p-2}dx \leq\int_\Omega F \phi(v)dx  .
\end{array}
\end{equation}
For $m>0$ set $\Omega_m=\{x \in \Omega:|K(x)|>m\}$. We then have
\begin{equation}{\label{hol73}}
\left|\int_\Omega K \phi(v) v|v|^{p-2}dx\right|\leq m  \int_{\Omega \backslash \Omega_m}|v|^{p-1}|\phi(v)|d x +\|K \|_{L^{n/ p}(\Omega_m)}\left(\int_\Omega| | v|^{p-1} \phi(v)|^{n/(n-p)} d x\right)^{\frac{n-p}{n}}   . 
\end{equation}
Moreover $\int_\Omega|\nabla v|^p(h^{\prime}(v))^p d x=\int_\Omega|\nabla(h(v))|^p dx \geqq S\left(\int_\Omega|h(v)|^{p^*} d x\right)^{p/{p *}}$. 
From the definition of $h$ and $\phi$ there exists $C$ independent of $k$ such that
\begin{equation}{\label{hol74}}
    ||r|^{p-1} \phi(r)| \leq C|h(r)|^p, \quad|\phi(r)| \leq C|h(r)|^{p(t+1-p) / t},
\end{equation}
for any $r$. If we choose $m$ such that $\|K\|_{L^{n / p}(\Omega_m)} \leq \frac{S}{2 C}$, we obtain from \cref{hol72} using \cref{hol73,hol74}
\begin{equation}{\label{hol75}}
    \frac{S}{2}\left(\int_\Omega|h(v)|^{p^*} d x\right)^{p/{p *}} \leq C m \int_\Omega|h(v)|^p \mathrm{~d} x+C \|F\|_{L^{n / p}(\Omega)}\left(\int_\Omega|h(v)|^{p^*(t+1-p) / t} d x\right)^{p / p^*}
\end{equation}
and
\begin{equation}{\label{hol76}}
\left(\int_\Omega|h(v)|^{p^*(t+1-p) / t} d x\right)^{p / p^*} \leq|\Omega|^{p(p-1) / p^*t}\left(\int_\Omega|h(v)|^{p^*} d x\right)^{(t+1-p) p / t p^*} .    
\end{equation}
If we assume now that $v \in L^t(\Omega)$ then $v\in L^{t n/(n-p)}(\Omega)$ (by \cref{hol75,hol76}) and
\begin{equation*}
    \frac{S}{2}\|v\|_{L^{tn /(n-p)}(\Omega)}^t \leq C m\|v\|_{L^t(\Omega)}^t+C|\Omega|^{p(p-1) / p^* t}\|F\|_{L^{n/ p}(\Omega)}\|v\|_{L^{tn/ n-p}(\Omega)}^{t+1-p} .
\end{equation*}
With Young's inequality, we obtain
\begin{equation*}
   \|v\|_{L^{tn /(n-p)}(\Omega)} \leq C_1\|v\|_{L^t(\Omega)}+C_2\|F\|_{L^{n /p}(\Omega)}^{\frac{1}{p-1}},
\end{equation*}
where $C_1$ and $C_2$ depend on $m, t, n, p$ and $|\Omega|$. Define now
\begin{equation*}
    t_0=p, \quad t_l=\left(\frac{n}{n-p}\right)^{l} p, \quad l \in \mathbb{N}^*
\end{equation*}
and get $v \in L^{t_l}(\Omega)$ for any $l \in \mathbb{N}^*$. Therefore $v\in L^t(\Omega)$ for all $t\in[1,\infty)$.
\smallskip\\\textbf{Step 2:} Set $
    K(x)=\operatorname{sign}(u) f(x, u) /(1+|u|^{p-1}) .$ Then $K$ satisfies
\begin{equation*}
    |K(x)| \leq C_0(1+|u|^{p^*-1}) /(1+|u|^{p-1}) \leq M_1|u|^{p^*-p}+M_2,
\end{equation*}
$M_1, M_2$ being constants. As $p^*-p=p^2 /(n-p)$ and $u \in L^{p^*}(\Omega)$ we deduce that $K \in L^{n/ p}(\Omega)$. Then \cref{unif} can be written in the form\begin{equation*}
    -\Delta_p u+(-\Delta )^s_q u=K(x)|u|^{p-2} u+\operatorname{sign}(u) K(x) .
\end{equation*}
By step 1, we deduce that $u \in \bigcap_{1 \leq t<+\infty} L^t(\Omega)$. Hence $K|u|^{p-2} u+\operatorname{sign}(u) K \in \bigcap_{1 \leq t<+\infty} L^t(\Omega)$, and from [\citealp{MG}, proposition 2.1] we deduce that $u \in L^{\infty}(\Omega)$. Finally by [\citealp{antonini}, Theorem 1.1] we get the $C^{1, \gamma}$ regularity of $u$ in $\overline{\Omega}$.
\section{Singular problems}{\label{singu}}
We now turn to singular problems. For this, first we consider the general problem for $1<q\leq p<\infty$
\begin{equation}{\label{problemg}}
\begin{array}{c}
\mathcal{L}u=f(x)u^{-\delta}\text { in } \Omega, u>0 \text{ in } \Omega,u=0  \text { in }\mathbb{R}^n \backslash \Omega;
\end{array}
\end{equation}
where $\mathcal{L}$ is given by \cref{general} and satisfies the growth conditions \cref{growth1,growth2,growth3,grrrr} with $\bar{\alpha}=1,\mu=0$ in \cref{growth1}, $\delta>0$ and $f$ satisfies the condition \cref{condf} for some $\beta\geq 0$. We show the existence of weak solutions and then take the particular operator $\mathcal{L}u=-\Delta_pu+(-\Delta)_q^s u$ to prove \cref{mainexis} to \cref{holder6}. As a result of \cref{growth1} it holds (see [\citealp{MG}, section 2.3]) for some $c>0$
\begin{equation}{\label{growth8}}
    |z|^p\leq cA(x,z)\cdot z
\end{equation}
for all $x\in\Omega$ and $z\in \mathbb{R}^n$. We follow the approximation method in \cite{orsina}. To this end, we investigate the approximating problem
\begin{equation}{\label{approximated}}
    -\operatorname{div} A(x, \nabla u(x))  +\text { P.V. } \int_{\mathbb{R}^n} \Phi(u(x)-u(y)) \frac{B(x, y)}{|x-y|^{n+s q}} d y = f_\ep(x)(u^{+}(x)+\ep)^{-\delta} \text { in } \Omega,\quad u=0 \text { on } \mathbb{R}^n\backslash\Omega,
\end{equation}
where \begin{equation*}
    f_\epsilon(x):= \begin{cases}\left(f(x)^{\frac{-1}{\beta}}+\epsilon ^\frac{p-1+\delta}{p-\beta}\right)^{-\beta} & \text { if } f(x)>0 \\ 0 & \text { otherwise. }\end{cases}
\end{equation*}
It is easy to observe that, for $\beta<p$, the function $f_\epsilon$ increases as $\epsilon \downarrow 0$ and $f_\epsilon \leq f$ for all $\epsilon>0$.
\begin{lemma}{\label{aproexist}}
  (a)  For every $\ep>0$, problem \cref{approximated} admits a positive weak solution $u_\ep \in W_0^{1, p}(\Omega) \cap L^{\infty}(\Omega)$.
(b) For $\beta<p$, the $u_\ep$ increases as $\epsilon \downarrow 0$ and $u_\ep$ is unique for every $\ep>0$.
(c) For every $\Omega^{\prime} \subset\subset \Omega$, there exists a positive constant $C_{\Omega^{\prime}}$ such that $u_\ep \geq C_{\Omega^{\prime}}>0$ in $\Omega^{\prime},\,\forall\ep\in(0,1]$. More specifically, $u_\ep \geq Cd(x,\partial\Omega)>0$ in $\Omega^{\prime},\, \forall\ep\in(0,1]$.
\end{lemma} 
\begin{proof}
Let $V=W_0^{1, p}(\Omega)$ with the norm $\|u\|=\|\nabla u\|_{L^{ p}(\Omega)}
$ and let $V^{\prime}$ be the dual of $V$. Note [\citealp{MG}, Section 2.5] and for every $\ep>0$, define $T: V \rightarrow V^{\prime}$ by
\begin{equation*}
\begin{array}{rcl}
    \langle T(v), \phi\rangle&=&\int_\Omega A(x, \nabla v(x)) \cdot \nabla \phi(x) d x+\int_{\mathbb{R}^n} \int_{\mathbb{R}^n} \frac{|v(x)-v(y)|^{q-2}(v(x)-v(y))(\phi(x)-\phi(y)) }{|x-y|^{n+sq}}d xdy\\&&-\int_\Omega f_\ep(x)(v^{+}(x)+\ep)^{-\delta} \phi(x) d x,
    \end{array}
\end{equation*}for every $v, \phi \in V$. By \cref{growth1}, $T$ is \textbf{well defined}. Using \cref{growth1}, Hölder's inequality, \cref{embedding2} and \cref{Sobolev embedding},
\begin{equation*}
    |\langle T(v), \phi\rangle|  \leq C \int_{\Omega}|\nabla v|^{p-1}|\nabla \phi| d x+C[v]_{s, q}^{q-1}[\phi]_{s, q}+C(\ep)\int_{\Omega}|\phi| d x \leq C(\|v\|^{p-1}\|\phi\|+[v]_{s, q}^{q-1}\|\phi\|+\|\phi\|) \leq C|\phi|,
\end{equation*}
for some constant $C$ depending on $\ep$. Now by Hölder's inequality, \cref{growth8} and \cref{Sobolev embedding}, we obtain
\begin{equation*}
    \langle T(v), v\rangle  =\int_\Omega A(x, \nabla v) \cdot \nabla v \,d x+\int_{\mathbb{R}^n} \int_{\mathbb{R}^n} \frac{|v(x)-v(y)|^{q} }{|x-y|^{n+sq}}d xdy-\int_\Omega f_\ep(v^{+}(x)+\ep)^{-\delta} v \,d x \geq C_1\|v\|^p-C\|v\| .
\end{equation*}
Since $1<p<\infty$, we may conclude that $T$ is \textbf{coercive}. We now show $T$ is demicontinuous. Let $v_k \rightarrow v$ in the norm of $V$ as $k \rightarrow \infty$. Then $\nabla v_k \rightarrow \nabla v$ in $L^p(\Omega)$ as $k \rightarrow \infty$. Therefore, up to a subsequence, still denoted by $v_k$, we have $v_k(x) \rightarrow v(x)$ and $\nabla v_k(x) \rightarrow \nabla v(x)$ pointwise for almost every $x \in \Omega$. Since the function $A(x, \cdot)$ is continuous in the second variable, we have
$
A(x, \nabla v_k) \rightarrow A(x, \nabla v)
$
pointwise for almost every $x \in \Omega$ as $k \rightarrow \infty$. By \cref{growth1} and the norm boundedness of the sequence $\{v_k\}_{k \in \mathbb{N}}$, we obtain
\begin{equation*}
    \|A(x, \nabla v_k)\|_{L^{p^{\prime}}(\Omega)}^{p^{\prime}}=\int_{\Omega}|A(x, \nabla v_k)|^{p^{\prime}} d x \leq C\int_{\Omega}|\nabla v_k(x)|^p d x \leq C,
\end{equation*}
for some positive constant $C$ independent of $k$. Therefore, up to a subsequence
$
A(x, \nabla v_k(x))\rightharpoonup A(x, \nabla v(x))
$
weakly in $L^{p^\prime}(\Omega)$ as $k \rightarrow \infty$, and since the weak limit is independent of the choice of the subsequence, the above weak convergence actually holds for the original sequence. As $\phi \in V$ implies that $\nabla \phi \in L^p(\Omega)$, by the weak convergence, we obtain
\begin{equation}{\label{harpoo}}
    \lim _{k \rightarrow \infty} \int_\Omega A(x, \nabla v_k(x)) \cdot \nabla \phi(x) d x=\int_\Omega A(x, \nabla v(x)) \cdot \nabla \phi(x) d x .
\end{equation}
Since $\{v_k\}_{k \in \mathbb{N}}$ is a bounded sequence in $V$ and $q\leq p$, \cref{embedding2} implies that $\left\{\frac{|v_k(x)-v_k(y)|^{q-2}(v_k(x)-v_k(y))}{|x-y|^{\frac{(n+sq)(q-1)}{q}}}\right\}$ is a bounded sequence in $L^{q^{\prime}}(\mathbb{R}^n \times \mathbb{R}^n)$. Since $\phi \in V$, we have
$
(\phi(x)-\phi(y)) |x-y|^{-\frac{n+sq}{q}} \in L^q(\mathbb{R}^n\times \mathbb{R}^n).
$
Therefore, by the weak convergence, we have
\begin{equation}{\label{harpoo1}}
\begin{array}{c}
      \lim _{k \rightarrow \infty} \int_{\mathbb{R}^n} \int_{\mathbb{R}^n} \frac{|v_k(x)-v_k(y)|^{q-2}(v_k(x)-v_k(y))(\phi(x)-\phi(y)) }{|x-y|^{n+sq}}d xdy\\=\int_{\mathbb{R}^n} \int_{\mathbb{R}^n}  \frac{|v(x)-v(y)|^{q-2}(v(x)-v(y))(\phi(x)-\phi(y)) }{|x-y|^{n+sq}}d xdy.
\end{array}
\end{equation}
Finally, as $v_k(x) \rightarrow v(x)$ pointwise almost everywhere in $\Omega$ and
\begin{equation*}
    \left|f_\ep(x)(v_k^{+}(x)+\ep)^{-\delta} \phi(x)-f_\ep(x)(v^{+}(x)+\ep)^{-\delta} \phi(x)\right| \leq C(\ep)|\phi(x)|,
\end{equation*}
for every $\phi \in V$. By the dominated convergence theorem, we have
\begin{equation}{\label{harpoo2}}
    \lim _{k \rightarrow \infty} \int_{\Omega} f_\ep(x)(v_k^{+}(x)+\ep)^{-\delta} \phi(x)=\int_{\Omega}f_\ep(x)(v^{+}(x)+\ep)^{-\delta} \phi(x),
\end{equation}
for every $\phi \in V$ and $\ep>0$. Merging \cref{harpoo,harpoo1,harpoo2} it follows that
\begin{equation*}
    \lim _{k \rightarrow \infty}\left(T\left(v_k\right), \phi\right\rangle=\langle T(v), \phi\rangle,
\end{equation*}
for every $\phi \in V$, and hence $T$ is \textbf{demicontinuous}. Finally, we show the monotonicity of $T$. Let $v_1, v_2 \in V$. Using \cref{p case} and \cref{g2} it is easy to observe that
\begin{equation*}
    \left\langle T(v_1)-T(v_2), v_1-v_2\right\rangle\geq -\int_{{\Omega}} f_\ep(x)\left((v_1^{+}(x)+\ep)^{-\delta}-(v_2^{+}(x)+\ep)^{-\delta}\right)(v_1(x)-v_2(x)) d x.
\end{equation*}
One can easily check that the integral present on the right is nonpositive. This shows that $T$ is \textbf{monotone}. As a result, by [\citealp{ciarlet}, Theorem 9.14], we conclude $T$ is surjective. Hence for every $\ep>0$, there exists $u_\ep \in V$ such that
\begin{equation*}
    \int_\Omega A(x, \nabla u_\ep) \cdot \nabla \phi\, d x+\int_{\mathbb{R}^n} \int_{\mathbb{R}^n} \frac{|u_\ep(x)-u_\ep(y)|^{q-2}(u_\ep(x)-u_\ep(y))(\phi(x)-\phi(y)) }{|x-y|^{n+sq}}d xdy=\int_\Omega f_\ep(u_\ep^{+}+\ep)^{-\delta} \phi\,d x,
\end{equation*}
for every $\phi \in V$. This shows that \cref{approximated} has a weak solution $u_\ep \in W_0^{1, p}(\Omega)$ for every $\ep>0$. Using \cref{growth8} and proceeding along the lines of the proof of [\citealp{garain}, Lemma 3.1], it follows that $u_\ep\in L^{\infty}(\Omega)$ for every $\ep>0$. Now let
$    g_\ep(x)=f_\ep(x)(u_\ep^+(x)+\ep)^{-\delta}.
$
Choosing $\phi=\min \{u_\ep, 0\}=(u_\ep)_-$ as a test function and using the non-negativity of $g_\ep$, we have by \cref{growth8}
\begin{equation*}
     C\int_\Omega|\nabla(u_\ep)_-)|^p d x+\int_{\mathbb{R}^n} \int_{\mathbb{R}^n}  \frac{|u_\ep(x)-u_\ep(y)|^{q-2}(u_\ep(x)-u_\ep(y))((u_\ep(x))_{-}-(u_\ep(y))_-) }{|x-y|^{n+sq}}d xdy=\int_\Omega g_\ep (u_\ep)_{-} d x \leq 0 .
\end{equation*}
Proceeding similarly to the proof of the estimate [\citealp{garain}, (3.13)], the second integral above is nonnegative. Hence
\begin{equation*}
    \int_\Omega|\nabla(u_\ep)_-)|^p d x=0,
\end{equation*}
which gives $u_\ep\geq 0$ in $\Omega$. \\
To show the monotonicity, let $0<\ep^\prime\leq \ep$, we choose $\phi=(u_\ep-u_{\ep^\prime})^{+}$ as a test function in \cref{approximated} to have
\begin{equation}{\label{harpoo4}}
    \begin{array}{l}
         \int_\Omega(A(x, \nabla u_\ep(x))-A(x, \nabla u_{\ep^\prime}(x))) \cdot \nabla(u_\ep(x)-u_{\ep^\prime}(x))^{+} d x\smallskip\\+\int_{\mathbb{R}^n} \int_{\mathbb{R}^n}(\mathcal{A} (u_\ep(x, y))-\mathcal{A}(u_{\ep^\prime}(x, y)))((u_\ep(x)-u_{\ep^\prime}(x))^{+}-(u_\ep(y)-u_{\ep^\prime}(y))^{+}) d \mu\smallskip\\=\int_\Omega(f_\ep(x)(u_\ep(x)+\ep)^{-\delta}-f_{\ep^\prime}(x)(u_{\ep^\prime}(x)+\ep^\prime)^{-\delta})(u_\ep(x)-u_{\ep^\prime}(x))^{+} d x \smallskip\\
\leq \int_\Omega f_{\ep^\prime}(x)((u_\ep(x)+\ep)^{-\delta}-(u_{\ep^\prime}(x)+\ep^\prime)^{-\delta})(u_\ep(x)-u_{\ep^\prime}(x))^{+} d x 
 \leq 0,
    \end{array}
\end{equation}
where $\mathcal{A}(g(x,y))=|g(x)-g(y)|^{q-2}(g(x)-g(y))$, $d\mu=\frac{dxdy}{|x-y|^{n+sq}}$ and the last inequality above follows using $f_\ep \leq f_{\ep^\prime}$ for $\beta<p$. Following the same arguments from the proof of [\citealp{FractionalEigenvalues}, Lemma 9], we obtain 
\begin{equation}{\label{harpoo5}}
    \int_{\mathbb{R}^n} \int_{\mathbb{R}^n}(\mathcal{A} (u_\ep(x, y))-\mathcal{A}(u_{\ep^\prime}(x, y)))((u_\ep(x)-u_{\ep^\prime}(x))^{+}-(u_\ep(y)-u_{\ep^\prime}(y))^{+}) d \mu\geq 0.
\end{equation}
Merging \cref{harpoo4,harpoo5}, one gets
\begin{equation*}
      \int_\Omega(A(x, \nabla u_\ep(x))-A(x, \nabla u_{\ep^\prime}(x))) \cdot \nabla(u_\ep(x)-u_{\ep^\prime}(x))^{+} d x\leq 0.
\end{equation*}
Thus, by \cref{g1,g2}, we may conclude that
$
    \nabla u_\ep=\nabla u_{\ep^\prime} \text { in }\{x \in \Omega: u_\ep(x)>u_{\ep^\prime}(x)\} .
$
Hence, we have $u_{\ep^\prime} \geq u_\ep$ in $\Omega$. The uniqueness of $u_\ep$ follows similarly. Finally, by \cref{holder4} (note that the condition on $u$ to be less than distance function is not needed if $f$ is bounded in \cref{holder4}), it follows, $u_1\in C^{1,\gamma}(\overline{\Omega})$ for some $\gamma\in(0,1)$. Since $g_1 \neq 0$ in $\Omega$, by the strong maximum principle [\citealp{antonini}, Proposition 6.1], it holds $u_1>0$. Therefore, for every $\Omega^{\prime} \subset\subset \Omega$, there exists a positive constant $C=C(\Omega^{\prime})$ such that $u_1 \geq C(\Omega^{\prime})>0$ in $\Omega^{\prime}$. Using the monotonicity property, we get $u_\ep \geq u_1$ in $\Omega$, and thus, for every $\Omega^{\prime} \subset\subset \Omega$, there exists a positive constant $C(\Omega^{\prime})$ (independent of $\ep$ ) such that $u_\ep \geq C(\Omega^{\prime})>0$ in $\Omega^{\prime}$. In fact, by [\citealp{antonini}, Corollary 1.3], it holds $u_\ep \geq u_1\geq Cd(x,\partial\Omega)>0$ in $\Omega^{\prime},\, \forall\ep\in(0,1]$, where $C>0$ is a constant independent of $\ep$.
\end{proof}
From now on, we will denote the distance of $x\in\Omega$ from the boundary $\partial\Omega$ by $d(x)$ and in short, we write $d$.
\begin{lemma}{\label{boundedinspace}} Let $\beta\in [0,p)$. The sequence $\{u_\ep\}$ is bounded in $W_0^{1,p}(\Omega)$ if $\beta+\delta\leq 1$, and the sequence $\{u_\ep^\theta\}$ is bounded in $W_0^{1, p}(\Omega)$ if $\beta+\delta>1$, for some $\theta>1$. Also in the later case $\{u_\ep\}$ is uniformly bounded in $W^{1.p}_{\mathrm{loc}}(\Omega)$.
\end{lemma}
\begin{proof} For the case $\beta+\delta\leq 1$, testing the weak formulation of the \cref{approximated} by $u_\ep$, and using \cref{growth2,growth3,growth8,condf} together with the Hardy's inequality, we get
\begin{equation*}
    \begin{array}{rcl}
\|u_\ep\|^p_{W^{1,p}_0(\Omega)}&\leq &c\int_\Omega A(x, \nabla u_\ep) \cdot \nabla u_\ep\, d x+c\int_{\mathbb{R}^n} \int_{\mathbb{R}^n} \frac{|u_\ep(x)-u_\ep(y)|^{q}}{|x-y|^{n+sq}}d xdy\smallskip\\&=&c\int_\Omega f_\ep(u_\ep+\ep)^{-\delta} u_\ep\,d x \leq C\int_{\Omega} d^{1-\beta-\delta}\left(\frac{u_\ep}{d}\right)^{1-\delta} \leq C\|u_\ep\|_{W_0^{1,p}(\Omega)}^{1-\delta} \text {, }
    \end{array}
\end{equation*}
where $C$ is independent of $\ep$. We have also used the fact that $0\leq f_\ep\leq f$ for every $\ep>0$. For the remaining case, we note [\citealp{MG}, Section 2.5] and test the weak formulation of the problem by $u_\ep^{m}$, where $m=p(\theta-1)+1$ for some $\theta \geq 1$ (to be specified later) to get
\begin{equation*}
    \begin{array}{l}
\|u_\ep^\theta\|^p_{W^{1,p}_0(\Omega)}
\leq c\int_\Omega A(x, \nabla u_\ep) \cdot \nabla u_\ep^{m}\, d x
+c\int_{\mathbb{R}^{2n}} {\mathcal{A}u_\ep(x,y)(u_\ep^{m}(x)-u^{m}_\ep(y))}d\mu
=c\int_\Omega f_\ep(u_\ep+\ep)^{-\delta} u_\ep^{m}d x,
    \end{array}
\end{equation*}
where $\mathcal{A}(v(x,y))=|v(x)-v(y)|^{q-2}(v(x)-v (y))$, $d\mu=\frac{dxdy}{|x-y|^{n+sq}}$. We now use [\citealp{parini}, Lemma A.2] for
$
g(t):=t^{p(\theta-1)+1},$ $G(t):=\int_0^t g^{\prime}(\tau)^{1 / q} d \tau,
$ to get that the nonlocal term 
\begin{equation*}
\begin{array}{l}
     \int_{\mathbb{R}^n} \int_{\mathbb{R}^n} \frac{|u_\ep(x)-u_\ep(y)|^{q-2}(u_\ep(x)-u_\ep(y))(u_\ep^{p(\theta-1)+1}(x)-u^{p(\theta-1)+1}_\ep(y))}{|x-y|^{n+sq}}d xdy\\\geq \int_{\mathbb{R}^{n}}\int_{\mathbb{R}^{n}}\frac{|G(u_\ep(x))-G(u_\ep(y))|^{q}}{|x-y|^{n+sq}}dxdy\geq 0,
\end{array}
\end{equation*}
which implies, on account of Hölder's and Hardy's inequalities, 
\begin{equation*}
   \|u_\ep^\theta\|^p_{W^{1,p}_0(\Omega)}\leq C \int_{\Omega} d^{-\beta+\frac{p(\theta-1)+1-\delta}{\theta}}\left(\frac{u_\ep^\theta}{d}\right)^{\frac{p(\theta-1)+1-\delta}{\theta}} \leq C\|u_\ep^\theta\|_{W_0^{1,p}(\Omega)}^{\frac{p(\theta-1)+1-\delta}{\theta}} \text {, }
\end{equation*}
where
\begin{equation}{\label{thet}}
    \theta>\max \left\{1, \frac{(p+\delta-1)(p-1 )}{p(p -\beta)}, \frac{p+\delta-1}{p}\right\}
\end{equation}
and $C>0$ is independent of $\ep$. This concludes that the sequence $\{u_\ep^\theta\}$ is bounded in $W_0^{1, p}(\Omega)$. Follow now the arguments of [\citealp{garain}, Lemma 3.7] (taking $\delta+p-1/p=\theta$ there) to get $\{u_\ep\}$ is uniformly bounded in $W^{1.p}_{\mathrm{loc}}(\Omega)$.
\end{proof}
\begin{remark}{\label{u_0}} Let $\beta\in [0,p)$. In view of \cref{boundedinspace}, we can say that up to a subsequence $u_\epsilon \rightharpoonup u$ (or $u_\epsilon^\theta \rightharpoonup u^\theta$) weakly in $W^{1,p}_0(\Omega)$ and strongly in $L^p(\Omega)$ if $\beta+\delta\leq 1$ (or $\beta+\delta> 1$), as $\epsilon \rightarrow 0$, for some $u \in W^{1,p}_0(\Omega)$ (or $u^\theta \in W^{1,p}_0(\Omega)$). Also in each case $u_\ep\rightharpoonup u$ weakly in $W^{1,p}_{\mathrm{loc}}(\Omega)$. As $\{u_\ep\}$ is monotone, we can treat $u$ as the pointwise supremum of $\{u_\ep\}$ and hence $u\geq u_\ep\geq 0$, for every $\ep>0$. Further for any $\Omega^\prime\subset\subset\Omega$, there exists $C_{\Omega^\prime}>0$ such that $u(x)\geq u_\ep(x)\geq C_{\Omega^\prime}d(x)>0$ in $\Omega^\prime$ for all $\ep\in(0,1]$.
\end{remark}
We now give the result regarding the convergence of gradients of $u_\varepsilon$ to the gradient of $u$ a.e. in $\Omega$.
\begin{lemma}{\label{gradconv}}
Let $\beta\in [0,p)$. Take $u$ of \cref{u_0}. Then up to a subsequence, $\nabla u_\varepsilon\rightarrow\nabla u$ pointwise a.e. in $\Omega$.
\end{lemma}
\begin{proof}
   Let us take a compact $K \subset \Omega$ and consider a function $\phi_K \in C_c^\infty(\Omega)$ such that $\operatorname{supp} \phi_K=\omega$, $0 \leq \phi_K \leq 1$ in $\Omega$ and $\phi_K \equiv 1$ in $K$. Now for $\mu>0$, we define the truncated function $T_\mu: \mathbb{R} \rightarrow \mathbb{R}$ by
\begin{equation*}
T_\mu(t)= \begin{cases}t, & \text { if }|t| \leq \mu, \\ \mu \frac{t}{|t|}, & \text { if }|t|>\mu .\end{cases}
\end{equation*}
Now choose $\psi_\epsilon=\phi_K T_\mu((u_\epsilon-u)) \in W_0^{1, p}(\Omega)$ as a test function in \cref{approximated} and note [\citealp{MG}, Section 2.5], to get
\begin{equation}{\label{estii}}
    I+J=R,
\end{equation}
where
\begin{equation*}
I=\int_{\Omega}A(x,\nabla u_\ep) \cdot \nabla \psi_\epsilon\, d x, J=\int_{\mathbb{R}^{2n}}\frac{|u_\epsilon(x)-u_\epsilon(y)|^{q-2}(u_\epsilon(x)-u_\epsilon(y))(\psi_\epsilon(x)-\psi_\epsilon(y)) }{|x-y|^{n+sq}}d x d y ,
R=\int_{\Omega} \frac{f_\ep \psi_\epsilon}{(u_\epsilon+\epsilon)^\delta}  d x .
\end{equation*}
\textbf{Estimate of $I$:} We have
\begin{equation}{\label{esti1}}
    \begin{array}{rcl}
I&= & \int_{\Omega}A(x,\nabla u_\ep) \cdot \nabla \psi_\epsilon\, d x\smallskip \\
&= & \int_{\Omega} \phi_K(A(x,\nabla u_\ep)-A(x,\nabla u) )\cdot \nabla T_\mu((u_\ep-u)) d x +\int_{\Omega} \phi_KA(x,\nabla u) \cdot \nabla T_\mu((u_\ep-u))  d x\smallskip\\&&+\int_{\Omega} T_\mu((u_\ep-u)) A(x,\nabla u_\ep)\cdot\nabla \phi_K d x 
=:  I_1+I_2+I_3 .    
    \end{array}
\end{equation}
\textbf{Estimate of $I_2$:} As $u_\ep\rightharpoonup u$ weakly in $W^{1,p}_{\mathrm{loc}}(\Omega)$ we get $T_\mu((u_\ep-u)) \rightharpoonup 0$ weakly in $W_{\mathrm{loc }}^{1, p}(\Omega)$. As a consequence using \cref{growth1}, we get
\begin{equation}{\label{esti2}}
    \lim _{\ep \rightarrow 0} I_2= \lim _{\ep \rightarrow 0}\int_{\Omega} \phi_KA(x,\nabla u) \cdot \nabla T_\mu((u_\ep-u))  d x=0 .
\end{equation}
\textbf{Estimate of $I_3$:} For $\omega=\operatorname{supp} \phi_K$, by Hölder's inequality and \cref{growth1}, we use the uniform boundedness of $\{u_\ep\}$ in $W^{1, p}(\omega)$, to have
\begin{equation*}
    \begin{array}{rcl}
         \left|I_3\right|&=&\left|\int_{\Omega} T_\mu((u_\ep-u)) A(x,\nabla u_\ep)\cdot\nabla \phi_K d x\right|\leq C\int_{\Omega}| \nabla \phi_K| | u_\ep-u| | \nabla u_\ep|^{p-1} d x \smallskip\\
& \leq&\|\nabla \phi_K\|_{L^{\infty}(\omega)}\left(\int_\omega|\nabla u_\ep|^p d x\right)^{\frac{p-1}{p}}\|u_\ep-u\|_{L^p(\omega)} \leq C\|u_\ep-u\|_{L^p(\omega)} ,
    \end{array}
\end{equation*}
for some constant $C>0$, independent of $\ep$. As $u_\ep \to u$ strongly in $L^p_{\mathrm{loc}}(\Omega)$ so
\begin{equation}{\label{esti3}}
    \lim _{\ep\rightarrow 0} I_3=\lim _{\ep\rightarrow 0}\int_{\Omega} T_\mu((u_\ep-u)) A(x,\nabla u_\ep)\cdot\nabla \phi_K d x=0.
\end{equation}
Therefore, merging \cref{esti1}, \cref{esti2} and \cref{esti3} we obtain
\begin{equation}{\label{esti4}}
    \limsup _{\ep \rightarrow 0} I=\limsup _{\ep \rightarrow 0} I_1=\limsup _{\ep \rightarrow 0} \int_{\Omega} \phi_K(A(x,\nabla u_\ep)-A(x,\nabla u) )\cdot \nabla T_\mu((u_\ep-u)) d x.
\end{equation}
\textbf{Estimate of $J$:}
Denoting by $\mathcal{A}g(x,y)=|g(x)-g(y)|^{q-2}(g(x)-g(y))$ and $d\nu=\frac{dxdy}{|x-y|^{n+sq}}$, we can write $J$ as:
\begin{equation}{\label{esti5}}
    \begin{array}{rcl}
         J &= & \int_{\mathbb{R}^n} \int_{\mathbb{R}^n} \mathcal{A} u_\epsilon(x, y)(\psi_\epsilon(x)-\psi_\epsilon(y)) d \nu\\
&= & \int_{\mathbb{R}^n} \int_{\mathbb{R}^n} \phi_K(x)(\mathcal{A} u_\epsilon(x, y)-\mathcal{A} u(x, y))(T_\mu((u_\epsilon-u)(x))-T_\mu((u_\epsilon-u)(y))) d \nu \smallskip\\\
&& +\int_{\mathbb{R}^n} \int_{\mathbb{R}^n} T_\mu((u_\epsilon-u)(y)) \mathcal{A} u_\epsilon(x, y)(\phi_K(x)-\phi_K(y)) d \nu \\
&& +\int_{\mathbb{R}^n} \int_{\mathbb{R}^n} T_\mu((u_\epsilon-u)(y)) \mathcal{A} u(x, y)(\phi_K(y)-\phi_K(x)) d \nu \\
&& +\int_{\mathbb{R}^n} \int_{\mathbb{R}^n} \mathcal{A} u(x, y)(\phi_K(x) T_\mu((u_\epsilon-u)(x))-\phi_K(y) T_\nu((u_\epsilon-u)(y))) d \nu:= J_1+J_2+J_3+J_4 .
    \end{array}
\end{equation}
Using the monotonicity of the map $t\mapsto |t|^{q-2}t$ for $t\in\mathbb{R}^n$, one can estimate $J_1$ similarly as the homonym term in [\citealp{garain}, Theorem A.1] to get
\begin{equation}{\label{esti6}}
    J_1=\int_{\mathbb{R}^n} \int_{\mathbb{R}^n} \phi_K(x)(\mathcal{A} u_\epsilon(x, y)-\mathcal{A} u(x, y))(T_\mu((u_\epsilon-u)(x))-T_\mu((u_\epsilon-u)(y))) d \nu\geq 0.
\end{equation} Let us now set
$
\omega=\operatorname{supp} \phi_K, \omega^{\prime}=\mathbb{R}^{2 n} \backslash\left(\omega^c \times \omega^c\right),
$
and as $\phi_K \in C_c^\infty(\Omega)$, for every $\eta>0$, there exists a compact set $\mathcal{K}=\mathcal{K}(\eta) \subset \mathbb{R}^{2 n}$ such that
\begin{equation}{\label{A22}}
    \left(\int_{\mathbb{R}^{2 n} \backslash \mathcal{K}} \frac{|\phi_K(x)-\phi_K(y)|^q}{|x-y|^{n+qs}} d x d y\right)^{\frac{1}{q}} \leq \frac{\eta}{2} .
\end{equation}
Then, applying Hölder's inequality, we obtain
\begin{equation}{\label{esti7}}
\begin{array}{l}
      \quad\int_{\mathbb{R}^{2 n} \backslash \mathcal{K}} T_\mu((u_\epsilon-u)(y)) \mathcal{A} u_\epsilon(x, y)(\phi_K(x)-\phi_K(y)) d \nu\\\leq\left(\int_{\omega^{\prime} \backslash \mathcal{K}} \frac{|u_\ep(x)-u_\ep(y)|^{q}}{|x-y|^{n+sq}} d x d y\right)^{\frac{q-1}{q}}\left(\int_{\omega^{\prime} \backslash \mathcal{K}} \frac{|T_\mu((u_\ep-u)(y))|^{q}|\phi_K(x)-\phi_K(y)|^{q}}{|x-y|^{n+sq}} d x d y\right)^{\frac{1}{q}} \\\leq \mu\left(\int_{\omega^{\prime} \backslash \mathcal{K}} \frac{|u_\ep(x)-u_\ep(y)|^{q}}{|x-y|^{n+sq}} d x d y\right)^{\frac{q-1}{q}}\left(\int_{\omega^{\prime} \backslash \mathcal{K}} \frac{|\phi_K(x)-\phi_K(y)|^{q}}{|x-y|^{n+sq}} d x d y\right)^{\frac{1}{q}} .
\end{array}
\end{equation}
If $\beta+\delta>1$, note \cref{aproexist}$_{(c)}$, the condition \cref{thet} on $\theta$ and choose $m=\frac{p+\delta-1}{p-\beta}>1$ in \cref{algebraic2}, to obtain
\begin{equation}{\label{esti8}}
    \frac{|u_\ep(x)-u_\ep(y)|^q}{|x-y|^{n+ sq}} \leq C \frac{\left|u_\ep(x)^{\frac{\delta+p-1}{p-\beta}}-u_\ep(y)^{\frac{\delta+p-1}{p-\beta}}\right|^q}{|x-y|^{n+sq}}\quad \text{ for almost all } x,y \in \omega^\prime.
\end{equation}
Note \cref{boundedinspace} and \cref{esti8}, so that both for $\beta+\delta\leq 1$ and $\beta+\delta>1$, we have from \cref{esti7}, for some $\mathcal{K}=\mathcal{K}(\eta)$
\begin{equation}{\label{esti9}}
    \int_{\mathbb{R}^{2 n} \backslash \mathcal{K}} T_\mu((u_\epsilon-u)(y)) \mathcal{A} u_\epsilon(x, y)(\phi_K(x)-\phi_K(y)) d \nu\leq C\left(\int_{\omega^{\prime} \backslash \mathcal{K}} \frac{|\phi_K(x)-\phi_K(y)|^{q}}{|x-y|^{n+sq}} d x d y\right)^{\frac{1}{q}}\stackrel{\cref{A22}}{\leq }\frac{\eta}{2}.
\end{equation}
For $\ep>0$, let us set $
g_\ep(x, y)=T_\mu((u_\ep-u)(y)) \mathcal{A} u_\ep(x, y)(\phi_K(x)-\phi_K(y)) .$ We show that the functions $g_\ep$ have uniformly absolutely continuous integrals. Let $E \subset \mathcal{K}$ be an arbitrary measurable set. Then, proceeding analogously to \cref{esti7,esti8,esti9}, for some positive constant $C$, independent of $\ep$, we have
\begin{equation}{\label{esti10}}
    \left|\int_E g_\ep(x, y) d \nu\right| \leq C\left(\int_E \frac{|\phi_K(x)-\phi_K(y)|^q}{|x-y|^{n+sq}} d x d y\right)^{\frac{1}{q}} .
\end{equation}
Thus, $\int_E g_\ep(x, y) d \nu \rightarrow 0$ uniformly on $\ep$ if $|E| \rightarrow 0$. Also $g_\ep \rightarrow 0$ pointwise in $\mathbb{R}^{2 n}$. Hence, by Vitali's convergence theorem, it follows that there exists a $\ep_0>0$ such that for all $\ep\leq \ep_0$, it holds
\begin{equation}{\label{esti11}}
    \int_{ \mathcal{K}} T_\mu((u_\epsilon-u)(y)) \mathcal{A} u_\epsilon(x, y)(\phi_K(x)-\phi_K(y)) d \nu\leq \frac{\eta}{2}.
\end{equation}Hence from \cref{esti9,esti11}, it holds
\begin{equation}{\label{esti12}}
\lim _{\ep \rightarrow 0} J_2=\lim _{\ep \rightarrow 0}\int_{\mathbb{R}^n} \int_{\mathbb{R}^n} T_\mu((u_\epsilon-u)(y)) \mathcal{A} u_\epsilon(x, y)(\phi_K(x)-\phi_K(y)) d \nu=0   .
\end{equation}
Recalling \cref{u_0}, we can proceed similarly like $J_2$ and end with
\begin{equation}{\label{esti13}}
\lim _{\ep \rightarrow 0} J_3=\lim _{\ep \rightarrow 0}\int_{\mathbb{R}^n} \int_{\mathbb{R}^n} T_\mu((u_\epsilon-u)(y)) \mathcal{A} u(x, y)(\phi_K(x)-\phi_K(y)) d \nu=0    .
\end{equation}
Now set $v_\ep=\phi_K T_\mu((u_\ep-u))$. Noting \cref{u_0}, by a similar estimate like \cref{esti8}, we have for $\beta+\delta>1$, for every $\eta>0$, there exists a compact set $\mathcal{K}=\mathcal{K}(\eta) \subset \mathbb{R}^{2 n}$ such that
\begin{equation}{\label{esti15}}
\left(\int_{\omega^\prime \backslash \mathcal{K}} \frac{|u(x)-u(y)|^q}{|x-y|^{n+sq}} d x d y\right)^{\frac{1}{q}} \leq C\left(\int_{\omega^\prime\backslash \mathcal{K}} \frac{|u(x)^{\frac{p+\delta-1}{p-\beta}}-u(y)^{\frac{p+\delta-1}{p-\beta}}|^q}{|x-y|^{n+ sq}} d x d y\right)^{\frac{1}{q}} \leq \frac{\eta}{2}.    
\end{equation}
\cref{esti15} clearly holds for $\beta+\delta \leq 1$ as in that case $u\in W^{1,q}_0(\Omega) $. Using \cref{esti15} and the uniform boundedness of $v_\ep$ in $W^{1,q}_0(\Omega)$, we get for both the cases $\beta+\delta\leq 1$ and $\beta+\delta >1$
\begin{equation}{\label{esti16}}
    \int_{\omega^\prime\backslash\mathcal{K}} \mathcal{A} u(x, y)(v_\ep(x)-v_\ep(y)) d \nu \leq \left(\int_{\omega^{\prime} \backslash \mathcal{K}} \frac{|u(x)-u(y)|^q}{|x-y|^{n+sq}} d x d y\right)^{\frac{q-1}{q}}\left(\int_{\omega^{\prime} \backslash \mathcal{K}} \frac{|v_\ep(x)-v_\ep(y)|^q}{|x-y|^{n+ sq}} d x d y\right)^{\frac{1}{q}} \leq \frac{\eta}{2}.
\end{equation}
Finally, proceeding analogously to \cref{esti10}, \cref{esti11} for any given $\eta>0$, there exists $\ep_0$ such that, if $\ep \leq \ep_0$, we get
\begin{equation}{\label{esti17}}
    \int_{\mathcal{K}} \mathcal{A} u(x, y)(v_\ep(x)-v_\ep(y)) d \nu \leq \frac{\eta}{2} .
\end{equation}
Merging \cref{esti16,esti17} one has
\begin{equation}{\label{esti18}}
\lim _{\ep \rightarrow 0} J_4=\lim _{\ep \rightarrow 0}\int_{\mathbb{R}^n} \int_{\mathbb{R}^n} \mathcal{A} u(x, y)(\phi_K(x) T_\mu((u_\epsilon-u)(x))-\phi_K(y) T_\nu((u_\epsilon-u)(y))) d \nu=0.    
\end{equation}
Now, employing the estimates \cref{esti6,esti12,esti13,esti18} in \cref{esti5} we obtain
\begin{equation}{\label{esti19}}
    \lim _{\ep \rightarrow 0} J \geq 0.
\end{equation}
\textbf{Estimate for $R$:} Recalling that $\operatorname{supp} \phi_K=\omega$, and $f\in L^\infty_{\mathrm{loc}}(\Omega)$ by \cref{aproexist}, there exists a constant $C(\omega)>0$, independent of $\ep$, such that $u \geq C(\omega)>0$ in $\omega$. Hence, we have
\begin{equation}{\label{esti20}}
    R=\int_{\Omega} \frac{f_\ep}{(u_\ep+\ep)^\delta} \psi_\ep d x \leq \frac{\|f\|_{L^1(\omega)}}{C(\omega)^\delta} \mu .
\end{equation}
Therefore, for every fixed $\mu>0$, using the estimates \cref{esti4,esti19,esti20} in \cref{estii}, we obtain
\begin{equation}{\label{esti21}}
    \limsup _{\ep \rightarrow 0} \int_{K} (A(x,\nabla u_\ep)-A(x,\nabla u) )\cdot \nabla T_\mu((u_\ep-u)) d x\leq C\mu,
\end{equation}
for some constant $C=C(\omega,\|f\|_{L^\infty(\omega)})>0$. Let us define the function $e_\ep(x)=\langle A(x,\nabla u_\ep)-A(x,\nabla u), \nabla(u_\ep-u)\rangle$
and note by \cref{g1,g2} we have $e_\ep\geq 0$ in $\Omega$. We divide the compact set $K$ by
\begin{equation*}
    E_\ep^\mu=\{x \in K:|(u_\ep-u)(x)| \leq \mu\}, F_\ep^\mu=\{x \in K:|(u_\ep-u)(x)|>\mu\} .
\end{equation*}
Let $\gamma \in(0,1)$ be fixed. Then, from Hölder's inequality,
\begin{equation}{\label{esti22}}
    \int_K e_\ep^\gamma d x=\int_{E_\ep^\mu} e_\ep^\gamma d x+\int_{F_\ep^\mu} e_\ep^\gamma d x \leq\left(\int_{E_\ep^\mu} e_\ep d x\right)^\gamma|E_\ep^\mu|^{1-\gamma}+\left(\int_{F_\ep^\mu} e_\ep d x\right)^\gamma|F_\ep^\mu|^{1-\gamma} .
\end{equation}
Now, since $\{u_\ep\}$ is uniformly bounded in $W^{1, p}(K)$, the sequence $\{e_\ep\}$ is uniformly bounded in $L^1(K)$ (by \cref{growth1}). Furthermore, $\lim _{\ep \rightarrow 0}|F_\ep^\mu|=0$. Hence, from \cref{esti21,esti22}, we have
\begin{equation}{\label{esti23}}
    \lim \sup _{\ep \rightarrow 0} \int_K e_\ep^\gamma d x \leq \lim \sup _{\ep \rightarrow 0}\left(\int_{E_\ep^\mu} e_\ep d x\right)^\gamma|E_\ep^\mu|^{1-\gamma} \leq(C \mu)^\gamma|\Omega|^{1-\gamma} .
\end{equation}
Letting $\mu \rightarrow 0$ in \cref{esti23}, the sequence $\{e_\ep^\gamma\}$ converges to $0$ strongly in $L^1(K)$. Therefore, using a sequence of compact sets $K$, up to a subsequence, one gets $
e_\ep(x) \rightarrow 0 \text { almost everywhere in } \Omega,
$ which along with \cref{g1,g2} gives $
\nabla u_n(x) \rightarrow \nabla u(x) \text { for almost every } x \in \Omega $.
\end{proof}
\begin{theorem}{\label{exis4.2}}
    If $\beta\in [0,p)$, then $u$ given in \cref{u_0} is a weak solution to \cref{problemg} in the sense of \cref{mainweaksol}.
\end{theorem}
\begin{proof}
  First let  $\beta+\delta\leq 1$. Then by \cref{boundedinspace} the sequence $\{u_\ep\}$ is uniformly bounded in $W_0^{1, p}(\Omega)$. Hence by \cref{growth1}, $A(x,\nabla u_\ep)$ is uniformly bounded in $L^{p^\prime}(\Omega)$. Therefore, using \cref{gradconv} and the continuity of $A$ in the second variable, for every $\phi \in C_c^\infty(\Omega)$, we have
\begin{equation}{\label{solu1}}
    \lim _{\ep \rightarrow 0} \int_{\Omega}A(x,\nabla u_\ep) \cdot \nabla \phi \,d x=\int_{\Omega}A(x,\nabla u) \cdot \nabla \phi\, d x .
\end{equation}
Since $\phi \in C_c^\infty(\Omega)$ and $\{u_\ep\}$ is uniformly bounded in $W_0^{1, p}(\Omega)$, by \cref{embedding2}, $\frac{|u_\ep(x)-u_\ep(y)|^{q-2}(u_\ep(x)-u_\ep(y))}{|x-y|^{\frac{n+sq}{q^{\prime}}}} \in L^{q^{\prime}}(\mathbb{R}^n \times \mathbb{R}^n),$ is uniformly bounded and also note
$\frac{\phi(x)-\phi(y)}{|x-y|^{\frac{n+sq}{q}}} \in L^q(\mathbb{R}^n\times \mathbb{R}^n) .
$ Denoting $d\mu=\frac{dxdy}{|x-y|^{n+qs}}$, by weak convergence, we thus have for $\mathcal{A}g(x,y)=|g(x)-g(y)|^{q-2}(g(x)-g(y))$, 
\begin{equation}{\label{solu2}}
    \lim _{\ep \rightarrow 0} \int_{\mathbb{R}^n} \int_{\mathbb{R}^n} \mathcal{A} u_\ep(x, y)(\phi(x)-\phi(y)) d \mu=\int_{\mathbb{R}^n} \int_{\mathbb{R}^n} \mathcal{A} u(x, y)(\phi(x)-\phi(y)) d \mu .
\end{equation}
By \cref{aproexist}, $u_\ep\geq C(\omega)>0$ on supp $\phi=\omega$ for some constant $C(\omega)>0$, independent of $\ep$. Therefore, for every $\phi \in C_c^\infty(\Omega)$, we have
\begin{equation*}
    \left|{f_\ep}{(u_\ep+\ep)^{-\delta}} \phi\right| \leq {\|f\|_{L^{\infty}(\omega)}}{C(\omega)^{-\delta}}|\phi|\quad \text { in } \Omega .
\end{equation*}
By the Lebesgue dominated convergence theorem, we have 
\begin{equation}{\label{solu3}}
    \lim _{\ep \rightarrow 0} \int_{\Omega} {f_\ep}{(u_\ep+\ep)^{-\delta}} \phi\, d x=\int_{\Omega}{f}{u^{-\delta}} \phi\, d x .
\end{equation}
Combining \cref{solu1,solu2,solu3} we conclude for $\beta+\delta \leq 1$. Note that $u$ satisfies the boundary condition in the sense of \cref{behaveu}. For the case $\beta+\delta>1$, by \cref{boundedinspace} the sequence $\{u_\ep^{\frac{p+\delta-1}{p-\beta}}\}$ is uniformly bounded in $W_0^{1, p}(\Omega)$. Moreover, the sequence $\{u_\ep\}$ is uniformly bounded in $W_{\mathrm{loc }}^{1, p}(\Omega)$. Hence $u \in W_{\mathrm{loc}}^{1, p}(\Omega)$. Now following the lines of the proof for the above case, for every $\phi \in C_c^\infty(\Omega)$, we obtain
\begin{equation}{\label{solu4}}
     \lim _{\ep \rightarrow 0} \int_{\Omega}A(x,\nabla u_\ep) \cdot \nabla \phi \,d x=\int_{\Omega}A(x,\nabla u) \cdot \nabla \phi\, d x ,\quad  \lim _{\ep \rightarrow 0} \int_{\Omega} {f_\ep}{(u_\ep+\ep)^{-\delta}} \phi\, d x=\int_{\Omega} {f}{u^{-\delta}} \phi\, d x.
\end{equation}
For the nonlocal part, following the same arguments as in [\citealp{scase}, Theorem 3.6], for every $\phi \in C_c^\infty(\Omega)$, we have
\begin{equation}{\label{solu5}}
    \lim _{\ep\rightarrow 0} \int_{\mathbb{R}^n} \int_{\mathbb{R}^n} \mathcal{A} u_\ep(x, y)(\phi(x)-\phi(y)) d \mu=\int_{\mathbb{R}^n} \int_{\mathbb{R}^n} \mathcal{A} u(x, y)(\phi(x)-\phi(y)) d \mu .
\end{equation}
We get the desired result by merging \cref{solu4,solu5}. Note that $u^\theta\in W^{1,p}_0(\Omega)$ for some $\theta\geq 1$ implies $u$ satisfies the boundary behavior in the sense of \cref{behaveu} (see for instance [\citealp{nachr}, Theorem 2.10]). Also note that we have used [\citealp{MG}, Section 2.5] at every step where required. 
\end{proof}
\subsection{\textbf{Uniqueness: Proof of \cref{uni}}} We prove the result for the general equation as given in \cref{problemg}. \smallskip\\\textbf{Step 1} Let $V=W_0^{1,p}(\Omega)$ with the norm $\|u\|=\|\nabla u\|_{L^p(\Omega)}$ and let $V^\prime$ be its dual. Define for all $m>0$, the truncated function $g_m:\mathbb{R}\to\mathbb{R}^+$ by \begin{equation*}
   g_m(t):= \begin{cases}\min\{t^{-\delta},m\}&\text{ if } t>0 \\ m & \text { otherwise. }\end{cases} \end{equation*}Consider the non-empty closed and convex set
$\mathbb{K}:=\{\phi \in V: 0 \leq \phi \leq v \text { a.e. in } \Omega\}, 
$ where $v$ is the given supersolution to \cref{problem} (or \cref{problemg}). Under the assumptions on $f$ and applying \cref{Sobolev embedding} one can define the operator $J_m: V \rightarrow V^\prime$ for every $u, \psi \in V$ by
\begin{equation*}
    \langle J_m(u), \psi\rangle:=\int_{\Omega} A(x, \nabla u) \cdot \nabla \psi \,d x+\int_{\mathbb{R}^n} \int_{\mathbb{R}^n} \frac{|u(x)-u(y)|^{q-2}(u(x)-u(y))(\psi(x)-\psi(y)) }{|x-y|^{n+sq}}d xdy-\int_{\Omega} f g_m(u) \psi\, d x .
\end{equation*}
First, we observe that $J_m$ is well-defined. For this, similarly like \cref{aproexist}, we use \cref{growth1}, and H\"older inequality for the first two terms and for the third term see by Hardy inequality that 
\begin{equation}{\label{unii}}
   \left| \int_{\Omega} f g_m(u) \psi\, d x \right|\leq c\int_{\Omega} d^{-\beta}|\psi|dx=c\int_{\Omega} d^{1-\beta}\frac{|\psi|}{d}dx\leq C\left(\int_{\Omega} d^{\frac{(1-\beta)p}{p-1}}dx\right)^{\frac{p-1}{p}}\left(\int_{\Omega}|\nabla |\psi||^pdx)\right)^{\frac{1}{p}}<\infty,
\end{equation}
for all $\psi\in V$ provided ${(1-\beta)p}/{(p-1)}>-1$ i.e. $\beta<2-{1}/{p}$. Now for any $u\in V$ and $\theta \in(0,1)$, to be chosen
\begin{equation}{\label{uuni1}}
    \int_{\Omega} d^{-\beta} u\, d x \leq\left(\int_{\Omega}\left(\frac{|u|}{d}\right)^p\right)^{\frac{1-\theta}{p}}\left(\int_{\Omega} |u|^r\right)^{\frac{\theta}{r}}\left(\int_{\Omega} d^{(1-\beta-\theta) l}\right)^{\frac{1}{l}} \leq C\|u\|^{1-\theta}\left(\int_{\Omega} |u|^r\right)^{\frac{\theta}{r}}\left(\int_{\Omega} d^{(1-\beta-\theta) l}\right)^{\frac{1}{l}},
\end{equation}
where $r<p^*$, if $p<n, \frac{1-\theta}{p}+\frac{\theta}{r}+\frac{1}{l}=1$ and in the last inequality, we used Hardy inequality. Here finiteness requires $(1-\beta-\theta) l>-1$, which is equivalent to $\theta<\frac{2 p r-p r \beta-r}{p r-r+p}$. Due to the fact that $\beta<2-\frac{1}{p}$ (note also that $p>1>\theta$) and by \cref{uuni1}, it is easy to deduce similarly like \cref{aproexist} that $J_m$ is coercive. Finally noting 
\begin{equation*}
    \left|f(x)g_m(u_k) \psi(x)-f(x)g(u) \psi(x)\right| \leq 2m|f||\psi|,
\end{equation*}
for every sequence $\{u_k\}$ in $V$ converging to $u$ in the norm of $V$, and using \cref{unii}, one can follow the same arguments as in \cref{aproexist}, to get $J_m$ is demicontinuous. Also, by \cref{g1,g2}, it is easy to get that $J_m$ strictly monotone (also see \cref{aproexist}). As a consequence of \cref{2.13}, there exists a unique $z \in \mathbb{K}$ such that for every $\psi \in z+(W^{1,p}_0(\Omega)\cap L_c^\infty(\Omega))$ with $\psi\in \mathbb{K}$, one has
\begin{equation}{\label{uuni2}}\begin{array}{c}
       \int_{\Omega} A(x, \nabla z) \cdot \nabla(\psi-z) d x +\int_{\mathbb{R}^n} \int_{\mathbb{R}^n} \mathcal{A}z(x,y)((\psi-z)(x)-(\psi-z)(y))d\mu\geq \int_{\Omega} f(x) g_m(z)(\psi-z) d x ,
\end{array}
\end{equation}
where $\mathcal{A}u(x,y)=|u(x)-u(y)|^{q-2}(u(x)-u(y)),\, d\mu=\frac{dxdy}{|x-y|^{n+sq}}$ and $ L_c^\infty(\Omega)$ is the set of all bounded functions with compact support in $\Omega$. Let us now consider a real valued function $g \in C_c^{\infty}(\mathbb{R})$ such that $0 \leq g \leq 1, g \equiv 1$ in $[-1,1]$ and $g \equiv 0$ in $(-\infty,-2] \cup[2, \infty)$. Define the function $\phi_h:=g\left(\frac{z}{h}\right) \phi$ and $\phi_{h, t}:=\min \left\{z+t \phi_h, v\right\}$ with $h \geq 1$ and $t>0$ for a given non-negative $\phi \in C_c^\infty(\Omega)$. Clearly, $\phi_{h, t} \in z+(W^{1,p}_0(\Omega)\cap L_c^\infty(\Omega))$. Then by \cref{uuni2},
\begin{equation}{\label{uuni3}}\begin{array}{c}
       \int_{\Omega} A(x, \nabla z) \cdot \nabla(\phi_{h, t}-z) d x +\int_{\mathbb{R}^n} \int_{\mathbb{R}^n} \mathcal{A}z(x,y)((\phi_{h, t}-z)(x)-(\phi_{h, t}-z)(y))d\mu\geq \int_{\Omega} f(x) g_m(z)(\phi_{h, t}-z) d x ,
\end{array}
\end{equation}
Define,
\begin{equation*}
    \begin{array}{rcl}
         I: & =&c\int_{\Omega}|\nabla(\phi_{h, t}-z)|^2(|\nabla \phi_{h, t}|+|\nabla z|)^{p-2} d x \\
& &+c\int_{\mathbb{R}^n} \int_{\mathbb{R}^n} \frac{(|\phi_{h, t}(x)-\phi_{h, t}(y)|+|z(x)-z(y)|^{q-2}((\phi_{h, t}-z)(x)-(\phi_{h, t}-z)(y))^2}{|x-y|^{n+sq}} d x d y,
    \end{array}
\end{equation*}
where $c$ is given by \cref{g1,g2}. By \cref{uuni3}, we now have
\begin{equation*}
\begin{array}{rcl}
     I &\leq & \int_{\Omega}(A(x, \nabla  \phi_{h, t})-A(x, \nabla z)) \cdot\nabla(\phi_{h, t}-z) d x \smallskip\\
& &+\int_{\mathbb{R}^n} \int_{\mathbb{R}^n}(\mathcal{A} \phi_{h, t}(x, y)-\mathcal{A} z(x, y))((\phi_{h, t}-z)(x)-(\phi_{h, t}-z)(y)) d \mu \smallskip\\
&\leq & \int_{\Omega}A(x, \nabla  \phi_{h, t}) \cdot\nabla(\phi_{h, t}-z) d x+\int_{\mathbb{R}^n} \int_{\mathbb{R}^2} \mathcal{A} \phi_{h, t}(x, y)((\phi_{h, t}-z)(x)-(\phi_{h, t}-z)(y)) d \mu \\
&& -\int_{\Omega} f(x) g_m(z)(\phi_{h, t}-z) d x .
\end{array}
\end{equation*}Therefore,
\begin{equation}{\label{uuni6}}
    \begin{array}{l}
      \quad  I-\int_{\Omega} f(x)(g_m(\phi_{h, t})-g_m(z))(\phi_{h, t}-z) d x\leq \int_{\Omega}(A(x, \nabla  \phi_{h, t}) \cdot\nabla(\phi_{h, t}-z) d x
        \smallskip\\\quad+\int_{\mathbb{R}^n} \int_{\mathbb{R}^n} \mathcal{A} \phi_{h, t}(x, y)((\phi_{h, t}-z)(x)-(\phi_{h, t}-z)(y)) d \mu -\int_{\Omega} f(x) g_m(\phi_{h, t})(\phi_{h, t}-z) d x 
        \smallskip
        \\
        =
        \left(\int_{\Omega} A(x,\nabla \phi_{h, t})\cdot \nabla(\phi_{h, t}-z-t \phi_h)d x
        +\int_{\mathbb{R}^n} \int_{\mathbb{R}^n} h(x, y) d x d y
        -\int_{\Omega} f(x) g_m(\phi_{h, t})(\phi_{h, t}-z-t \phi_h) d x\right) 
        \smallskip\\
\quad+t\left(\int_{\Omega}A (x,\nabla \phi_{h,t} )\cdot \nabla \phi_h\,d x+\int_{\mathbb{R}^n} \int_{\mathbb{R}^n} \mathcal{A} \phi_{h, t}(x, y)(\phi_h(x)-\phi_h(y)) d \mu-\int_{\Omega} f(x)g_m(\phi_{h, t}) \phi_h d x\right)
    \end{array}
\end{equation}
where $h(x, y)=\frac{\mathcal{A} \phi_{h, t}(x, y)((\phi_{h, t}-z-t \phi_h)(x)-(\phi_{h, t}-z-t \phi_h)(y))}{|x-y|^{n+sq}} .
$ \\Set $\Omega=S_v\cup S_v^c$, where $S_v := \{x \in \Omega : \phi_{h,t}(x) = v(x)\}$ and $S_v^
c:= \Omega\backslash S_v$. Observe that $A(x,\nabla \phi_{h, t})\cdot \nabla(\phi_{h, t}-z-t \phi_h)= A(x,\nabla v)\cdot \nabla(\phi_{h, t}-z-t \phi_h) = 0$ on $S_v^
c$ and $A(x,\nabla \phi_{h, t})\cdot \nabla(\phi_{h, t}-z-t \phi_h)= A(x,\nabla v)\cdot \nabla(\phi_{h, t}-z-t \phi_h) $ on $S_v$. Let us denote by $h_v(x, y)=\frac{\mathcal{A} v(x, y)((\phi_{h, t}-z-t \phi_h)(x)-(\phi_{h, t}-z-t \phi_h)(y))}{|x-y|^{n+sq}} .
$ Now following the exact arguments as in the proof of [\citealp{scase}, Lemma 4.1], we obtain
\begin{equation*}
    \int_{\Omega} A(x,\nabla \phi_{h, t})\cdot \nabla(\phi_{h, t}-z-t \phi_h) d x=\int_{\Omega}A(x,\nabla v)\cdot \nabla(\phi_{h, t}-z-t \phi_h)  d x,\quad  \iint_{\mathbb{R}^{2n}}  h(x, y) d x d y=\iint_{\mathbb{R}^{2n}} h_v(x, y) d x d y .
\end{equation*}
Noting the above property along with the fact that $v$ is a weak supersolution of \cref{problemg}, we choose $(z+t \phi_h-\phi_{h, t})\in W^{1,p}_0(\Omega)\cap L_c^\infty(\Omega)$ as a test function and using the definition of $g_m$ and $\phi_{h,t}$, 
we obtain
\begin{equation*}
    \begin{array}{c}
          \int_{\Omega} A(x,\nabla \phi_{h, t})\cdot \nabla(\phi_{h, t}-z-t \phi_h) d x+\int_{\mathbb{R}^{2n}}  h(x, y) d x d y-\int_{\Omega} f(x) g_m(\phi_{h, t})(\phi_{h, t}-z-t \phi_h) d x\smallskip\\=\int_{\Omega}A(x,\nabla v)\cdot \nabla(\phi_{h, t}-z-t \phi_h)  d x+\int_{\mathbb{R}^{2n}} h_v(x, y)-\int_{\Omega} f(x) g_m(\phi_{h, t})(\phi_{h, t}-z-t \phi_h) d x\leq 0.
    \end{array}
\end{equation*}
Using the above fact in \cref{uuni6} with the observation that $I \geq 0$ and $\phi_{h, t}-z \leq t \phi_h$, we have
\begin{equation}{\label{uuni7}}
\begin{array}{l}
    \int_{\Omega}A (x,\nabla \phi_{h,t} )\cdot \nabla \phi_h\,d x+\int_{\mathbb{R}^{n}} \int_{\mathbb{R}^{n}} \mathcal{A} \phi_{h, t}(x, y)(\phi_h(x)-\phi_h(y)) d \mu-\int_{\Omega} f(x)g_m(\phi_{h, t}) \phi_h dx\smallskip\\\geq-\int_{\Omega} f(x)|g_m(\phi_{h, t}) -g_m(z)|\phi_h dx.
    \end{array}
\end{equation}
Note now $z \in W_0^{1, p}(\Omega)$ and recall the definition of $\phi_{h, t}$, and use the Lebesgue dominated convergence theorem, letting $t \rightarrow 0$ in \cref{uuni7}, to obtain
\begin{equation*}
    \int_{\Omega}A(x,\nabla z)\cdot \nabla \phi_h d x+\int_{\mathbb{R}^n} \int_{\mathbb{R}^n} \mathcal{A} z(x, y)(\phi_h(x)-\phi_h(y)) d \mu \geq \int_{\Omega} f(x) g_m(z) \phi_h d x .
\end{equation*}
Finally, passing the limit as $h \rightarrow \infty$ in the above inequality, one gets
\begin{equation}{\label{uni8}}
    \int_{\Omega}A(x,\nabla z)\cdot \nabla \phi\, d x+\int_{\mathbb{R}^n} \int_{\mathbb{R}^n} \mathcal{A} z(x, y)(\phi(x)-\phi(y)) d \mu \geq \int_{\Omega} f(x) g_m(z) \phi \,d x ,
\end{equation}
for every non-negative $\phi\in C_c^\infty(\Omega)$.
\smallskip\\\textbf{Step 2} Let $m>0$ and $\epsilon=2 m^{-\frac{1}{\delta}}$. Since $u \leq 0$ on $\partial \Omega$, for $z \in \mathbb{K}$, we have $(u-z-\epsilon)^{+} \in W_0^{1, p}(\Omega)$. For $\eta>0$, by standard density arguments, it follows from \cref{uni8} that
\begin{equation}{\label{uuni9}}
    \begin{array}{c}
          \int_{\Omega}A(x,\nabla z)\cdot\nabla T_\eta((u-z-\epsilon)^{+})  d x+\int_{\mathbb{R}^n} \int_{\mathbb{R}^n} \mathcal{A} z(x, y)(T_\eta((u-z-\epsilon)^{+})(x)-T_\eta((u-z-\epsilon)^{+})(y)) d \mu \smallskip\\\geq \int_{\Omega} f(x) g_m(z) T_\eta((u-z-\epsilon)^{+}) d x.
    \end{array}
\end{equation}
Since $(u-z-\epsilon)^{+} \in W_0^{1, p}(\Omega)$, $\exists$ a sequence $\phi_k \in C_c^{\infty}(\Omega)$ such that $\phi_k \rightarrow(u-z-\epsilon)^{+}$ strongly in $W_0^{1, p}(\Omega)$. We define  $\psi_{k, \eta}:=T_\eta(\min \{(u-z-\epsilon)^{+}, \phi_k^{+}\}) \in W_0^{1, p}(\Omega) \cap L_c^{\infty}(\Omega) $.
Since $u$ is a weak subsolution of \cref{problemg}, we have
\begin{equation*}
    \int_{\Omega}A(x,\nabla u)\cdot \nabla \psi_{k, \eta} d x+\int_{\mathbb{R}^n} \int_{\mathbb{R}^n} \mathcal{A} u(x, y)(\psi_{k, \eta}(x)-\psi_{k, \eta}(y)) d \mu \leq \int_{\Omega} f(x) u^{-\delta} \psi_{k, \eta} d x ,
\end{equation*}
Consequently, using dominated convergence theorem and Fatou's lemma, we can pass the limit as $k \rightarrow \infty$ to get
\begin{equation}{\label{uuni10}}
    \begin{array}{c}
          \int_{\Omega}A(x,\nabla u)\cdot\nabla T_\eta((u-z-\epsilon)^{+})  d x+\int_{\mathbb{R}^n} \int_{\mathbb{R}^n} \mathcal{A} u(x, y)(T_\eta((u-z-\epsilon)^{+})(x)-T_\eta((u-z-\epsilon)^{+})(y)) d \mu \smallskip\\\leq \int_{\Omega}f(x) u^{-\delta}  T_\eta((u-z-\epsilon)^{+}) d x .
    \end{array}
\end{equation}
Following the proof of [\citealp{scase}, Theorem 4.2] we have
\begin{equation*}
    \mathcal{A} u(x, y)(T_\eta((u-z-\epsilon)^{+}(x))-T_\eta((u-z-\epsilon)^{+}(x))) =\mathcal{A} u(x, y)((u-z)(x)-(u-z)(y)) H(x, y),
\end{equation*}
\begin{equation*}
    \mathcal{A} z(x, y)(T_\eta((u-z-\epsilon)^{+}(x))-T_\eta((u-z-\epsilon)^{+}(x))) =\mathcal{A} z(x, y)((u-z)(x)-(u-z)(y)) H(x, y)
\end{equation*}
with
\begin{equation*}
    H(x, y):=\frac{T_\eta((u-z-\epsilon)^{+}(x))-T_\eta((u-z-\epsilon)^{+}(y))}{(u-z)(x)-(u-z)(y)},
\end{equation*}
where $(u-z)(x)-(u-z)(y) \neq 0$. Subtracting \cref{uuni9} from \cref{uuni10} and then, using the above two displays along with the non-negativity of $H(x,y)$ and \cref{p case} (or \cref{g1,g2}), we have
\begin{equation*}
    \begin{array}{l}
      \quad C \int_{\Omega}|\nabla T_\eta((u-z-\epsilon)^{+})|^2(|\nabla u|+|\nabla z|)^{p-2} d x \smallskip\\  \quad+C \int_{\mathbb{R}^n} \int_{\mathbb{R}^n} \frac{(|u(x)-u(y)|+|z(x)-z(y)|)^{q-2}((u-z)(x)-(u-z)(y))^2}{|x-y|^{n+sq}} H(x, y) d x d y \smallskip\\
      \leq \int_{\Omega}(A(x,\nabla u)-A(x,\nabla z))\cdot \nabla T_\eta((u-z-\epsilon)^{+}) d x\smallskip \\
 \quad+\int_{\mathbb{R}^n} \int_{\mathbb{R}^n}(\mathcal{A} u(x, y)-\mathcal{A} z(x, y))(T_\eta((u-z-\epsilon)^{+}(x))-T_\eta((u-z-\epsilon)^{+}(y))) d \mu 
   \end{array}
\end{equation*}
\begin{equation*}
    \begin{array}{l}\leq \int_{\Omega} f(x)({u^{-\delta}}-g_m(z)) T_\eta((u-z-\epsilon)^{+}) d x \smallskip\\
 \leq \int_{\Omega} f(x)(g_m(u)-g_m(z)) T_\eta((u-z-\epsilon)^{+}) d x \leq 0.
    \end{array}
\end{equation*}
In the final estimate above, we have used $\epsilon>m^{-\frac{1}{\delta}}$, the definition of $g_m$ and the fact that $u \geq \epsilon$ in the support of $(u-z-\epsilon)^{+}$. Therefore, using the nonnegativity of the integrand, letting $\eta \rightarrow \infty$, the above estimate yields
\begin{equation*}
\int_{\Omega}|\nabla((u-z-\epsilon)^{+})|^2(|\nabla u|+|\nabla z|)^{p-2} d x=0 .    
\end{equation*}
As a consequence, $u \leq z+2 m^{-\frac{1}{\delta}} \leq v+2 m^{-\frac{1}{\delta}}$. Letting $m \rightarrow \infty$, we get $u \leq v$ in $\Omega$.
\subsection{\textbf{Existence: Proof of \cref{mainexis}}} In view of \cref{exis4.2} and \cref{u_0}, we deduced the existence of a weak solution to \cref{problem} in the sense of \cref{mainweaksol}. Further for any $\Omega^\prime\subset\subset\Omega$, there exists $C_{\Omega^\prime}>0$ such that $u(x)\geq C_{\Omega^\prime}d(x)>0$ in $\Omega^\prime$. Now note that such a solution of \cref{problem} is indeed a local weak solution in the sense of \cref{11}, which easily follows via density arguments. Using now \cref{holder2} (or \cref{holder1}), it only remains to verify that $u\in \mathcal{C}_{d_{\beta, \delta}}$ where $\mathcal{C}_{d_{\beta, \delta}}$ is defined in \cref{conicalshell}.
\smallskip\\\textbf{Case 1} $\beta+\delta>1$ with $\frac{p-\beta}{p+\delta-1}\in (0,s)$\\Denote $\tau=\frac{p-\beta}{p+\delta-1}$. Clearly $\tau\in(0,1)$. We have assumed $\tau\in (0,s)$. Let $\eta \in(0,1)$ be a constant to be specified later. Define the function ${v}_\ep(x)=\eta \underline{w}_\rho(x)$, where $\rho>0$ is such that
\begin{equation*}
\underline{w}_\rho(x)= \begin{cases}((d_e(x)+\ep^{\frac{1}{\tau}})^+)^{\tau}-\ep & \text { if } x \in \Omega \cup\left(\Omega^c\right)_\rho, \\ -\ep & \text { otherwise, }\end{cases}    
\end{equation*}
where \begin{equation*}
d_e(x)= \begin{cases}\operatorname{dist}(x,\partial\Omega) & \text { if } x \in \Omega; \\ -\operatorname{dist}(x,\partial\Omega) & \text { if }x\in (\Omega^c)_\rho; \\-\rho &\text{ otherwise,}\end{cases}   \end{equation*}
and $(\Omega^c)_\rho=\left\{x \in \Omega^c: \operatorname{dist}(x, \partial \Omega)<\rho\right\}$. Choose $\rho$ small enough so that $\nabla d\in L^\infty(\Omega_\rho)$ and there exists $M>0$ such that $|\Delta d|\leq M$ in $\Omega_\rho$. Therefore, for $\psi \in C_c^{\infty}(\Omega_\rho)$ with $\psi \geq 0$, and applying [\citealp{aroranodea}, Theorem 3.3], we obtain upon relabeling $\rho$ if needed, for $\ep$ sufficiently small
\begin{equation}{\label{proof1}}
    \begin{array}{rcl}
         \int_{\mathbb{R}^{n}}\int_{\mathbb{R}^n} \frac{|v_\ep(x)-v_\ep(y)|^{q-2}(v_\ep(x)-v_\ep(y))(\psi(x)-\psi(y))}{|x-y|^{n+sq}} d x d y & \leq &C \eta^{q-1} \int_{\Omega_\rho}(d(x)+\ep^{\frac{1}{\tau}})^{\tau(q-1)-sq} \psi \,d x\smallskip\\&=&C \eta^{q-1} \int_{\Omega_\rho}(d(x)+\ep^{\frac{1}{\tau}})^{(\tau-s)(q-1)-s} \psi\, d x\smallskip\\&\leq &C \eta^{q-1} \int_{\Omega_\rho}(d(x)+\ep^{\frac{1}{\tau}})^{(\tau-1)(p-1)-1} \psi \,d x,
    \end{array}
\end{equation}
where $C>0$ is independent of $\ep$, while for the local term we have $\nabla v_\ep=\eta\tau(d(x)+\ep^{\frac{1}{\tau}})^{\tau-1} \nabla d \text { in } \Omega_{\rho} .$ As $|\nabla d|=1$,
\begin{equation}{\label{proof2}}
    \begin{array}{rcl}
         -\int_{\Omega_\rho} \Delta_p v_\ep \psi\,d x & =& (\tau\eta)^{p-1} \int_{\Omega_\rho}(d(x)+\ep^{\frac{1}{\tau}})^{(\tau-1)(p-1)} \nabla d \cdot\nabla \psi\,d x \smallskip\\
& =&(\tau\eta)^{p-1} \int_{\Omega_\rho}\left[-\Delta d (d(x)+\ep^{\frac{1}{\tau}})^{(\tau-1)(p-1)} \psi+(p-1)(1-\tau)(d(x)+\ep^{\frac{1}{\tau}})^{(\tau-1)(p-1)-1}\psi\,\right]d x
\smallskip\\
& \leq &(\tau\eta)^{p-1} \int_{\Omega_\rho}\left[M (d(x)+\ep^{\frac{1}{\tau}})^{(\tau-1)(p-1)} \psi+(p-1)(1-\tau)(d(x)+\ep^{\frac{1}{\tau}})^{(\tau-1)(p-1)-1}\psi\,\right]d x\smallskip\\&\leq &(\tau\eta)^{p-1} \int_{\Omega_\rho}(d(x)+\ep^{\frac{1}{\tau}})^{(\tau-1)(p-1)-1}\psi\,dx. \end{array}
\end{equation}
Merging \cref{proof1,proof2} we get
\begin{equation*}
    \int_{\Omega_\rho}(-\Delta_{p} v_\ep+(-\Delta)_q^s v_\ep) \psi\, d x \leq C(\eta^{p-1}+\eta^{q-1}) \int_{\Omega_\rho}(d(x)+\ep^{\frac{1}{\tau}})^{-\delta\tau-\beta}\psi\,dx.
\end{equation*}
Using \cref{condf} and definition of $f_\ep$, we deduce that
\begin{equation}{\label{proof3}}
    -\Delta_{p} v_\ep+(-\Delta)_q^s v_\ep\leq C(\eta^{p-1}+\eta^{q-1}) (d(x)+\ep^{\frac{1}{\tau}})^{-\delta\tau}f_\ep \quad\text{ weakly in } \Omega_\rho.
\end{equation}
Next, we observe that $(v_\ep+\ep)^{-\delta}=(\eta(d(x)+\ep^{\frac{1}{\tau}})^\tau+(1-\eta) \ep)^{-\delta}$ for $ x \in \Omega_{\rho}$ and distinguish the following cases:\\
(1) If $\eta(d(x)+\ep^{\frac{1}{\tau}})^\tau\geq(1-\eta) \ep$ for $x \in \Omega_\rho$, then 
\begin{equation*}
   (v_\ep+\ep)^{-\delta}\geq (2 \eta)^{-\delta}(d(x)+\ep^{\frac{1}{\tau}})^{-\delta\tau}.
\end{equation*}
Thus, from \cref{proof3} for sufficiently small $\eta>0$ (independent of $\ep$), we can infer that:
\begin{equation*}
      -\Delta_{p} v_\ep+(-\Delta)_q^s v_\ep\leq C(\eta^{p-1}+\eta^{q-1}) (d(x)+\ep^{\frac{1}{\tau}})^{-\delta\tau}f_\ep \leq (2 \eta)^{-\delta}(d(x)+\ep^{\frac{1}{\tau}})^{-\delta\tau}f_\ep\leq(v_\ep+\ep)^{-\delta}f_\ep. \end{equation*}
(2) If $\eta(d(x)+\ep^{\frac{1}{\tau}})^\tau\leq(1-\eta) \ep$ for $x\in \Omega_\rho$, we have
\begin{equation*}
   (v_\ep+\ep)^{-\delta}\geq 2^{-\delta}(1-\eta)^{-\delta}\ep^{-\delta}.
\end{equation*}
We choose $\eta>0$ sufficiently small and independent of $\ep$ such that
\begin{equation*}
      -\Delta_{p} v_\ep+(-\Delta)_q^s v_\ep\leq C(\eta^{p-1}+\eta^{q-1}) (d(x)+\ep^{\frac{1}{\tau}})^{-\delta\tau}f_\ep \leq (\eta^{p-1}+\eta^{q-1}) \ep^{-\delta}f_\ep\leq 2^{-\delta}(1-\eta)^{-\delta}\ep^{-\delta}f_\ep\leq(v_\ep+\ep)^{-\delta}f_\ep. \end{equation*}
In both cases, we choose $\eta>0$ sufficiently small and independent of $\ep$, such that
\begin{equation*}
      -\Delta_{p} v_\ep+(-\Delta)_q^s v_\ep\leq(v_\ep+\ep)^{-\delta}f_\ep \quad\text{ weakly in } \Omega_\rho\end{equation*}
As $u_\ep\geq c$ independent of $\ep$ in $\Omega\backslash\Omega_\rho$, we choose $\eta>0$ small enough and independent of $\ep$ such that the following inequality holds:
\begin{equation*}
    v_\ep(x)\leq \eta(\operatorname{diam}(\Omega)+1)^\tau\leq C_\rho\leq u_1(x)\leq u_\ep(x)\quad \text{ in }\Omega\backslash\Omega_\rho.
\end{equation*}
Now $u_\ep\in W^{1,p}_0(\Omega)\cap L^\infty(\mathbb{R}^n)$ for each $\ep>0$. Note by [\citealp{aroranodea}, Theorem 3.3] that $v_\ep\in \overline{W}^{s,q}(\Omega_\rho)$ where \begin{equation*}
    \overline{W}^{s,q}(\Omega)=\left\{u\in L^q_{\mathrm{loc}}(\mathbb{R}^n):\exists K\text{ s.t. } \Omega\subset\subset K, \|u\|_{W^{s,q}(K)}+\int_{\mathbb{R}^n}\frac{|u(x)|^{q-1}}{(1+|x|)^{n+sq}}dx<\infty\right\}.\end{equation*} Therefore by weak comparison principle, see for instance [\citealp{antonini}, Proposition 4.1] (also see the way in \cref{aproexist} where the monotonicity of $\{u_\ep\}$ was shown), $v_\ep(x) \leq u_\ep(x)$ in $\Omega_\rho$ i.e.
\begin{equation}{\label{proof4}}
    \eta((d(x)+\ep^{\frac{1}{\tau}})^{\tau}-\ep) \leq u_\ep\quad\text{ in }\Omega.
\end{equation}
Next, we will prove the upper bound for $u_\epsilon$, for this consider $w_\ep=\Gamma \bar{w}_\rho$, where\begin{equation*}
\bar{w}_\rho(x)= \begin{cases}((d_e(x)+\ep^{\frac{1}{\tau}})^+)^{\tau} & \text { if } x \in \Omega \cup\left(\Omega^c\right)_\rho, \\ 0 & \text { otherwise, }\end{cases}\end{equation*} and $\Gamma\geq 1$ is a constant to be specified later. Proceeding as \cref{proof2}, for $\psi \in C_c^{\infty}(\Omega_\rho)$ with $\psi \geq 0$, we obtain
\begin{equation}{\label{proof5}}
    \begin{array}{rcl}
         -\int_{\Omega_\rho} \Delta_p w_\ep \psi\,d x 
         & =&(\Gamma\tau)^{p-1} \int_{\Omega_\rho}\left[-\Delta d (d(x)+\ep^{\frac{1}{\tau}})^{(\tau-1)(p-1)} \psi+(p-1)(1-\tau)(d(x)+\ep^{\frac{1}{\tau}})^{(\tau-1)(p-1)-1}\psi\,\right]d x
\smallskip\\
& \geq &(\Gamma\tau)^{p-1} \int_{\Omega_\rho}\left[-M (d(x)+\ep^{\frac{1}{\tau}})^{(\tau-1)(p-1)} \psi+(p-1)(1-\tau)(d(x)+\ep^{\frac{1}{\tau}})^{(\tau-1)(p-1)-1}\psi\,\right]d x.
\end{array}
\end{equation}
On a similar note as before, we have used [\citealp{aroranodea}, Theorem 3.3] to get
\begin{equation}{\label{proof6}}
         \int_{\mathbb{R}^{2n}} \frac{|w_\ep(x)-w_\ep(y)|^{q-2}(w_\ep(x)-w_\ep(y))(\psi(x)-\psi(y))}{|x-y|^{n+sq}} d x d y  \geq 
         C \Gamma^{q-1} \int_{\Omega_\rho}(d(x)+\ep^{\frac{1}{\tau}})^{(\tau-s)(q-1)-s} \psi \,d x\geq 0.
\end{equation}
Merging \cref{proof5} \cref{proof6} and if necessary by reducing $\rho$ further, we may assume that there exists $C>0$ such that
\begin{equation*}
    (p-1)(1-\tau)(d(x)+\epsilon^{\frac{1}{\tau}})^{(\tau-1)(p-1)-1}-M(d(x)+\epsilon^{\frac{1}{\tau}})^{(\tau-1)(p-1)} \geq C(d(x)+\epsilon^{\frac{1}{\tau}})^{(\tau-1)(p-1)-1} \text { in } \Omega_{\rho}.
\end{equation*}Therefore,
\begin{equation*}
    \int_{\Omega_{\rho}}(-\Delta_p w_\epsilon \psi+(-\Delta)^s_q w_\epsilon) \psi \,dx\geq C(\Gamma \tau)^{p-1} \\
\int_{\Omega_{\rho}}(d(x)+\epsilon^{\frac{1}{\tau}})^{(\tau-1)(p-1)-1} \psi\,dx.
\end{equation*}
Noting that $w_\ep\in \overline{W}^{s,q}(\Omega_\rho)$, by [\citealp{aroranodea}, Theorem 3.3] and taking into account \cref{condf}, we obtain
\begin{equation}{\label{proof8}}
(-\Delta_p w_\epsilon+(\Delta)^s_q w_\epsilon) \geq C_1(\Gamma \tau)^{p-1}(d(x)+\epsilon^{\frac{1}{\tau}})^{-\delta \tau}{f_\epsilon(x)} \geq C_1(\Gamma \tau)^{p-1}\Gamma^{\delta}w_\ep^{-\delta}{f_\epsilon(x)} \quad \text { weakly in } \Omega_{\rho} .    
\end{equation}
By using the lower estimate of $u_\epsilon$ by $v_\epsilon$, we obtain
\begin{equation*}
    f_\epsilon(x)(u_\epsilon+\epsilon)^{-\delta} \leq f(x)(v_\epsilon+\epsilon)^{-\delta} \leq f(x) \eta^{-\delta} d^{-\tau \delta} .
\end{equation*}
Since $f \in L_{\mathrm{loc}}^{\infty}(\Omega)$, we observe that $f(x) \eta^{-\delta} d^{-\tau \delta} \in L_{\mathrm{loc}}^{\infty}(\Omega)$. Next, we claim that $u_\epsilon$ is bounded, independently of $\epsilon$, in $\Omega \backslash \Omega_{\rho}$. Let $\{B_{\rho / 4}(x_i)\}_{i=1, \ldots, m}$ be a finite covering of $\overline{\Omega \backslash \Omega_{\rho}}$ such that $\overline{\Omega \backslash \Omega_{\rho}} \subset \cup_{i=1}^m B_{\rho / 4}(x_i) \subset \Omega \backslash \Omega_{\rho / 2} \subset\subset \Omega .$ Therefore, by $L^{\infty}$ estimate of \cref{local boundedness}, we get
\begin{equation}{\label{proof7}}
\|u_\epsilon\|_{L^{\infty}(B_{\rho / 4}(x_i))}  \leq C\left(\fint_{B_{\rho / 2}(x_i)} u_{\varepsilon}^p d x\right)^{1 / p}+ T_{q-1}(u_\epsilon ; x_i, \frac{\rho}{4})^{(q-1) /(p-1)}+C+\|f_{ \epsilon}(u_\epsilon+\epsilon)^{-\delta}\|_{L^{\infty}(B_{\rho/ 2})}^{1 /(p-1)},    \end{equation}where $C=C(n, p, q, s)>0$ is a constant. We see that the last term of \cref{proof7} is independent of $\epsilon$. For the other terms, we have from \cref{boundedinspace} that $\{u_\epsilon^\theta\}_{\varepsilon}$ is bounded in $W_0^{1, p}(\Omega)$ for some $\theta \geq 1$. Therefore,
\begin{equation*}
    \fint_{B_{\rho / 2}(x_i)} u_{\varepsilon}^p d x\leq C(1+\|u_\epsilon^\theta\|_{L^p(\Omega)}) \leq C(1+\|u_\epsilon^\theta\|_{W_0^{1, p}(\Omega)}) \leq C .
\end{equation*}
Noting that $u_\epsilon=0$ in $\mathbb{R}^n\backslash \Omega$, we obtain using H\"older inequality
\begin{equation*}
    T_{q-1}(u_\epsilon ; x_i, \frac{\rho}{4})^{q-1}\leq C \rho^{sq} \int_{\Omega \backslash B_{\rho / 4}(x_i)} \frac{|u_\epsilon(x)|^{q-1}}{\rho^{n+s q}} d x \leq C \rho^{-n}\|u_{\varepsilon}^\theta\|_{L^{p}(\Omega)}^{\frac{q-1}{\theta}} \leq C. 
\end{equation*}
Hence, \cref{proof7} implies that the sequence $\{\|u_\epsilon\|_{L^{\infty}(\Omega \backslash \Omega_\rho)}\}_\epsilon$ is uniformly bounded with respect to $\epsilon$, say $\|u_\epsilon\|_{L^{\infty}(\Omega \backslash \Omega_{\rho})} \leq K$. Now, we choose $\Gamma$ sufficiently large independent of $\epsilon$  such that $C_1(\Gamma \tau)^{p-1}\Gamma^{\delta}\geq 1$ in \cref{proof8} and
\begin{equation*}
    w_\epsilon=\Gamma(d+\epsilon^{1 / \tau})^\tau\geq \Gamma d(x)^\tau\geq \rho^\tau\Gamma \geq K \geq u_\epsilon \quad \text { in } \Omega \backslash \Omega_{\rho}.
\end{equation*} Then, by comparison principle in $\Omega_\rho$ and the above inequality, we get \begin{equation}{\label{proof9}}
    u_\epsilon \leq \Gamma(d(x)+\ep^{\frac{1}{\tau}})^{\tau}\quad\text{ in }\Omega.
\end{equation}
Taking $\ep\to 0$ in \cref{proof9,proof4}, we get the desired bound.\smallskip\\
\textbf{Case 2} $\beta+\delta>1$ with $\frac{p-\beta}{p+\delta-1}\in [s,1) $ and $\beta\neq p-q^\prime s(p+\delta-1)$\\
The lower bound is already proved in \cref{aproexist} and see also \cref{u_0}. To show the upper bound, we take $w_\ep=\Gamma\bar{w}_\rho$. Proceeding similarly as \cref{proof5,proof8} and using [\citealp{giacomonipq}, Lemma 3.12] one gets for some constant $C$
\begin{equation*}
    -\Delta_{p} w_\ep+(-\Delta)_q^s w_\ep\geq C(\Gamma\tau)^{p-1}(d(x)+\ep^{\frac{1}{\tau}})^{-\delta\tau-\beta}+\Gamma^{q-1}h \quad\text{ weakly in } \Omega_\rho,
\end{equation*}
for some $h\in L^\infty(\Omega_\rho)$ and $h$ is independent of $\ep$. Note that 
\begin{equation*}{\label{proof10}}
  \Gamma^{q-1}  \|h\|_{L^\infty(\Omega_\rho)}\leq \frac{C}{2}\Gamma^{p-1}(d(x)+\ep^{\frac{1}{\tau}})^{-\delta\tau-\beta},
\end{equation*} for sufficiently small $\ep$ and $\rho$ which may depend on $h$ also. Therefore by \cref{condf} it holds 
\begin{equation}{\label{proof12}}
    -\Delta_{p} w_\ep+(-\Delta)_q^s w_\ep\geq C\Gamma^{p-1} (d(x)+\ep^{\frac{1}{\tau}})^{-\delta\tau-\beta}-\Gamma^{q-1}\|h \|_{L^\infty(\Omega_\rho)}\geq \frac{C}{2}C_1\Gamma^{p-1} (d(x)+\ep^{\frac{1}{\tau}})^{-\delta\tau}f_\ep\quad\text{ in } \Omega_\rho.
\end{equation} Furthermore, in $\Omega_{\rho}$, we have
\begin{equation*}
\Gamma^{p-1} \frac{C}{2}C_1 (d(x)+\ep^{\frac{1}{\tau}})^{-\delta\tau}\geq(\Gamma(d(x)+\epsilon^{\frac{1}{\tau}})^\tau)^{-\delta} \geq(\Gamma(d(x)+\epsilon^{\frac{1}{\tau}})^\tau+\ep)^{-\delta} =(\Gamma \bar{w}_\rho+\epsilon)^{-\delta}=(w_\ep+\ep)^{-\delta},    
\end{equation*}
provided $\Gamma^{p-1+\delta}{C}C_1 \geq 2$. Thus, from \cref{proof12}, we obtain
\begin{equation*}
      -\Delta_{p} w_\ep+(-\Delta)_q^s w_\ep\geq f_\ep(w_\ep+\ep)^{-\delta} \quad \text{ weakly in }\Omega_\rho.
\end{equation*}
On account of \cref{condf} and the lower bound of $u_\ep$ inside $\Omega\backslash\Omega_\rho$, we note that
\begin{equation*}
    f_\ep(u_\ep+\ep)^{-\delta}\leq fu_\ep^{-\delta}\leq C_\rho^{-\delta}d(x)^{-\beta}\leq  C_\rho^{-\delta}\rho^{-\beta}\quad \text{ in }\Omega\backslash\Omega_\rho.
\end{equation*}
The rest follows proceeding similarly to the previous case.  \smallskip\\
\textbf{Case 3} $\beta+\delta>1$ with $\beta= p-q^\prime s(p+\delta-1)$\\ The lower bound follows from \cref{aproexist}. To show the upper bound, similarly to case 2 above, we fix $\Gamma>1$ (to be specified later) and choose $\beta_1 \in(\beta, p )$. Then $\beta_1+\delta>1>0$ and $\beta_1 \neq p -q^{\prime} s(p-1+\delta)$. Thus, using [\citealp{aroranodea}, Theorem 3.3] and [\citealp{giacomonipq}, Lemma 3.12], for $\tau_1=\frac{ p-\beta_1}{p-1+\delta} \in(0, 1)$, we have
\begin{equation*}
-\Delta_p(\Gamma \bar{w}_\rho)+(-\Delta)_q^{s}(\Gamma \bar{w}_\rho)\geq \Gamma^{p-1}\frac{C}{2}(d(x)+\ep^{\frac{1}{\tau_1}})^{-\beta_1-\delta\tau_1}\geq \frac{C}{2}C_1\Gamma^{p-1} (d(x)+\ep^{\frac{1}{\tau_1}})^{-\beta_1+\beta}(d(x)+\ep^{\frac{1}{\tau_1}})^{-\delta\tau_1}f_\ep,
\end{equation*}
weakly in $\Omega_{\rho}$, for sufficiently small $\epsilon$. Furthermore, in $\Omega_{\rho}$, we have
\begin{equation*}
\begin{array}{rcl}
    \frac{C}{2}C_1\Gamma^{p-1} (d(x)+\ep^{\frac{1}{\tau_1}})^{-\beta_1+\beta}(d(x)+\ep^{\frac{1}{\tau_1}})^{-\delta\tau_1}&\geq& \frac{C}{2}C_1\Gamma^{p-1} (\rho+1)^{-\beta_1+\beta}(d(x)+\ep^{\frac{1}{\tau_1}})^{-\delta\tau_1}\\&\geq &(\Gamma(d(x)+\ep^{\frac{1}{\tau_1}})^{\tau_1})^{-\delta}\geq (\Gamma\bar{w}_\rho+\ep)^{-\delta},
    \end{array}
\end{equation*}
provided $\Gamma^{p-1+\delta} {C}C_1(\rho+1)^{-\beta_1+\beta} \geq 2$. Then, proceeding similarly to case 1,2 above, we obtain
$
u_\epsilon \leq \Gamma \bar{w}_\rho$ in $ \Omega$, for sufficiently small $\epsilon$. Passing to the limit as $\epsilon \rightarrow 0$, we get $u(x) \leq \Gamma d(x)^{\frac{p-\beta_1}{p-1+\delta} }$ in $ \Omega$, for all $\beta_1\in(\beta,p)$.\smallskip\\
\textbf{Case 4} $\beta+\delta\leq 1$\\ Note that the lower bound has already been obtained in \cref{aproexist}. To show the upper bound, we fix $\bar{\beta}>0$ and $\sigma>0$ such that $\bar{\beta}=(1-\delta) +\sigma(p-1+\delta)>(1-\delta) \geq \beta
$ We further impose the conditions $\sigma<\frac{p-1+p \delta }{p(p-1+\delta)}$ (this implies $\bar{\beta}<2-\frac{1}{p}$ ) and $\sigma \neq 1-q^{\prime} s$ (this implies $\bar{\beta} \neq p -q^{\prime} s(p-1+\delta)$). Thanks to \cref{condf}, we can choose $m \geq 1$ such that $f\leq m\tilde{f} $ in $\Omega$ where $\tilde{f}\in L^\infty_{\mathrm{loc}}(\Omega)$ satisfies for some constants $C_1,C_2$ the inequality $C_1d(x)^{-\bar{\beta}}\leq \tilde{f}(x)\leq C_2d(x)^{-\bar{\beta}}$ in $\Omega_\rho$. Let $v\in W_{\mathrm {loc}}^{1, p}(\Omega)$ be the solution (as obtained in \cref{exis4.2}) to
\begin{equation*}
-\Delta_p v+(-\Delta)_q^{s} v  =m \tilde{f} v^{-\delta}\text{ in }\Omega, \quad  v>0 \text { in } \Omega,   \quad v=0\text{ in }\mathbb{R}^n\backslash\Omega. 
\end{equation*}
By choice of $m$, we see that $u$ is a sub-solution to the above problem. Since $\bar{\beta}<2-\frac{1}{p}$, by applying the weak comparison principle as in \cref{uni}, and by case 2 above, we get
\begin{equation*}
u(x) \leq v(x) \leq \Gamma d(x)^{\frac{p-\bar{\beta}}{p+\delta-1}}=\Gamma d(x)^{1-\sigma} \quad \text { in } \Omega     .
\end{equation*}
Therefore, for any $\bar{\sigma} \in(0, 1)$, we can choose $\sigma \in(0, 1-\bar{\sigma})$ satisfying all the assumptions above. This completes the proof of the theorem.
\subsection{\textbf{Optimal Sobolev regularity: Proof of \cref{optimalregu}}} Define $\tau=\frac{(p-\beta)}{(p+\delta -1)}$, when $\beta\neq p-q^\prime s(p+\delta-1)$. 
For $\beta+\delta>1$, let $u_\epsilon$ be the weak solution of \cref{approximated} with $\mathcal{L}=-\Delta_p+(-\Delta)^s_q$. Then, using the boundary behavior of the approximating sequence $u_\epsilon$ as obtained in \cref{mainexis} and taking $\phi=u_\epsilon$ as a test function in the weak formulation of \cref{approximated}, we obtain
\begin{equation*}
\|u_\epsilon\|_{W_0^{1, p}(\Omega)}^p\leq\int_{\Omega} f_\ep u_\epsilon^{1-\delta} d x \leq C_1 \int_{\Omega} d^{(1-\delta) \tau-\beta}(x) d x \leq C  ,  
\end{equation*}
if $(1-\delta)( p-\beta)>(\beta-1)(p+\delta-1) \Leftrightarrow p(\delta-1)+\beta p<(p+\delta-1) \Leftrightarrow \Lambda<1$.
Similarly, by taking $\phi=u_\epsilon^{p(\theta-1)+1}$ as a test function, we obtain for $\theta>\Lambda>1$
\begin{equation*}
\|u^\theta_\epsilon\|_{W_0^{1, p}(\Omega)}^p\leq C\int_{\Omega} f_\ep u_\epsilon^{(\theta-1)(p-1)+\theta-\delta} d x\leq C \int_{\Omega} d^{(\theta p-(p-1+\delta)) \tau-\beta}(x) d x \leq C.
\end{equation*}
Now, by passing limits $\epsilon \rightarrow 0$, we get the solution $u \in W_0^{1, p}(\Omega)$ if $\Lambda<1$ and $u^\theta \in W_0^{1, p}(\Omega)$ if $\theta>\Lambda \geq 1$. 
\\The only if statement for the case $p<\beta+s(p+\delta-1)$ follows from the Hardy inequality and the boundary behavior of the weak solution. Precisely, if $\Lambda \geq 1$, then $u \notin W_0^{1, p}(\Omega)$. Indeed, we have
\begin{equation*}
   \|u_\epsilon\|_{W_0^{1, p}(\Omega)}^p \geq C \int_{\Omega}\left|\frac{u(x)}{d(x)}\right|^p d x \geq C \int_{\Omega} d^{p(\tau-1)}(x) d x=\infty .
\end{equation*} In the same way, if $\theta \in[1, \Lambda]$, then
\begin{equation*}
   \|u^\theta_\epsilon\|_{W_0^{1, p}(\Omega)}^p \geq C \int_{\Omega}\left|\frac{u^\theta(x)}{d(x)}\right|^p d x \geq C \int_{\Omega} d^{p(\theta\tau-1)}(x) d x=\infty , 
\end{equation*}
and we deduce $u^\theta \notin W_0^{1, p}(\Omega)$.
\subsection{\textbf{H\"older regularity: Proof of \cref{holder6}}}
For $\beta+\delta\leq 1$, we have $u\in W^{1,p}_0(\Omega)$. For every $\phi \in C_c^\infty(\Omega)$, $u$ satisfies
\begin{equation*}
    \int_\Omega |\nabla u|^{p-2}\nabla u\cdot\nabla \phi\,dx+\int_{\mathbb{R}^n}\int_{\mathbb{R}^n}\frac{|u(x)-u(y)|^{q-2}(u(x)-u(y))(\phi(x)-\phi(y))}{|x-y|^{n+sq}}dxdy=\int_\Omega fu^{-\delta}\phi\,dx.
\end{equation*}
Now for every $\psi \in W^{1,p}_0(\Omega)$, there exists a sequence of function $0 \leq \psi_k \in C_c^\infty(\Omega) \rightarrow|\psi|$ strongly in $W^{1,p}_0(\Omega)$ as $k \rightarrow \infty$ and pointwise almost everywhere in $\Omega$. We observe that
\begin{equation}{\label{hul}}
    \begin{array}{l}
         \left|\int_{\Omega} fu^{-\delta}\psi\, d x\right| \leq  \int_{\Omega} fu^{-\delta}|\psi| d x\leq \liminf _{k \rightarrow \infty} \int_{\Omega} fu^{-\delta}\psi_k dx\smallskip\\\qquad=  \liminf _{k \rightarrow \infty}\left(\int_\Omega |\nabla u|^{p-2}\nabla u\cdot\nabla \psi_kdx+\int_{\mathbb{R}^{2n}}\frac{|u(x)-u(y)|^{q-2}(u(x)-u(y))(\psi_k(x)-\psi_k(y))}{|x-y|^{n+sq}}dxdy\right)\smallskip\\\qquad\leq  C(\|u\|^{p-1}+\|u\|^{q-1}) \lim _{k \rightarrow \infty}\|\psi_k\| \leq C(\|u\|^{p-1}+\|u\|^{q-1}) \||\psi|\| \leq C(\|u\|^{p-1}+\|u\|^{q-1}) \|\psi\|.
         \end{array}
\end{equation}
for some positive constant $C$. Let $\phi \in W^{1,p}_0(\Omega) $, then there exists a sequence $\{\phi_k\} \subset C_c^\infty(\Omega)$ which converges to $\phi$ strongly in $W^{1,p}_0(\Omega)$. We claim that

\begin{equation*}
\lim _{k \rightarrow \infty}    \int_\Omega fu^{-\delta}\phi_kdx=\int_\Omega fu^{-\delta}\phi\,dx.
\end{equation*}
Indeed, using $\psi=\phi_k-\phi$ in \cref{hul}, we obtain
\begin{equation}{\label{hul1}}
    \lim _{k \rightarrow \infty}\left|\int_{\Omega} fu^{-\delta}(\phi_k-\phi) d x\right| \leq C(\|u\|^{p-1}+\|u\|^{q-1}) \lim _{k\rightarrow \infty}\|\phi_k-\phi\|=0 .
\end{equation}
Again, since $\phi_k \rightarrow \phi$ strongly in $W^{1,p}_0(\Omega)$ as $k \rightarrow \infty$, we have
\begin{equation}{\label{hul2}}
\begin{array}{l}
  \lim _{k \rightarrow \infty}\left(   \int_\Omega |\nabla u|^{p-2}\nabla u\cdot\nabla \phi_kdx+\int_{\mathbb{R}^n}\int_{\mathbb{R}^n}\frac{|u(x)-u(y)|^{q-2}(u(x)-u(y))(\phi_k(x)-\phi_k(y))}{|x-y|^{n+sq}}dxdy\right) \\
   = \int_\Omega |\nabla u|^{p-2}\nabla u\cdot\nabla \phi\,dx+\int_{\mathbb{R}^n}\int_{\mathbb{R}^n}\frac{|u(x)-u(y)|^{q-2}(u(x)-u(y))(\phi(x)-\phi(y))}{|x-y|^{n+sq}}dxdy.\end{array}
\end{equation}
Merging \cref{hul1,hul2} we get 
\begin{equation*}
    \int_\Omega |\nabla u|^{p-2}\nabla u\cdot\nabla \phi\,dx+\int_{\mathbb{R}^n}\int_{\mathbb{R}^n}\frac{|u(x)-u(y)|^{q-2}(u(x)-u(y))(\phi(x)-\phi(y))}{|x-y|^{n+sq}}dxdy=\int_\Omega fu^{-\delta}\phi\,dx, 
\end{equation*}
for every $\phi \in W^{1,p}_0(\Omega)$. One can now apply \cref{holder4} and the boundary behavior of $u$ obtained as in \cref{mainexis} to get the result for $\beta+\delta<1$.\smallskip\\Note now that a solution of \cref{problem} defined in the sense of \cref{mainweaksol} and as obtained in \cref{exis4.2} is indeed a local weak solution in the sense of \cref{11}, which can be deduced by a standard density argument. Also by \cref{holder2} (or \cref{holder1}), $u\in C^{0,\eta}_{\mathrm{loc}}(\Omega)$ for all $\eta\in (0,1)$. Further by \cref{holder3}, $u\in C^{1,\gamma}_{\mathrm{loc}}(\Omega)$ for some $\gamma\in (0,1)$. Also, by the boundary behavior of $u$ as obtained in \cref{mainexis}, the corresponding boundary H\"older regularity follows for the case (ii),(iii),(iv). For the sake of completeness, we include a few details.
In this regard, let $x, y \in \Omega_\rho$ and suppose without loss of generality $d(x) \geq d(y)$. Now two cases occur:\\
(I) either $|x-y| < {d(x)}/{64}$, in which case set $64 R=d(x)$ and $y \in B_{R}(x)$. Hence, we apply the local H\"older regularity in $B_{R}(x)$ and we obtain the regularity.\\
(II) or $|x-y| \geq {d(x)}/{64} \geq {d(y)}/{64}$ in which case we claim that the exponent of global H\"older regularity matches with the power $\zeta$ such that $u(x)\leq Cd(x)^\zeta$. Indeed for a constant $C>0$ large enough, we get
\begin{equation*}
\frac{|u(x)-u(y)|}{|x-y|^{\zeta}} \leq \frac{|u(x)|}{|x-y|^{\zeta}}+\frac{|u(y)|}{|x-y|^{\zeta}} \leq C_1\left(\frac{u(x)}{d^{\zeta}(x)}+\frac{u(y)}{d^{\zeta}(y)}\right) \leq C,    
\end{equation*}
Thus, we get our claim, and the proof is complete.
\subsection{\textbf{Non-existence: Proof of \cref{nonexist}}}
On the contrary assume $\exists$ a solution $v \in W_{l o c}^{1, p}(\Omega)$ of \cref{problem} and $\gamma_0 \geq 1$ such that $v^{\gamma_0} \in W_0^{1, p}(\Omega)$. For $\tilde{\beta}<p$, which will be specified later, we choose $f_{\tilde{\beta}} \in L_{l o c}^{\infty}(\Omega)$ such that $m f_{\tilde{\beta}}(x) \leq f(x)$ a.e. in $\Omega$, for some constant $m \in(0,1)$ independent of $\tilde{\beta}$, and for some positive constants $c_3, c_4$,
\begin{equation*}
    c_3 d(x)^{-\tilde{\beta}} \leq m f_{\tilde{\beta}}(x) \leq c_4 d(x)^{-\tilde{\beta}} \text { in } \Omega_{\rho} .
\end{equation*}We now construct a suitable subsolution near the boundary $\partial \Omega$, to arrive at some contradiction. For $\epsilon>0$, let $w_\epsilon \in W_0^{1, p}(\Omega)$ be the unique solution to the following problem (see \cref{aproexist})
\begin{equation}{\label{nonexi1}}
    -\Delta_p w_\epsilon+(-\Delta)^s_q w_\epsilon=m f_{\tilde{\beta}, \epsilon}(x)\left(w_\epsilon+\epsilon\right)^{-\delta},
\end{equation}
where $f_{\tilde{\beta}, \epsilon}(x):=\left(f_{\tilde{\beta}}(x)^{\frac{-1}{\tilde{\beta}}}+\epsilon^{\frac{p-1+\delta}{p-\tilde{\beta}}}\right)^{-\tilde{\beta}}$ if $f_{\tilde{\beta}}(x)>0$ and 0 otherwise.
Next, we will prove $w_\epsilon \leq u_0$ in $\Omega$. Note that $w_\epsilon \in C^{1, \alpha}(\overline{\Omega})$ (by \cref{holder4}) and $w_\epsilon=0$ on $\mathbb{R}^n\backslash \Omega$. Hence, for given $\sigma>0$, $\exists\,\varrho>0$ such that $w_\epsilon \leq \sigma / 2$ in $\Omega_\varrho$. Moreover, $w_\epsilon-v-\sigma \leq-\sigma / 2<0$ in $\Omega_\varrho$, because $v\geq 0$, and we have $\operatorname{supp}(w_\epsilon-v-\sigma)^{+} \subset \Omega \backslash \Omega_\varrho \subset\subset \Omega .
$ Therefore, $(w_\epsilon-v-\sigma)^{+} \in W_0^{1, p}(\Omega)$ and from the weak formulation of \cref{nonexi1}, we obtain
\begin{equation}{\label{nonexi2}}
\begin{array}{l}
    \int_{\Omega}|\nabla w_\epsilon|^{p-2} \nabla w_\epsilon \cdot\nabla T_k((w_\epsilon-v-\sigma)^{+})dx+\iint_{\mathbb{R}^{2n}}\mathcal{A}w_\ep(x,y)(T_k((w_\epsilon-v-\sigma)^{+})(x)-T_k((w_\epsilon-v-\sigma)^{+})(y))d\mu\\=\int_{\Omega} \frac{m f_{\tilde{\beta}, \epsilon}(x)}{(w_\epsilon+\epsilon)^\delta} T_k((w_\epsilon-v-\sigma)^{+}),
    \end{array}
\end{equation}
where $\mathcal{A}g(x,y)=|g(x)-g(y)|^{q-2}(g(x)-g(y)), \, d\mu=\frac{dxdy}{|x-y|^{n+qs}},\, T_k(t):=\min \{t, k\}$ for $k>0$ and $t \geq 0$. Furthermore, since $v\in W_{\mathrm{loc }}^{1, p}(\Omega)$ is a weak solution to \cref{problem}, for all $\psi \in C_c^{\infty}(\Omega)$, we have

\begin{equation}{\label{nonexi3}}
\int_{\Omega}|\nabla v|^{p-2} \nabla v\cdot \nabla \psi+\iint_{\mathbb{R}^{2n}}\mathcal{A}v(x,y)(\psi(x)-\psi(y))d\mu=\int_{\Omega} f(x)v^{-\delta} \psi .    
\end{equation}
Let $\{\psi_j\} \in C_c^{\infty}(\Omega)$ be such that $\psi_j \rightarrow(w_\epsilon-v-\sigma)^{+}$ in $W_0^{1, p}(\Omega)$. Set $\tilde{\psi}_{j, k}:=T_k(\min\{(w_\epsilon-v-\sigma)^{+}, \psi_j^{+}\})$. Then, $\tilde{\psi}_{j, k} \in W_0^{1, p}(\Omega) \cap L_c^{\infty}(\Omega)$, therefore by a standard density argument, from \cref{nonexi3}, we infer that
\begin{equation*}
\int_{\Omega}|\nabla v|^{p-2} \nabla v\cdot \nabla \tilde{\psi}_{j, k}+\iint_{\mathbb{R}^{2n}}\mathcal{A}v(x,y)(\tilde{\psi}_{j, k}(x)-\tilde{\psi}_{j, k}(y))d\mu=\int_{\Omega} f(x)v^{-\delta} \tilde{\psi}_{j, k} .    
\end{equation*}
Using the fact that $\operatorname{supp}(w_\epsilon-v-\sigma)^{+} \subset\subset \Omega$ and Fatou lemma, we obtain
\begin{equation}{\label{nonexi4}}
    \begin{array}{l}
         \int_{\Omega}|\nabla v|^{p-2} \nabla v\cdot \nabla  T_k((w_\epsilon-v-\sigma)^{+}+\iint_{\mathbb{R}^{2n}}\mathcal{A}v(x,y)( T_k((w_\epsilon-v-\sigma)^{+})(x)- T_k((w_\epsilon-v-\sigma)^{+})(y))d\mu\smallskip\\\geq\int_{\Omega} f(x) v^{-\delta} T_k((w_\epsilon-v-\sigma)^{+}) \geq\int_{\Omega} m f_{\tilde{\beta}, \epsilon}(x) v^{-\delta}  T_k((w_\epsilon-v-\sigma)^{+}).
    \end{array}
\end{equation}
Deducting \cref{nonexi4} from \cref{nonexi2}, we infer that
\begin{equation}{\label{nonexi5}}
    \begin{array}{l}
    \quad\int_{\Omega} (|\nabla w_\epsilon|^{p-2} \nabla w_\epsilon-|\nabla v|^{p-2} \nabla v) \cdot\nabla T_k((w_\epsilon-v-\sigma)^{+}) d x \smallskip\\\quad+\iint_{\mathbb{R}^{2n}}(\mathcal{A}w_\ep(x,y)-\mathcal{A}v(x,y))( T_k((w_\epsilon-v-\sigma)^{+})(x)- T_k((w_\epsilon-v-\sigma)^{+})(y))d\mu\smallskip\\\leq \int_{\Omega} m f_{\tilde{\beta}, \epsilon}(x)((w_\epsilon+\epsilon)^{-\delta}-v^{-\delta}) T_k((w_\epsilon-v-\sigma)^{+}) d x 
\leq  \int_{\Omega} m f_{\tilde{\beta}, \epsilon}(x)(w_\epsilon^{-\delta}-v^{-\delta}) T_k((w_\epsilon-v-\sigma)^{+}) d x .
    \end{array}
\end{equation}
Following the proof of [\citealp{scase}, Theorem 4.2] we have
\begin{equation*}
  \mathcal{A} w_\ep(x, y)(T_k((w_\epsilon-v-\sigma)^{+})(x)- T_k((w_\epsilon-v-\sigma)^{+})(y)) =\mathcal{A} w_\ep(x, y)((w_\ep-v)(x)-(w_\ep-v)(y)) H(x, y)  
\end{equation*}
and
\begin{equation*}
  \mathcal{A} v(x, y)(T_k((w_\epsilon-v-\sigma)^{+})(x)- T_k((w_\epsilon-v-\sigma)^{+})(y)) =\mathcal{A} v(x, y)((w_\ep-v)(x)-(w_\ep-v)(y)) H(x, y)  
\end{equation*}
with
\begin{equation*}
    H(x, y):=\frac{T_k((w_\epsilon-v-\sigma)^{+})(x)- T_k((w_\epsilon-v-\sigma)^{+})(y)}{(w_\ep-v)(x)-(w_\ep-v)(y)}
\end{equation*}
where $(w_\ep-v)(x)-(w_\ep-v)(y) \neq 0$. Using the above two identities and \cref{p case} and the non-negativity of $H(x,y)$ we deduce from \cref{nonexi5}
\begin{equation*}
    \begin{array}{l}
\quad C \int_{\Omega}(|\nabla w_\epsilon|+|\nabla v|)^{p-2}|\nabla T_k((w_\epsilon-v-\sigma)^{+})|^2 \smallskip\\\quad+C \int_{\mathbb{R}^n} \int_{\mathbb{R}^n} \frac{(|w_\ep(x)-w_\ep(y)|+|v(x)-v(y)|)^{q-2}((w_\ep-v)(x)-(w_\ep-v)(y))^2}{|x-y|^{n+sq}} H(x, y) d x d y \smallskip\\\leq \int_{\Omega} (|\nabla w_\epsilon|^{p-2} \nabla w_\epsilon-|\nabla v|^{p-2} \nabla v) \cdot\nabla T_k((w_\epsilon-v-\sigma)^{+}) d x \smallskip\\\quad+\iint_{\mathbb{R}^{2n}}(\mathcal{A}w_\ep(x,y)-\mathcal{A}v(x,y))( T_k((w_\epsilon-v-\sigma)^{+})(x)- T_k((w_\epsilon-v-\sigma)^{+})(y))d\mu\smallskip\\\leq\int_{\Omega} m f_{\tilde{\beta}, \epsilon}(x)(w_\epsilon^{-\delta}-v^{-\delta}) T_k((w_\epsilon-v-\sigma)^{+}) d x \leq 0,
    \end{array}
\end{equation*}
this implies that $T_k((w_\epsilon-v-\sigma)^{+})=0$ a.e. in $\Omega$ and since it is true for every $k>0$, we get $w_\epsilon \leq v+\sigma$ in $\Omega$. Moreover, the arbitrariness of $\sigma$ proves $w_\epsilon \leq v$ in $\Omega$. Owing to the estimates of $w_\epsilon$ given by \cref{mainexis}, we have
$
\eta\left(\left(d(x)+\epsilon^{\frac{p-1+\delta}{p-\tilde{\beta}}}\right)^{\frac{p-\tilde{\beta}}{p-1+\delta}}-\epsilon\right) \leq w_\epsilon(x) \leq v(x) \text { in } \Omega ,
$ for each $\tilde{\beta}>p-s(\delta+p-1)$. Since $v^{\gamma_0} \in W_0^{1, p}(\Omega)$, by Hardy inequality, we obtain
\begin{equation*}
    \eta^{\gamma_0 p} \int_{\Omega} \frac{\left(\left(d(x)+\epsilon^{\frac{p-1+\delta}{p-\tilde{\beta}}}\right)^{\frac{p-\tilde{\beta}}{p-1+\delta}}-\epsilon\right) ^{\gamma_0 p}}{d^p} \leq C \int_{\Omega} \frac{v^{\gamma_0 p}}{d^p}<\infty,
\end{equation*}
choosing $p-s(\delta+p-1)<\tilde{\beta}<p$ sufficiently close to $p$ and taking $\epsilon \downarrow 0$, we obtain the quantity on the left side is not finite, which yields a contraction. This completes the proof of \cref{nonexist}.
\section{Application to a perturbed problem: Proof of \cref{finalcor}}{\label{appli}}
\textbf{Step 1:} We follow the argument used in \cref{prtb}. Let $h$ and $\phi$ as defined in \cref{hhh,phii}. Let $v\in W^{1,p}_0(\Omega)$ be any weak solution to 
\begin{equation*}
        -\Delta_p v+(-\Delta)^s_q v+K(x)((v-1)^+)^{p-1}=F(x) +{\lambda}{(v^+)^{-\delta}}\text{ in }\Omega; \quad v=0\text{ in }\mathbb{R}^n\backslash\Omega,
    \end{equation*} 
    where $K,F\in L^{n/p}(\Omega)$. We take $\phi((v-1)^+) \in W_0^{1, p}(\Omega)$ as a test function to get
\begin{equation}{\label{coroo1ar}}
\begin{array}{l}
       \int_\Omega (h^{\prime}((v-1)^+))^p|\nabla v|^{p-2}\nabla v\cdot \nabla (v-1)^+ dx\\+\int_{\mathbb{R}^{n}}\int_{\mathbb{R}^{n}}\frac{|v(x)-v(y)|^{q-2}(v(x)-v(y))(\phi((v-1)^+(x))-\phi((v-1)^+(y)))}{|x-y|^{n+sq}}dxdy\smallskip\\+\int_\Omega K \phi ((v-1)^+)((v-1)^+)^{p-1}dx =\int_\Omega F \phi((v-1)^+)dx +\int_\Omega \frac{\lambda}{(v^+)^{\delta}}\phi ((v-1)^+)dx.
\end{array}
\end{equation}
We show that the nonlocal integral 
(denote it by $J$) is non-negative. For this, the following cases can occur:
\\(i) If $v(x),v(y)\leq 1$, then $\phi((v-1)^+(x))=\phi((v-1)^+(y)=0$.\\(ii) If $v(x)>1\geq v(y)$, then by the monotonicity of the map $t\mapsto |t|^{q-2}t$, $t\in \mathbb{R}$, and non-negativity of $\phi$, we get 
\begin{equation*}
\begin{array}{c}
    |v(x)-v(y)|^{q-2}(v(x)-v(y))(\phi((v-1)^+(x))-\phi((v-1)^+(y)))=|v(x)-v(y)|^{q-2}(v(x)-v(y))\phi((v-1)^+(x))\smallskip\\\geq |v(x)-1|^{q-2}(v(x)-1)\phi((v-1)^+(x))\smallskip\\= |(v-1)^+(x)-(v-1)^+(y)|^{q-2}((v-1)^+(x)-(v-1)^+(y))(\phi((v-1)^+(x))-\phi((v-1)^+(y)))\geq 0.
\end{array}
\end{equation*}
Note that we have used [\citealp{parini}, Lemma A.2] in the last line. Due to symmetry, the same holds for $v(y)>1\geq v(x)$.
\\(iii) Finally if $v(x), v(y)>1$, then again by [\citealp{parini}, Lemma A.2], we have 
\begin{equation*}
\begin{array}{l}
    |v(x)-v(y)|^{q-2}(v(x)-v(y))(\phi((v-1)^+(x))-\phi((v-1)^+(y)))\smallskip\\=|(v-1)^+(x)-(v-1)^+(y)|^{q-2}((v-1)^+(x)-(v-1)^+(y))(\phi((v-1)^+(x))-\phi((v-1)^+(y)))\geq 0.
\end{array}
\end{equation*}
Combining above three cases and noting that $(v^+)^{-\delta} \phi((v-1)^+)\neq0$ only if $v\geq 1$, \cref{coroo1ar} yields 
\begin{equation*}{\label{hol7000}}
\begin{array}{l}
       \int_\Omega(h^{\prime}((v-1)^+))^p|\nabla (v-1)^+|^p dx+\int_\Omega K \phi ((v-1)^+)((v-1)^+)^{p-1}dx\leq\int_\Omega (F+\lambda) \phi((v-1)^+)dx  .
\end{array}
\end{equation*}
Observe that $F+\lambda$ is also in $ L^{n/p}(\Omega)$, and thus one can follow the same procedure as the proof of \cref{prtb}, to get $(v-1)^+\in L^t(\Omega)$ for all $t\in [1,\infty)$.\smallskip\\\textbf{Step 2} Now let $ u\in W^{1,p}_0(\Omega)$ be a solution to \cref{problem2}. 
Set $
    K(x)= b(x, u) /(1+((u-1)^+)^{p-1}) .$ Then $K$ satisfies
\begin{equation*}
    |K(x)| \leq C_b\frac{1+|u|^{p^*-1}}{1+((u-1)^+)^{p-1}}=C_b\frac{1+|u-1+1|^{p^*-1}}{1+((u-1)^+)^{p-1}}\leq C_1+C_2 \frac{|u-1|^{p^*-1}}{1+((u-1)^+)^{p-1}}\leq C_1+C_2(1+((u-1)^+)^{p^*-p}),
\end{equation*}
$C_1, C_2$ being constants. As $p^*-p=p^2 /(n-p)$ and $(u-1)^+\in L^{p^*}(\Omega)$ we deduce that $K \in L^{n/ p}(\Omega)$. Then as $u>0$ in $\Omega$, \cref{problem2} can be written in the form\begin{equation*}
    -\Delta_p u+(-\Delta )^s_q u=K(x)((u-1)^+)^{p-1}+ K(x) +\lambda (u^+)^{-\delta}.
\end{equation*}
By step 1, we deduce that $(u-1)^+ \in 
L^t(\Omega)$ for each $t\in[1,\infty)$. Hence $K(x)((u-1)^+)^{p-1}+ K(x) \in 
L^t(\Omega)$ for each $t\in[1,\infty)$, i.e. $K(x)((u-1)^+)^{p-1}+ K(x) \in L^r(\Omega)$ where $r>n/p$ if $p\leq n$ and $r=1$ if $p>n$. For simplicity we write the equation satisfied by $u$ as 
\begin{equation*}
    -\Delta_p u+(-\Delta )^s_q u=\lambda (u^+)^{-\delta}+G(x),
\end{equation*}
where $G(x)\in L^r(\Omega)$, $r$ is described above. Note that for any $k\geq 1$, $\int_\Omega \lambda (u^+)^{-\delta}(u-k)^+dx\leq \int_\Omega \lambda (u-k)^+dx$, and then one can follow the steps in [\citealp{MG}, Proposition 2.1], to conclude $u\in L^\infty(\Omega)$.\smallskip\\\textbf{Step 3:} As $u \in L^{\infty}(\Omega)$,
\begin{equation*}
    -\Delta_p u+(-\Delta)_q^{s} u=\lambda u^{-\delta}+b(x, u) \leq C_b\left(\lambda+\|u\|_{L^{\infty}(\Omega)}^\delta+\|u\|_{L^{\infty}(\Omega)}^{p^*-1+\delta}\right) u^{-\delta}:=\lambda_* u^{-\delta} .
\end{equation*}
Let $v \in W_0^{1, p}(\Omega)$ be the solution of 
\begin{equation*}
\begin{array}{c}
-\Delta_pv+(-\Delta)_q^s v={\lambda_*}{v^{-\delta}}\text { in } \Omega, \quad v>0 \text{ in } \Omega,\quad v =0  \text { in }\mathbb{R}^n \backslash \Omega;
\end{array}
\end{equation*} as obtained by \cref{mainexis} (put $\beta=0$ in \cref{condf}). Then by weak comparison principle (note that by the same argument, as given in the first case of proof of \cref{holder6}, it can be deduced that test functions can be taken from $W_0^{1,p}(\Omega)$), and the boundary behavior of $v$, we obtain
\begin{equation*}
  \eta d(x)\leq  u(x) \leq v(x) \leq \Gamma d(x)^{1-\sigma} \quad \text { in } \Omega,\quad \text{ for all } \sigma \in(0, 1),
\end{equation*}
for some $\eta,\Gamma>0$, depending on $\lambda, p, \delta $ and $\|u\|_{L^{\infty}(\Omega)}$. We obtained the lower bound on $u$ as $u$ is a super-solution to 
\begin{equation*}\begin{array}{c}
-\Delta_pu+(-\Delta)_q^s u={\lambda}{(u+1)^{-\delta}}\text { in } \Omega, \quad u>0 \text{ in } \Omega,\quad u =0  \text { in }\mathbb{R}^n \backslash \Omega.
\end{array}
\end{equation*} 
Finally we have $ \lambda u^{-\delta}+b(x,u)\leq cd(x)^{-\delta}+C_b(1+\|u\|_{L^\infty(\Omega)}^{p^*-1})\leq C_1 d(x)^{-\delta}$, for some constant $C_1$. Therefore using \cref{holder4}, we get the $C^{1, \gamma}$ regularity of $u$ in $\overline{\Omega}$.
\section*{Data availability} There are no associated data concerning this manuscript.
\section*{Conflict of interest} We declare no Conflict of interest.
\bibliography{main.bib}

\bibliographystyle{plain}


\end{document}